\documentclass[aap,preprint]{imsart}

\RequirePackage{amsthm,amsmath,amsfonts,amssymb}
\RequirePackage[authoryear]{natbib} 
\RequirePackage{graphicx}

\startlocaldefs
\theoremstyle{plain}
\newtheorem{Sat}{Satz}[section]
\newtheorem{Thm}[Sat]{Theorem}
\newtheorem{Lem}[Sat]{Lemma}
\newtheorem{Pro}[Sat]{Proposition}
\newtheorem{Cor}[Sat]{Corollary}

\newtheorem{DefThm}[Sat]{Definition and Theorem}

\newtheorem{Def}[Sat]{Definition}
\newtheorem{Ass}[Sat]{Assumption}
\theoremstyle{remark}
\newtheorem{Ex}[Sat]{Example}
\newtheorem{Rem}[Sat]{Remark}
\usepackage{stmaryrd} 
\usepackage{enumitem} 
\usepackage{bbold}
\usepackage{mathtools}
\mathtoolsset{showonlyrefs} 
\newcommand{\BBA}{\citingnopage{BB18_2}{Proposition~2.21}} 
\newcommand{\BBB}{\citingnopage{BB18_2}{Theorem~2.14}} 
\newcommand{\BBC}{\citingnopage{BB18_2}{Eqn.~(5)}} 
\newcommand{\BBD}{\citingnopage{BB18_2}{Proposition~3.2~(i)}} 

\newcommand{\BBF}{\citingnopage{BB18_2}{Proposition~3.7}} 
\newcommand{\BBG}{\citingnopage{BB18_2}{Proposition~3.2~(ii)}} 
\newcommand{\BBH}{\citingnopage{BB18_2}{Definition~2.30}} 
\newcommand{\BBI}{\citingnopage{BB18_2}{Proposition~2.28~(ii)}} 
\newcommand{\BBJ}{\citingnopage{BB18_2}{Proposition~2.31}} 
\newcommand{\BBK}{\citingnopage{BB18_2}{Theorem~2.39}} 
\newcommand{\BBL}{\citingnopage{BB18_2}{Proposition~2.26~(ii)}} 
\newcommand{\BBM}{\citingnopage{BB18_2}{Lemma 3.4~(ii)}}

\newcommand{\EW}{\mathbb{E}}
\newcommand{\WM}{\mathbb{P}}
\newcommand{\NZ}{\mathbb{N}}
\newcommand{\RZ}{\mathbb{R}}
\newcommand{\md}{\mathrm{d}}
\newcommand{\mdS}{{^\ast\!\mathrm{d}}}
\newcommand{\mdST}{{_{\ast}\mathrm{d}}}
\newcommand{\FA}{\mathcal{F}}
\newcommand{\aFA}{\mathcal{F}^\Lambda}
\newcommand{\aFAe}{\mathcal{F}^{\Lambda^\eta}}
\newcommand{\midG}{\,\middle|\,}
\newcommand{\uP}{c_0}
\newcommand{\ldc}[1]{#1^{\text{c,l}}}
\newcommand{\rdc}[1]{#1^{\text{r}}}
\newcommand{\dc}[1]{#1^{\text{c}}}
\newcommand{\ld}[1]{#1^{\text{l}}}
\newcommand{\st}{\mathcal{S}}
\newcommand{\stm}{\mathcal{S}^\Lambda}
\newcommand{\stp}{\mathcal{S}^{\mathcal{P}}}
\newcommand{\stme}{\mathcal{S}^{\Lambda^\eta}}
\newcommand{\stmd}{{{\mathcal{S}^{\Lambda,\mathrm{div}}}}}	
\newcommand{\Vn}{\widetilde{V}}
\newcommand{\Cn}{\widetilde{\mathcal{C}}(c_0)}
\newcommand{\Vm}{V}
\newcommand{\Cm}{\mathcal{C}(c_0)}
\newcommand{\cem}{\dagger}
\newcommand{\essinf}{\mathop{\mathrm{essinf}}}
\newcommand{\esssup}{\mathop{\mathrm{esssup}}}
\newcommand{\argmax}{\mathop{\mathrm{argmax}}}

\newcommand{\stsetRO}[2]{\left\llbracket #1,#2 \right\llbracket}
\newcommand{\stsetC}[2]{\left\llbracket #1,#2 \right\rrbracket}
\newcommand{\stsetO}[2]{\left\rrbracket #1,#2 \right\llbracket}
\newcommand{\stsetG}[1]{\left\llbracket #1 \right\rrbracket}

\newcommand{\lsl}[1]{{^\ast #1}}
\newcommand{\lsr}[2]{{#1^\ast_{#2}}}

\newcommand{\citing}[3]{\cite{#1}, #2, p.#3}
\newcommand{\citingnopage}[2]{\cite{#1}, #2}

\endlocaldefs

\begin{document}

\begin{frontmatter}
\title{Modelling information flows by Meyer-$\sigma$-fields  in the singular stochastic control problem of irreversible investment}
\runtitle{Information Flows in Singular Control }

\begin{aug}
\author[A]{\fnms{Peter} \snm{Bank}\ead[label=e1]{bank@math.tu-berlin.de}}
\and
\author[A]{\fnms{David} \snm{Besslich}\ead[label=e2,mark]{besslich@math.tu-berlin.de}}
\address[A]{Technische Universit{\"a}t
    Berlin, Institut f{\"u}r Mathematik, Stra{\ss}e des 17. Juni 136,
    10623 Berlin, Germany, \printead{e1,e2}}

\end{aug}

\begin{abstract}
	In stochastic control problems delicate issues arise when the controlled system can jump due to both exogenous shocks and endogenous controls. Here one has to specify what the controller knows when about the exogenous shocks and how and when she can act on this information. We propose to use Meyer-$\sigma$-fields as a flexible tool to model information flow in such situations. The possibilities of this approach are illustrated first in a very simple linear stochastic control problem and then in a fairly general formulation for the singular stochastic control problem of irreversible investment with inventory risk. For the latter, we illustrate in a first case study how different signals on exogenous jumps lead to different optimal controls, interpolating between the predictable and the optional case in a systematic manner.
\end{abstract}



\end{frontmatter}


\section{Introduction}

In stochastic control problems one seeks to influence a given system in an optimal way while taking into account the dynamically revealed information on this system. It is clear that this information flow is crucial for the determination of which controls can be used at all and for what an optimal control looks like. Of particular importance are moments in time where significant new information becomes  available, for instance, on an impending exogenous jump. If the controller is restricted to predictable controls, she can only react after the jump has hit the system. In the case of optional controls she can react to jumps as they happen. Apart from these classical choices, it is perfectly conceivable though that the controller at times receives a signal on the upcoming jump that she can use for a proactive intervention and then still react after the jump is fully revealed. 

We show how one can use Meyer-$\sigma$-fields $\Lambda$ (introduced in \cite{EL80}) embedded between the optional and predictable $\sigma$-field  to model information in such situations. As a toy example, we consider a simple linear control problem, which we also use to introduce the basic tools from the \emph{th\'eorie g\'en\'erale des processus} which pertain to Meyer-$\sigma$-fields (e.g.\ \cite{BS77}, \cite{DM78}, \cite{EL80}, \cite{EK81}, \cite{DM82}). 
For a more serious control problem, we discuss in depth an irreversible investment problem with inventory risk. Irreversible investment problems have been considered in great detail in the literature before (e.g.\ \cite{arr66}, \cite{DP_94}, \cite{Ber98}, \cite{MZ_07}, \cite{RX11}, \cite{FP_14}, \cite{Fer15}, \cite{AZ_17}, \cite{DFF_17}). This kind of problem can be formulated
as  the task to
\begin{align}\label{Main:129}
	\text{Maximize}\qquad \Vn(\tilde{C}):=\EW\left[\int_{[0,\infty)} 
	P_t \md \tilde{C}_{t}
	-\int_{[0,\infty)}\rho_t(\tilde{C}_t)  \md R_t
	\right],
\end{align}
over $\Lambda$-measurable, increasing, right-continuous controls $\tilde{C}$ starting at a given level $c_0$. 
Here, the integral over $\rho$ yields a convex risk assessment for a given control with the integrator $R$ serving as a risk clock; the process $P$ describes the rewards accrued from an increase in the control $\tilde{C}$. It is conceivable and, in fact perfectly natural for the examples in this paper,  that the reward process is not fully observable to the controller when she makes her decisions. Instead, she will work with the reward process's $\Lambda$-projection ${^\Lambda P}$ in~\eqref{Main:129}, which is observable to her. To ensure existence of optimal controls we impose suitable mathematical semi-continuity assumptions on the reward process, but we also need to relax the set of controls to include also increasing controls $C$ which are not right-continuous. We introduce a suitable integral $\int_{[0,\infty)}{^\Lambda P}\ \mdS C$ for such relaxed controls with possible double jumps (see Definition~\ref{Main:139}) and show that the ensuing relaxed optimization problem
 \begin{align}\label{Main:136}
    \text{Maximize}\qquad V(C):=\EW\left[\int_{[0,\infty)} 
	{^\Lambda P}_t \mdS C_{t}
	-\int_{[0,\infty)}\rho_t\left(C_t\right)  \md R_t
	\right]
\end{align}
over suitable increasing $\Lambda$-measurable $C$
has the same value as the original problem~\eqref{Main:129}. An optimal control for~\eqref{Main:136} is constructed in terms of the solution of a closely related stochastic representation problem, following an approach first studied in \cite{BR01}. The general setting here requires a considerably refined argument though, which is based on a similarly refined version of the representation theorem from \cite{BK04} which we study in the companion paper~\cite{BB18}. 
 
 To illustrate this general result by a nontrivial explicit example, we focus on the special case $\rho_t(c):=\frac12c^2$ and let $P$ be a discounted compound Poisson process with initial value $\tilde{p}$, i.e.~$P_t: =\mathrm{e}^{-rt}\tilde{P}_t:=\mathrm{e}^{-rt}(\tilde{p}+\sum_{i=1}^{N_t}Y_i)$ with i.i.d.\ $Y_i \in \mathrm{L}^2(\WM)$, $i\in \NZ$, independent of the Poisson process $N$, that also drives our risk clock $R_t:=\int_{(0,t]}\mathrm{e}^{-rs}\md N_s$ with discount rate $r>0$. Apart from the classical choices of predictable and optional controls, we consider $\Lambda^\eta$-measurable controls where the Meyer-$\sigma$-field $\Lambda^\eta$ is defined by 
\begin{align}
	\Lambda^\eta := 
	{^\WM \sigma\left(Z \text{ c\`adl\`ag and $\tilde{\FA}^\eta$-measurable}\right)}
\end{align}
for a fixed sensor sensitivity $\eta \in [0,\infty]$, where $\tilde{\FA}^\eta$ is generated by
$\tilde{P}^\eta:=\tilde{P}_-+\Delta\tilde{P}\mathbb{1}_{\{|\Delta \tilde{P}|\geq \eta\}}$ and $\WM$ symbolizes that we consider the $\WM$-completion of the $\sigma$-field at hand (see Definition and Theorem \ref{Main:107} below).  
A controller with information flow $\Lambda^\eta$ may receive a warning about an impending jump, namely when the jump's absolute value is at least $\eta$.
The case $p(\eta):=\WM(|Y_1|\leq \eta)=1$ corresponds to the predictable-$\sigma$-field while $p(\eta)=0$ leads to the optional-$\sigma$-field. For any sensitivity $\eta$, we derive a closed-form solution to the stochastic representation problem associated with this example and thus obtain an explicit optimal control for the optimization problem. In general, this optimal control turns out to be neither left-continuous nor right-continuous; instead, it is merely l\`adl\`ag. Hence, optimal controls may exhibit ``double jumps'' which correspond to the controller's ability to proactively intervene to reduce the risk before the risk clock ``rings'' and to adjust her position afterwards in order to benefit from higher rewards available then. As one intuitively would expect, we indeed find a variety of optimal controls as we vary the sensitivity $\eta$ of the considered jump sensor, allowing us to assess, for instance, how much in extra value for the controller can be generated by a further improvement in sensor sensitivity. 
	
The article is structured in the following way. In Section~\ref{sec:1} we consider a toy example illustrating the idea of using Meyer-$\sigma$-fields in optimal control problems and we introduce along the way the basic notions from the theory of Meyer-$\sigma$-fields. In  Section~\ref{sec:1.2} we formulate a general irreversible investment problem with inventory risk and we show how to reduce it to a suitable representation problem. In Section~\ref{sec:Poisson} we give an explicit example for the solution of this problem, where the reward process is given by a compound Poisson process. In Appendix~\ref{app:integration} we collect some results concerning the special $\mdS C$-integral we have introduced for l\`adl\`ag controls $C$.

\section{A general optimal control framework with Meyer-$\sigma$-fields}\label{sec:1}
		
In this section we motivate and develop a continuous-time
framework for the flow of information in 
optimal control problems by using Meyer-$\sigma$-fields. This framework is first illustrated using a toy example and will be fully exploited in Section~\ref{sec:1.2}, where we formulate a singular stochastic control problem, namely an irreversible investment problem with inventory risk.

 Uncertainty is described by a filtered probability space 
$(\Omega,\mathbb{F}, 
\FA:=(\FA_t)_{t\geq 0},\WM)$ with $\mathbb{F}:=\bigvee_{t} \FA_t$
and $(\FA_t)_{t\geq 0}$ satisfying
the usual conditions of right-continuity and completeness. 
The filtration $\FA$ can be thought of as the
information flow from observing the exogenous noise driving the controlled system.
The immediacy with which this information can be acted upon by the
controller is clearly crucial for the optimization problem to be 
studied, particularly in a 
setting with jumps. 

To illustrate this, let us give a toy example and consider a compound Poisson process
\[
	\tilde{P}_t=\tilde{p}+\sum_{k=1}^{N_t}Y_k,\quad t\geq 0,
\]
where $\tilde{p}\in \RZ$ and where the i.i.d.\ uniformly distributed jumps $Y_k\sim U[-1,1]$, $k=1,2,\dots$ are independent from the
Poisson process	$N$ with intensity $\lambda>0$.
Let $\FA$ be the augmented filtration generated by
$\tilde{P}$. Let us study how to maximize
\[
	\EW\left[\int_{[0,1]} C_s\md \tilde{P}_s\right]
\]
over controls $C$ with $|C_s|\leq 1$, $0\leq s\leq 1$.
When restricted to $\FA$-predictable 
controls $C$,  $\int_{(0,\cdot]} C_s\md \tilde{P}_s$
is obviously an $\mathbb{F}$-martingale and so
\[
	\EW\left[\int_{[0,1]} C_s\md \tilde{P}_s\right]=0
\]
for \emph{any} such control. By contrast, when controls are
allowed to be optional, we can estimate 
\[
	\EW\left[\int_{[0,1]} C_s\md \tilde{P}_s\right]
	\leq \EW\left[\sum_{0 \leq s\leq 1} |C_s||\Delta \tilde{P}_s|\right]	
	\leq \EW\left[\sum_{0 \leq s\leq 1} |\Delta \tilde{P}_s|\right]
		=\EW\left[N_1\right]\EW[|Y_1|]=\frac{\lambda}{2}
\]
with equality holding true in all the above estimates for the
(then optimal) choice
\[
	\hat{C}_s^\mathcal{O}:=\mathrm{sgn}(\Delta \tilde{P}_s),
	\quad s\in [0,1],
\]
with $\mathrm{sgn}(0)=0$. 
Of course, it is conceivable that, rather than being able to
directly account for all jumps as in the 
optional case, the controller can, for lack of a perfect 
jump sensor, only react immediately
to large enough jumps, say those of absolute value at least
$\eta>0$. This would suggest to consider 
\begin{align}\label{Main:44}
	\hat{C}_s^{\eta}:=\mathrm{sgn}(\Delta \tilde{P}_s)
	\mathbb{1}_{\{|\Delta \tilde{P}_s|\geq \eta\}},\quad s\in [0,1],
\end{align}
as the optimal choice -- but among which controls exactly? 

This question can be answered in a precise way by 
considering Meyer-$\sigma$-fields $\Lambda$ (see 
Definition~\ref{Main:45} below)
 satisfying $\mathcal{P}(\mathcal{F})\subset 	\Lambda\subset
\mathcal{O}(\mathcal{F})$, where $\mathcal{P}(\FA)$,
$\mathcal{O}(\FA)$ denote, respectively, the predictable and the
optional $\sigma$-field associated with $\FA$. The theory of Meyer-$\sigma$-fields was initiated in \cite{EL80}.
We review and extend some of this material in the companion paper \cite{BB18_2}. 
Let us recall here the basic concepts and results.

\begin{Def}[Meyer-$\sigma$-field, \citing{EL80}{Definition 2}{502}]
\label{Main:45}
	A $\sigma$-field $\Lambda$ on $\Omega\times [0,\infty)$
	is called a \emph{Meyer-$\sigma$-field},
	if the following conditions hold:
	\begin{enumerate}[label=(\roman*)]
		\item It is generated by some right-continuous, left-limited
		(RCLL or c\`adl\`ag for short) processes.
		\item It contains 
		$\{\emptyset,\Omega\}\times \mathcal{B}([0,\infty))$,
		where $\mathcal{B}([0,\infty))$
		denotes the Borel-$\sigma$-field on $[0,\infty)$.
		\item It is stable with respect to stopping at deterministic
		time points, i.e.~for a $\Lambda$-measurable process $Z$ and for any
		$s\in [0,\infty)$, also the stopped process 
		$(\omega,t)\mapsto Z_{t\wedge s}(\omega)$ is 
		$\Lambda$-measurable.
	\end{enumerate}
\end{Def}			

Like for filtrations, also for Meyer-$\sigma$-fields there is a notion of completeness with respect to a probability measure $\WM$:
\begin{DefThm}[$\WM$-complete Meyer-$\sigma$-field, see \cite{EL80}, p.507-508]\label{Main:107}
    A Meyer-$\sigma$-field $\Lambda\subset \mathbb{F}\otimes \mathcal{B}([0,\infty))$ is called \emph{$\WM$-complete} if any process $\tilde{Z}$ which is indistinguishable
    from a $\Lambda$-measurable process $Z$ is itself already $\Lambda$-measurable. For any Meyer-$\sigma$-field $\tilde{\Lambda}\subset \mathbb{F}\otimes \mathcal{B}([0,\infty))$
    there exists a smallest $\WM$-complete Meyer-$\sigma$-field $\Lambda$ containing $\tilde{\Lambda}$; it is called the
    \emph{$\WM$-completion} of $\tilde{\Lambda}$.\qed
\end{DefThm}

\begin{Ex}[\citing{EL80}{Example}{509}]
    Let $\tilde{\FA}:=(\tilde{\FA}_t)_{t\geq 0}$ be a filtration on a probability space $(\Omega,\mathbb{F},\WM)$
    and denote by $\FA$ the smallest filtration containing $\tilde{\FA}$ that satisfies the usual conditions.
    Then the $\WM$-completion of the $\tilde{\FA}$-predictable $\sigma$-field is the $\FA$-predictable $\sigma$-field. The $\WM$-completion of the $\tilde{\FA}$-optional $\sigma$-field is contained in the $\FA$-optional $\sigma$-field; if $\Tilde{\FA}$ is right-continuous, then the $\WM$-completion of the $\tilde{\FA}$-optional $\sigma$-field is equal to the $\FA$-optional $\sigma$-field.
\end{Ex}

In our example the $\WM$-complete Meyer-$\sigma$-field encapsulating the jump information
in a convenient manner is given by 
\begin{align}\label{Main:27}
	\Lambda^\eta := 
	{^\WM \sigma\left(Z \text{ c\`adl\`ag and $\tilde{\FA}^\eta$-measurable}\right)}
\end{align}
for a fixed sensor sensitivity $\eta \in [0,\infty]$, where $\tilde{\FA}^\eta$ is generated by
$\tilde{P}^\eta:=\tilde{P}_-+\Delta\tilde{P}\mathbb{1}_{\{|\Delta \tilde{P}|\geq \eta\}}$
 and $\WM$ symbolizes that we consider the $\WM$-completion of the $\sigma$-field at hand.

That $\Lambda^\eta$ is indeed a Meyer-$\sigma$-field can be checked by
the following result:

\begin{Thm}[\citing{EL80}{Theorem 5}{509}]
	A $\sigma$-field on  
	$\Omega\times [0,\infty)$ generated by
	c\`adl\`ag processes is a $\WM$-complete Meyer-$\sigma$-field if and only if
	 it lies between the predictable
	and the optional $\sigma$-field of a filtration satisfying the usual conditions.\qed
\end{Thm}	

\begin{Rem}[Meyer-$\sigma$-fields vs. Filtrations]
The main advantages of a Meyer-$\sigma$-field $\Lambda$ compared to a filtration are technical but powerful tools
like the Meyer Section Theorem below, which for example gives us uniqueness up to indistinguishability of
two $\Lambda$-measurable processes once they coincide at every $\Lambda$-stopping time (cf. Definition~\ref{Main:137} below for this notion). As one can see
for example in \citing{DM78}{Remark 91 (b)}{144}, adapted processes or even progressively measurable processes cannot in general be pinned down up to indistinguishability in this way.
\end{Rem}

Now we can make precise and corroborate our 
above optimality conjecture by showing that $\hat{C}^\eta$ from~\eqref{Main:44} satisfies 
\begin{align}\label{Main:75}
\hat{C}^\eta\in \underset{|C|\leq 1, 
								\text{ $C$ $\Lambda^\eta$-measurable}}{\mathrm{argmax}}
								\EW\left[\int_{[0,1]} C_s \md P_s\right].
\end{align}

To verify this we will need a generalization of classical stopping times to Meyer- or $\Lambda$-stopping times:
\begin{Def}[Following \citing{EL80}{Definition 1}{502}]\label{Main:137}
A mapping $S$ from $\Omega$ to $[0,\infty]$
is a \emph{$\Lambda$-stopping time},
if 
\[
	[[S,\infty[[\,:=\left\{(\omega,t)\in \Omega\times 
	[0,\infty)\, \middle | \, S(\omega)\leq t\right\}\in \Lambda.
\]			
The set of all $\Lambda$-stopping times is denoted
by $\stm$.
Additionally, we define for each random mapping 
$S:\Omega\rightarrow [0,\infty]$ the
$\sigma$-field 
\[
\mathcal{F}^\Lambda_S := \sigma(Z_S \,|\, Z \text{ is a $\Lambda$-measurable process}).
\]	
\end{Def}
Next, we introduce some classification of random times with respect to a Meyer-$\sigma$-field:
\begin{DefThm}[compare \citing{EL80}{Definition and Theorem 6}{510}]\label{Main:138}
	A random time 
	$T:\Omega\rightarrow [0,\infty]$ is called \emph{$\Lambda$-accessible} for a Meyer-$\sigma$-field $\Lambda$ if there exists
	a sequence of $\Lambda$-stopping times $(T_n)_{n\in \mathbb{N}}$ such that 
	\[
		\mathbb{P}\left(\bigcup_{n\in \NZ} \left\{T_n=T<\infty\right\}\right)=\mathbb{P}(T<\infty).
	\]
	The random variable $T$ is called \emph{totally $\Lambda$-inaccessible} if $\mathbb{P}(S=T<\infty)=0$ for all $\Lambda$-stopping times $S$.
	
	Moreover, for $\Lambda\subset \mathcal{O}(\FA)$, where $\FA$ denotes some filtration satisfying the usual conditions, there exists for each $\FA$-stopping time a partition of $\{T<\infty\}$, unique up to $\WM$-null-sets, into two sets $A,I \in \FA_T$, such that $T_A$ is $\Lambda$-accessible and $T_I$ is $\Lambda$-totally inaccessible. Here, e.g., $T_A$ denotes the restriction of $T$ to $A$ given by $T_A=T$ on $A$ and $T_A=\infty$ on $A^c$. \hfill$\qed$
\end{DefThm}
Having introduced totally $\Lambda$-inaccessible stopping times we can now state the following result, which will be crucial for proving that $\hat{C}^\eta$ from~\eqref{Main:44} is optimal.
\begin{Pro}[compare \BBA]\label{Main:165}
    Let $\Lambda$ be $\mathbb{P}$-complete and denote by $\FA:=(\FA_t)_{t\geq 0}$ a filtration satisfying the usual conditions such that $\mathcal{P}(\FA)\subset \Lambda\subset \mathcal{O}(\FA)$. Then we have $\WM$-almost surely for any $\Lambda$-measurable bounded process $C$
     at any totally $\Lambda$-inaccessible $\mathcal{F}$-stopping time $T$, 
    \[
    C_T=(^\mathcal{P} C)_T.
    \] \hfill$\qed$
\end{Pro}			
	
To prove our optimality claim~\eqref{Main:75}, 				
take an arbitrary $\Lambda^\eta$-measurable control $C$ and put $C_\infty:= 0$ for notational convenience. Define furthermore
\[
	T_k:=\inf\{t\in [0,1]\, | \,N_t=k\}
\]
(with $\inf \emptyset := \infty$) and observe that this $\FA$-stopping time has the decomposition 
$T_k=(T_k)_{A_k^\eta}\wedge (T_k)_{(A_k^\eta)^c}$ with $A_k^\eta:=\{|Y_k|\geq \eta\}$ 
into a $\Lambda^\eta$-accessible $\FA$-stopping time $(T_k)_{A_k^\eta}$ and a totally $\Lambda^\eta$-inaccessible $\FA$-stopping time $(T_k)_{(A_k^\eta)^c}$
(see Definition~\ref{Main:138}).
By Proposition~\ref{Main:165} it then follows that $C_{T_k}={^\mathcal{P} C_{T_k}}$
 on $(A_k^\eta)^c$ almost surely. Therefore,
 \begin{align}
 	\EW\left[C_{T_k} Y_k\right]
 	=\EW\left[C_{(T_k)_{\{|Y_k|\geq \eta\}}} Y_k\right]+
 	\EW\left[{^\mathcal{P} C}_{(T_k)_{\{|Y_k|< \eta\}}} Y_k\right].
 \end{align}
 We thus obtain for all $\Lambda^\eta$-measurable $C$ with $0\leq |C_s| \leq 1$,
\begin{align}
	\EW\left[\int_{[0,1]} C_s\md \tilde{P}_s\right]
	&=\sum_{k=1}^\infty\EW\left[C_{T_k} Y_k\right]
	=\sum_{k=1}^\infty \left(\EW\left[C_{(T_k)_{\{|Y_k|\geq \eta\}}} Y_k\right]+
 	\EW\left[{^\mathcal{P} C}_{(T_k)_{\{|Y_k|< \eta\}}} Y_k\right]\right)\\
	&= \EW\left[\sum_{0 \leq s\leq 1} C_s\Delta
	\tilde{P}_s\mathbb{1}_{\{|\Delta \tilde{P}_s|\geq \eta\}}\right]+
	\EW\left[\sum_{0 \leq s\leq 1} {^\mathcal{P} C}_s\Delta
	\tilde{P}_s\mathbb{1}_{\{|\Delta \tilde{P}_s|< \eta\}}\right]	\\
	&= \EW\left[\sum_{0 \leq s\leq 1} C_s\Delta
	\tilde{P}_s\mathbb{1}_{\{|\Delta \tilde{P}_s|\geq \eta\}}\right]	+ 0\label{Main:783}\\
	&\leq \EW\left[\sum_{0 \leq s\leq 1}|\Delta \tilde{P}_s|\mathbb{1}_{\{|\Delta \tilde{P}_s|\geq \eta\}}\right]\label{Main:784}
	=\EW[N_1]\EW[|Y_1|\mathbb{1}_{\{|Y_1|\geq \eta\}}]=\frac{\lambda}{2}(1-\eta^2),
\end{align}
 where we have used in~\eqref{Main:783} that the process 
$\bar{P}_t:=p+ \sum_{k=1}^{N_t} Y_k\mathbb{1}_{\{|Y_k|< \eta\}}$ defines an $\FA$-martingale
and hence
\[
	\EW\left[\sum_{0 \leq s\leq 1} {^\mathcal{P} C}_s\Delta
	\tilde{P}_s\mathbb{1}_{\{|\Delta \tilde{P}_s|< \eta\}}\right]=\EW\left[\int_{[0,1]}{^\mathcal{P}C}_s\md \bar{P}_s\right]=0.
\]
To finally conclude~\eqref{Main:75}, we now just have to observe that we have equality in~\eqref{Main:784} for $C=\hat{C}^\eta$ from~\eqref{Main:44}.			

Before using Meyer-$\sigma$-fields in a more relevant setting than the above toy problem, let us conclude this section with another two important concepts from the general theory of stochastic processes that are Meyer-measurable: the Meyer Section Theorem and the Meyer Projection Theorem. 

\begin{Thm}[Meyer Section Theorem, 
			\citing{EL80}{Section Theorem 1}{506}]\label{Main:84}
			Let $B$ be an element of a $\WM$-complete Meyer-$\sigma$-field $\Lambda$. For every
			$\epsilon>0$, there exists $S\in \stm$
			such that $B$ contains the graph of $S$, i.e.~
			\[ 
				B\supset \stsetG{S}:= \mathrm{graph}(S):=\{(\omega,S(\omega))
				\in \Omega\times [0,\infty)	\, | \, S(\omega)<\infty\}
			\]								
			and
			\[
				\mathbb{P}(S<\infty)>\mathbb{P}(\pi(B))-\epsilon,
			\]
			where $\pi(B):=\left\{\omega\in \Omega \,
			\middle |\, (\omega,t)\in B
			\text{ for some $t\in [0,\infty)$}\right\}$ 
			is the projection of $B$ onto $\Omega$.\qed
\end{Thm}			

An important consequence is the following corollary:

\begin{Cor}[\citing{EL80}{Corollary}{507}]\label{Main:68}
	If $Z$ and $Z'$ are two $\Lambda$-measurable processes,
	such that for each bounded $T\in \stm$ we have
	$Z_T\leq Z_T'$ a.s.\ (resp.\ $Z_T=Z_T'$ a.s.), then the 
	set $\{Z>Z'\}$ is evanescent (resp.\  $Z$ and $Z'$ are
	indistinguishable).\qed
\end{Cor}

Finally, we state an equivalent definition of $\Lambda$-projections. These projections were introduced in \citing{EL80}{Definition}{512}, and they provide us with a generalization of the well known
optional and predictable projections:

\begin{DefThm}[compare \BBB]\label{Main:154}
			For any nonnegative $\mathbb{F}\otimes
			\mathcal{B}([0,\infty))$-measurable process $Z$,  
			there exists a nonnegative
			$\Lambda$-measurable process $^\Lambda Z$, unique up 
			to indistinguishability, such that 
			\[
				\EW\left[\int_{[0,\infty)} Z_s\md A_s\right]
				=\EW\left[\int_{[0,\infty)} {^\Lambda Z}_s\md A_s\right]
			\]
			for any c\`adl\`ag, $\Lambda$-measurable, increasing process
			$A$.
			This process ${^\Lambda Z}$ is called the \emph{$\Lambda$-projection of $Z$,} and it is also uniquely determined by 
			\[
			    {^\Lambda Z_S}=\EW\left[Z_S\midG\aFA_S\right],\quad S\in \stm.
			\]\qed
\end{DefThm}
	
\section{Irreversible investment with inventory risk}\label{sec:1.2}

Naturally, much more intricate information dynamics involving
jumps can be considered than the simple sensor used in the previous section (see~\eqref{Main:27}). It is thus of interest to
develop a general approach to optimal control with
Meyer-$\sigma$-fields.
 It is the goal of
the present paper to do so for the problem of
irreversible investment, a non-linear stochastic 
singular control problem which has been of considerable
interest in the literature; see  e.g.\ \cite{arr66}, \cite{DP_94}, \cite{Ber98}, \cite{MZ_07}, \cite{RX11}, \cite{FP_14}, \cite{Fer15}, \cite{AZ_17}, \cite{DFF_17}.
	
Let us consider a controller who can choose her actions based on the information flow conveyed by a Meyer-$\sigma$-field $\Lambda$ satisfying
  $\mathcal{P}(\mathcal{F})\subset 
 	\Lambda\subset \mathcal{O}(\mathcal{F})$,
where $\mathcal{F}$ is a complete, right-continuous filtration generated, e.g., by the monitoring of exogenous random shocks hitting the controlled system.
So, in our irreversible investment problem, \emph{controls} are $\Lambda$-measurable,increasing and (for now) c\`adl\`ag processes $C$ starting from a given value 
\begin{align}
    C_{0-}:=\uP\in \RZ.
\end{align}
A control $C$ will incur a risk described by
\begin{align}
	\mathbb{E}\left[\int_{[0,\infty)}
					\rho_t(C_t) \md R_t\right],
\end{align}
where $\EW$ corresponds to a given fixed probability measure $\WM$ on the measurable space $(\Omega,
\FA_\infty:=\bigvee_{t\geq 0} \FA_t)$ and where $\rho$ and $R$ are as follows:
\begin{Ass}\label{Main:66}
	\begin{enumerate}[label=(\roman*)]
		\item $\md R$ is a random Borel measure on $[0,\infty)$ with $\md R(\{\infty\}):=0$.
		\item The stochastic field $\rho:\Omega\times [0,\infty)
				\times \RZ \rightarrow \RZ,\ (\omega,t,c)
				\mapsto \rho_t(\omega,c)$ satisfies:
				\begin{enumerate}[label=(\alph*)]
					\item For $\omega\in \Omega$, 
						$t\in [0,\infty)$, the function $\rho_t(\omega,\cdot)$
						is strictly convex and continuously differentiable on $\RZ$
						with 
						\[
								\lim_{c\downarrow -\infty}\frac{\partial}{\partial c}
							\rho_t(\omega,c)=-\infty,\quad 
							\lim_{c\uparrow \infty}\frac{\partial}{\partial c}
							\rho_t(\omega,c)=\infty.
						\]
					\item For $c\in \RZ$, the process 
					$\rho_{\cdot}(\cdot,c):
					\Omega\times[0,\infty)\rightarrow \RZ;
					(\omega,t)\mapsto
					\rho_{t}(\omega,c)$ is $\FA\otimes\mathcal{B}([0,\infty))$-measurable.
					\item We have
						\begin{align}\label{eq:cintegrable}
							\EW\left[\int_{[0,\infty)}
									 |\rho_t(c)|\md R_t\right]<\infty, \quad  c\in \RZ,
						\end{align}
						and
						\begin{align}\label{eq:infintegrable}
							\EW\left[\int_{[0,\infty)}
									 \inf_{c \in \RZ} \rho_t(c) \md R_t\right]>-\infty.
						\end{align}
				 \end{enumerate}
	 \end{enumerate}
\end{Ass}	

\begin{Rem}
 The process $R_t :=\md R([0,t])$, $t \geq 0$, can be viewed as a risk clock. Its jumps correspond to atoms of $\md R$ and indicate times of particular importance for a control's risk assessment. The random field $\rho=\rho_t(c)$ can be viewed as a description of, e.g., inventory risk emerging from the inventory level $c$ installed at time $t$.
\end{Rem}

The set $\Cn$ of \emph{admissible controls} consist of all controls 
$\tilde{C}$ with $\tilde{C}_{0-}:=c_0$, which exhibit limited risk in the sense that
\begin{align}\label{Main:70}
	&\EW\left[\int_{[0,\infty)}
					\rho_t(\tilde{C}_t) \md R_t\right]<\infty
\end{align}
and which have reasonable expected rewards
\begin{align}
	\mathbb{E}\left[\int_{[0,\infty)}
					(P_t\wedge 0) \md \tilde{C}_t\right]>-\infty,
\end{align}
where $P$ denotes a given $\mathbb{F}\otimes\mathcal{B}([0,\infty))$-measurable reward process.

Admissible controls then yield the (possibly infinite) value 
\begin{align}
	\Vn(\tilde{C}):=\EW\left[\int_{[0,\infty)} P_t\md \tilde{C}_t
	-\int_{[0,\infty)} \rho_t(\tilde{C}_t)\md R_t\right]\in (-\infty,\infty]
\end{align}	
and we are led to consider the control problem:
\begin{align}\label{Main:222}
	\text{Find \quad $\hat{C} \in \Cn$ \quad attaining } \quad \sup_{\tilde{C}\in\Cn} \Vn(\tilde{C}).
\end{align}
In some applications the reward process $P$ may not be fully observable to the controller at the time of decision making, i.e.~it may not be $\Lambda$-measurable. As a consequence, also the running reward $\int_{[0,t]}P_s\md C_{s}$ may not be observable to the controller unless $(P_s)_{0\leq s\leq t}$ is. Hence, the controller may not know immediately about the revenues generated from an intervention, but, of course, the controller can (and should) form an expectation about these revenues based on her information flow $\Lambda$. Mathematically, this is captured by the passage to the $\Lambda$-projection $^\Lambda P$ of $P$ as introduced in Definition and Theorem~\ref{Main:154}. Thus, denoting by $^\Lambda P$ the $\Lambda$-projection of $P$, we can rewrite the value generated by a control $\Tilde{C}\in \Cn$ as
\[
    \Vn(\Tilde{C})=\EW\left[\int_{[0,\infty)} {^\Lambda P}_t\md \tilde{C}_t
	-\int_{[0,\infty)} \rho_t(\tilde{C}_t)\md R_t\right].
\]

In order to ensure existence of a solution to the optimization problem~\eqref{Main:222}, we impose the following mild regularity conditions on the reward process $P$ which ensure just the right form of semi-continuity:

\begin{Ass}[Assumptions on the reward process]\label{Main:6}
	The reward process $P$ with $P_\infty:=0$ admits a $\Lambda$-projection $^\Lambda P$  and it satisfies the following conditions:
			\begin{enumerate}[label=(\roman*)]
				\item ${^\Lambda P}$ is of class($D^\Lambda$), i.e., the family
				$\{{^\Lambda P}_T\,|\, T\in \stm\}$ is uniformly integrable, and we have ${^\Lambda P}_S=0$ for any $\Lambda$-stopping time $S$
					such that $\md R([S,\infty))=0$ 
					almost surely.
				\item We have left-upper-semicontinuity in expectation at any predictable stopping time $S$ in the sense that for any non-decreasing sequence $(S_n)_{n\in \NZ}\subset \stm$ with $S_n<S$ on $\{S>0\}$ and $\lim_{n\rightarrow \infty} S_n=S$ we have
								\begin{align*}
									\EW\left[P_S\right]\geq
									\limsup_{n\rightarrow \infty}
											\EW\left[P_{S_n}\right].
								\end{align*}	
				\item We have $\Lambda$-$\md R$-right-upper-semicontinuity
						in expectation at every 	$S\in \stm$  in the
						sense that for any sequence $(S_n)_{n\in \NZ}\subset
						\stm$ with $S_n\geq S$ for all $n\in \NZ$ such that
						 $\lim_{n\rightarrow \infty} \md R([S,S_n))=0$ almost surely
						we have
						\[
							\EW\left[P_S\right]
							\geq \limsup_{n\rightarrow \infty}
							\EW\left[P_{S_n}\right].  
						\]
			\end{enumerate}			
\end{Ass}				

\begin{Rem}
  Condition~(i) ensures that rewards do not explode and no more rewards can be expected after all risk has evaporated from the system. Conditions~(ii) and~(iii) are needed to rule out obvious counterexamples for the existence of optimal controls.
\end{Rem}

Assumption~\ref{Main:6} does not suffice to guarantee existence of an optimal control for~\eqref{Main:222} as also shown by the explicitly constructed optimal controls in our later examples. Indeed, we need to relax to increasing controls which are merely l\`adl\`ag and not necessarily right-continuous. For such controls
we introduce the following $\mdS$-integral inspired by  \cite{CS14}:
\begin{Def}\label{Main:139}
For an increasing process $A=(A_t)_{t\geq 0}$ with $A_{0-}>-\infty$, we define the right-continuous, increasing processes for $t\in [0,\infty)$ 
\begin{align}\label{Main:143}
	\rdc{A}_t&:=\sum_{s\leq  t} \Delta^+A_s,\quad  \rdc{A}_{0-}:=0,\nonumber\\
	\ld{A}_t&:=\sum_{s\leq  t} \Delta^-A_s,\quad \ld{A}_{0-}:=0,\nonumber\\
	\dc{A}_t&:=A_t-\ld{A}_{t}-\rdc{A}_{t-},\quad \dc{A}_{0-}:=A_{0-},\nonumber\\
	\ldc{A}_t&:=A_t-\rdc{A}_{t-}=\dc{A}_t+\ld{A}_t,\quad \ldc{A}_{0-}:=A_{0-},
\end{align}
where $\Delta^+ A_t:=A_{t+}-A_t$,  $\Delta^- A_t:=A_{t}-A_{t-}$ and  $\Delta A_t:=A_{t+}-A_{t-}$ for $t\geq 0$. Moreover, we set,
for $t\geq 0$,
\begin{align}\label{Main:4}
		\int_{[0,t]}\phi_v\mdS A_v
		:=\int_{[0,t]}\phi_v\md \ldc{A}_v
					+\int_{[0,t)}(\phi_v^\ast) \md \rdc{A}_v,\nonumber\\
		\int_{[t,\infty)}\phi_v\mdS A_v
		:=\int_{[t,\infty)}\phi_v\md \ldc{A}_v
					+\int_{[t,\infty)}(\phi_v^\ast) \md \rdc{A}_v,
\end{align}	
for any measurable process $\phi$ with 
\[
	\int_{[0,\infty)}(\phi_v\wedge 0)\md \ldc{A}_v
			+\int_{[0,\infty)}(\phi_v^\ast \wedge 0) \md \rdc{A}_v>-\infty.
\]
Here, integration with respect to $\rdc{A}$ and $\ldc{A}$ is to be understood in the usual Lebesgue-Stieltjes sense; $\phi^\ast$ is the right-upper-semicontinuous
envelope of $\phi$ defined by
\begin{align}\label{Main:5}
	\phi_{t}^\ast(\omega):=\limsup_{s\downarrow t} \phi_s(\omega)
			:=\lim_{n\rightarrow \infty}\sup_{s\in (t,t+\frac1n)}
			\phi_s(\omega),\quad t\in [0,\infty).
\end{align}
Analogously, we define $\mdST$-integration by replacing in~\eqref{Main:4} the right-upper-semicontinuous envelope $\phi^\ast$ with the right-lower-semicontinuous envelope $\phi_\ast$
given by
\begin{align}\label{Main:164}
	\phi_{t\ast}:=\liminf_{s\downarrow t} \phi_s
			:=\lim_{n\rightarrow \infty}\inf_{s\in (t,t+\frac1n)}
			\phi_s, \quad t\in [0,\infty).
\end{align}
\end{Def}
Some results concerning $\mdS$-integration  
and a comparison to similar integrals in the literature are collected in Appendix~\ref{app:integration}.

Denote by  $\Cm$ the set of non-decreasing (and thus l\`adl\`ag), $\Lambda$-measurable controls $C$ starting in $C_{0-} =c_0$ which incur limited risk in the sense that they satisfy~\eqref{Main:70}
and which generate reasonable expected rewards in the sense that
\begin{align}\label{Main:9}
	\mathbb{E}\left[\int_{[0,\infty)}
					({^\Lambda P}_t\wedge 0) \mdS C_t\right]>-\infty,
\end{align}
Now we consider the following relaxed optimization problem:
\begin{align}\label{Main:2}
	\text{Find \quad $\hat{C} \in \Cm$\quad attaining }\quad\sup_{C\in\Cm} \Vm(C),
\end{align}
where
\begin{align}
        \Vm(C):=
	\EW\left[\int_{[0,\infty)} {^\Lambda P}_t \mdS C_t
	-\int_{[0,\infty)} \rho_t(C_t)\md R_t\right].
\end{align}	
Theorem~\ref{Main:14} below constructs a solution to this relaxed concave optimization problem, ensuring in particular existence in this relaxation. 
The $\mdS$-integral will turn out to be natural because we seek to maximize $V(C)$; $\mdST$-integration would be the choice in a minimization problem.
The explicit solutions computed in Section~\ref{sec:Poisson} reveal that l{\`a}dl{\`a}g controls cannot be avoided, thus underlining the need to relax the original problem~\eqref{Main:222}.
The following proposition verifies that the relaxed optimization problem~\eqref{Main:2} admits the same value as our original problem~\eqref{Main:222}:
\begin{Pro}\label{Main:307}
    Assume $P$ is an $\mathbb{F}\otimes\mathcal{B}([0,\infty))$-measurable process such that $^\Lambda P$ exists and suppose Assumption~\ref{Main:66} is satisfied.
    Then for any l\`adl\`ag control $C\in \Cm$ there are c\`adl\`ag controls $\tilde{C}^n\in \Cn$, $n\in \NZ$, such that 
    \begin{align}\label{Main:305}
        V(C)=\lim_{n\rightarrow \infty} \Vn(\tilde{C}^n).		    
    \end{align}
    In particular, we have
    \begin{align}\label{Main:309}
        \sup_{C\in \Cm}\Vm(C)=\sup_{\tilde{C}\in \Cn}\Vn(\tilde{C}).
    \end{align}
 \end{Pro}

For the proof of Proposition~\ref{Main:307} we will need the following result:

\begin{Pro}\label{Main:38}
    Under the Assumptions of Proposition~\ref{Main:307} for any $C\in \Cm$ and $\tilde{C}\in \Cn$ we have
	\begin{align*}
			\Vm(C)= \lim_{n\rightarrow \infty} \Vm(C\wedge n),\quad 
			\Vn(\tilde{C})= \lim_{n\rightarrow \infty} \Vn(\tilde{C}\wedge n).
	\end{align*}
\end{Pro}	

\begin{proof}
    This result follows via monotone convergence on the reward part and dominated convergence in the risk part of our target functionals $V$, $\tilde{V}$ as by convexity  of $c\mapsto \rho_s(c)$ we have for $n\geq c_0$ the estimate
	\[
		\left|\rho_s(C_s\wedge n)\right|\leq 
		\max\left\{\rho_s(C_s)\vee 0,\rho_s(\uP)\vee 0,
		-\left(\inf_{c\in \RZ}\rho_s(c)\wedge 0\right)\right\}.
	\]		
\end{proof}
	
\begin{proof}[Proof of Proposition~\ref{Main:307}]
By Proposition~\ref{Main:38}, we can assume without loss of generality that  $C\in \Cm$ is bounded. Note, we will also construct bounded $\tilde{C}^n\in \Cn$, $n\in \NZ$ with~\eqref{Main:305}. For such $\tilde{C}^n$, $n\in \NZ$, we get via Definition and Theorem~\ref{Main:154} and admissibility of $\tilde{C}^n$ that $\Vn(\tilde{C}^n)=\Vm(\tilde{C}^n)$.

\emph{Case $V(C)=\infty$:} In this case, one of the two right-continuous, admissible controls $\ldc{C}$, $\rdc{C}$ (see~\eqref{Main:143}) generates an infinite value for $\tilde{V}$, which then shows~\eqref{Main:305}. 

\emph{Case $V(C)<\infty$:} In this case, we can use \citing{KS98}{Proposition 2.26}{10}, to exhaust the jumps of $\rdc{C}$ by a sequence of $\FA$-stopping times $(T_n)_{n\in \NZ}$, where we can assume without loss of generality that the graphs of the stopping times are disjoint.
Indeed, if this is not the case we can consider instead
the sequence $(\tilde{T}_n)_{n\in \NZ}$ given by
\[
	\tilde{T}_1:=T_1,\quad \tilde{T}_n:=(T_n)_{\cup_{k=1}^{n-1}\{T_n\neq T_k\}},\quad n = 2,3,\dots,
\]
which exhibits the desired properties. Moreover, by \BBD,  there exists for each $n\in \NZ$ a sequence $(T_n^k)_{k\in \NZ}\subset \stm$ such that $T_n^k\geq T_n$ with $\infty>T_n^k>T_n$ on $\{T_n<\infty\}$, $\lim_{k\rightarrow \infty} T_n^k=T_n$ and $({^\Lambda P})_{T_n}^\ast=\lim_{k\rightarrow \infty} {^\Lambda P}_{T_n^k}$ almost surely. Now define
for $k\in \NZ$, $N\in \NZ$, $\tilde{C}^{k,N}\in \Cn$ by
\[
    \tilde{C}_t^{k,N}:=C_t^{c,l}+\sum_{n=1}^N\Delta^+C_{T_n}\mathbb{1}_{\stsetRO{T_n^k}{\infty}}(t),
    \quad t\in [0,\infty),
\]
where $\ldc{C}$ is defined in~\eqref{Main:143}. One can easily see that for a.e. $\omega \in \Omega$, $t\in [0,\infty)$ we have $\lim_{N\rightarrow \infty}\lim_{k\rightarrow \infty} \tilde{C}^{k,N}_t(\omega)=C_t(\omega)$. Hence, using dominated convergence (which is applicable because of~\eqref{eq:cintegrable}, \eqref{eq:infintegrable} and admissibility of $C$), we obtain by continuity of $\rho$ that
\begin{align}\label{Main:308}
    	\EW\left[\int_{[0,\infty)} \rho_t(C_t)\md R_t\right]=\lim_{N\rightarrow \infty}\lim_{k\rightarrow \infty} \EW\left[\int_{[0,\infty)} \rho_t(\tilde{C}_t^{k,N})\md R_t\right].
\end{align}
Next,
\[
   	\EW\left[\int_{[0,\infty)} {^\Lambda P_t} \md \tilde{C}_t^{k,N}\right]
   	=\EW\left[\int_{[0,\infty)} {^\Lambda P_t} \md \ldc{C}_t\right]+\sum_{n=1}^N\EW\left[{^\Lambda P_{T_n^k}}\Delta^+ C_{T_n}\right]
\]
and we have
\begin{align}
    \sum_{n=1}^N\EW\left[{^\Lambda P_{T_n^k}}\Delta^+ C_{T_n}\right]\overset{k\rightarrow \infty}{\longrightarrow}
    \sum_{n=1}^N\EW\left[({^\Lambda P})_{T_n}^\ast\Delta^+ C_{T_n}\right]\overset{N\rightarrow \infty}{\longrightarrow} \sum_{n=1}^\infty\EW\left[({^\Lambda P})_{T_n}^\ast\Delta^+ C_{T_n}\right],
\end{align}
where we have used that ${^\Lambda P}$ is of class($D^\Lambda$) and  that $C$ is bounded to apply Lebesgue's theorem twice.
Therefore, by $V(C)<\infty$,
\begin{align}\label{Main:56}
    \lim_{N\rightarrow\infty}\lim_{k\rightarrow\infty}\EW\left[\int_{[0,\infty)} {^\Lambda P_t} \md \tilde{C}_t^{k,N}\right]
    =\EW\left[\int_{[0,\infty)} {^\Lambda P_t} \mdS C_t\right].
\end{align}
Combining~\eqref{Main:308} and~\eqref{Main:56} shows that the value $\Vm(C)$ is attained as the limit of $\Vn(\tilde{C}^{k,N})$, $k,N\in \NZ$, which finishes our proof.
\end{proof}

An optimal control for~\eqref{Main:2}
will be constructed in  terms of a reference process $L^\Lambda$ emerging from the following stochastic representation of the reward process, similar to an approach taken in \cite{BR01}:
\begin{Lem}\label{Main:19}
Under Assumptions~\ref{Main:66} and~\ref{Main:6}, there exists a 
$\Lambda$-measurable process $L^\Lambda$ such that for any $\Lambda$-stopping time
$S$ we have
\begin{align}\label{Main:33}
	\EW\left[\int_{[S,\infty)}
	\left|\frac{\partial}{\partial c}\rho_t
	\left(\sup_{v\in [S,t]} L^\Lambda_v\right)\right|\md R_t
			\right]<\infty,\\
			\label{Main:13}
	 {^\Lambda P_S}=
	\EW\left[\int_{[S,\infty)}
	\frac{\partial}{\partial c}\rho_t\left(
			\sup_{v\in [S,t]} L^\Lambda_v\right) \md R_t
			\midG \aFA_{S}\right].				    
\end{align}
Here, we can choose $L^\Lambda$ to be maximal in the
sense that for any other $\Lambda$-measurable solution $\tilde{L}$,
satisfying mutatis 
mutandis the two properties~\eqref{Main:33}, \eqref{Main:13}, we have
$\tilde{L}_S\leq L^\Lambda_S$ at any $\Lambda$-stopping time $S$.
The maximal solution $L^\Lambda$ is unique up to indistinguishability and additionally satisfies
\begin{align}\label{Main:24}
			L^\Lambda_S= \essinf_{T\in \stm, T> S} \ell^\Lambda_{S,T}, \quad S \in \stm,
\end{align}					
where, for any $\Lambda$-stopping time $T$ with $S<T$, $\ell^{\Lambda}_{S,T}$ 
is defined as the (up to a $\WM$-nullset) unique $\aFA_S$-measurable random variable solving
\begin{align}\label{Main:25}
\EW\left[{P}_S
		-{P}_T\midG \aFA_{S}\right]=
		\EW\left[\int_{[S,T)} 
		\frac{\partial}{\partial c}\rho_t(	\ell^{\Lambda}_{S,T})\md R_t\midG \aFA_{S}\right]
\end{align}	
on $\{\WM\left(R_{T-}-R_{S-}>0\midG \aFA_{S}\right)>0\}$ and $\ell^{\Lambda}_{S,T}:=\infty$ elsewhere.
 \end{Lem}
	 
\begin{proof}
	Set $X_t(\omega):={^\Lambda P}_t(\omega)$, 
	$g_t(\omega,\ell):=\frac{\partial}{\partial c}\rho_t(\omega,\ell)$,
	$\mu(\omega,dt):=\md R_t(\omega)$. Then Assumptions~\ref{Main:66} and~\ref{Main:6} allow us to apply the representation theorem of \cite{BB18},
 to obtain $L^\Lambda$ with the
	desired properties.
\end{proof}
	 
 The next theorem  shows that under some additional integrability assumptions on $L^\Lambda$ the value in~\eqref{Main:2}	is finite and attained by a control explicitly constructed in terms of $L^\Lambda$:
	 
\begin{Thm}\label{Main:14} Let Assumptions~\ref{Main:66} and~\ref{Main:6} be satisfied and let $L^\Lambda$ be a $\Lambda$-measurable process 
	  satisfying~\eqref{Main:33}, \eqref{Main:13}. Suppose
	that the control $C^{L^\Lambda}$ defined by
	\begin{align}\label{Main:48}
		C^{L^\Lambda}_{0-}:=\uP ,\quad C^{L^\Lambda}_t
		:=\uP \vee \sup_{v\in [0,t]} L^\Lambda_v,\quad
				t\in [0, \infty),
	\end{align}
	satisfies	
	\begin{align}
	    \label{Main:81}
	    \EW\left[ \int_{[0,\infty)} \left\{
		\frac{\partial}{\partial c}\rho_t\left(
				 C^{L^\Lambda}_t\right)(C^{L^\Lambda}_t-\uP) \right\}\vee 0 \;\md R_t\right]<\infty.
	\end{align}
	Then $C^{L^\Lambda}$ is contained in $\Cm$ and is optimal for the relaxed problem~\eqref{Main:2}:
	\begin{align}
	    &C^{L^\Lambda}\in \argmax_{C\in \Cm}\Vm(C)
	\end{align}
	and
	\begin{align}\label{Main:22} 
	   V(C^{L^\Lambda})= \EW\left[ \int_{[0,\infty)} \left\{
		\frac{\partial}{\partial c}\rho_t\left(
				 C^{L^\Lambda}_t\right)(C^{L^\Lambda}_t-\uP) -\rho_t\left(C^{L^\Lambda}_t\right) \right\}\md R_t\right]<\infty. 
	\end{align}			
\end{Thm} 
	
\begin{Rem}
 One way to motivate the construction of an optimal control via~\eqref{Main:24} and~\eqref{Main:48} is the following. Due to concavity of the problem, first order conditions are necessary and sufficient for optimality of a control. For our problem, they essentially mean that a control $\widehat{C}$ will be optimal if, at any time $S$, it balances the (perceived) rewards of an intervention with its impact on a control's future risk assessment in such a way that, at any decision time $S\in \stm$,
 \begin{align*}
     {^\Lambda P}_{S} \leq \EW\left[\int_{[S,\infty)}\frac{\partial}{\partial c}\rho_t(\widehat{C}_t)\md R_t\midG \aFA_{S}\right]
 \end{align*}
  with equality holding true if $S$ is an optimal time to intervene. So, if $S \in \stm$ is such a time, we get for any time $T \in \stm$ with $T > S$ that
  \begin{align*}
      \EW\left[{^\Lambda P}_S
		-{^\Lambda P}_T\midG \aFA_{S}\right] &\geq
		\EW\left[\int_{[S,T)} 
		\frac{\partial}{\partial c}\rho_t(\widehat{C}_t)\md R_t\midG \aFA_{S}\right]
		 \geq
		\EW\left[\int_{[S,T)} 
		\frac{\partial}{\partial c}\rho_t(\widehat{C}_S)\md R_t\midG \aFA_{S}\right].
  \end{align*}
  Comparing this with~\eqref{Main:25} and recalling that $T$ was arbitrary shows that at times of intervention we must have 
  \begin{align}\label{Main:811}
  \widehat{C}_S \leq \essinf_T \ell^{\Lambda}_{S,T} = L^\Lambda_S,
  \end{align} 
  where the last identity is just~\eqref{Main:24}. In fact, we even get equality in the above estimates if $T=T_S$ is the next time of intervention after time $S$. This suggests that equality should hold true in~\eqref{Main:811} at times of intervention. Conversely, if $S$ is not a time of intervention, similar considerations lead to
  \begin{align*}
      \EW\left[{^\Lambda P}_S
		-{^\Lambda P}_{T_S}\midG \aFA_{S}\right] &\leq
		\EW\left[\int_{[S,T_S)} 
		\frac{\partial}{\partial c}\rho_t(\widehat{C}_t)\md R_t\midG \aFA_{S}\right]
		 =
		\EW\left[\int_{[S,T_S)} 
		\frac{\partial}{\partial c}\rho_t(\widehat{C}_S)\md R_t\midG \aFA_{S}\right],
  \end{align*}
  and, so, $\widehat{C}_S \geq \ell^{\Lambda}_{S,T_S} \geq L^\Lambda_S$. Admissible controls being increasing above $c_0$, we thus expect $\widehat{C}_t = c_0 \vee \sup_{v \in [0,t]} L^{\Lambda}_v$ as in~\eqref{Main:48} to be optimal, a conjecture confirmed rigorously by our Theorem~\ref{Main:14}.
\end{Rem}
			
The rest of this section is devoted to the proof of Theorem~\ref{Main:14}. We start by showing a technical result:

\begin{Lem}\label{Main:72}
	In the setting of Theorem~\ref{Main:14}, we have for any $\FA$-stopping time $T$ on 
	$\{\Delta^+ C_T^{L^\Lambda}>0\}$ that
				\begin{align}\label{Main:53}
					\left({^\Lambda\left(\int_{[\cdot,\infty)}
					\frac{\partial}{\partial c}\rho_t\left(
							 C_t^{L^\Lambda} \right) \md R_t\right)}\right)_T^\ast
					=\left({^\Lambda\left(\int_{[\cdot,\infty)}
					\frac{\partial}{\partial c}\rho_t\left(
							 \sup_{v\in [\cdot,t]}L^\Lambda_v\right) \md R_t\right)}\right)_T^\ast,						\end{align}
	where $(\cdot)^\ast$ is defined in~\eqref{Main:5}.
\end{Lem}

\begin{proof}
    For the sake of notational simplicity we will write in the following just $L$ instead of $L^\Lambda$.
	By monotonicity of $\Lambda$-projections it is clear that 
	``$\geq$'' is satisfied and we only have to prove ``$\leq$'' in~\eqref{Main:53}. 
	
	 First, let $\tilde{T}:=T_{\Gamma}$ for $\Gamma:=\{\Delta^+ C_T^L>0\}$
	 and consider a sequence of $\FA$-stopping times $(T_n)_{n\in \NZ}$ such that $T_{n}\geq T_{n+1}\geq \tilde{T}$ with $\infty>T_n>\tilde{T}$ on $\{\tilde{T}<\infty\}=\Gamma$ for all $n\in \mathbb{N}$ and such that $\lim_{n\rightarrow \infty} T_n=\tilde{T}$.
	 Then, for $n\in \NZ$ and $t\geq 0$,
	\begin{align*}
		0\wedge \frac{\partial}{\partial c}\rho_t\left(
				 c_0\right) \leq \mathbb{1}_{[T_n,\infty)}(t)\frac{\partial}{\partial c}\rho_t\left(c_0\vee 
				 \sup_{v\in [T_n,t]}L_v\right)\leq \left|\frac{\partial}{\partial c}\rho_t\left(c_0\vee
				 \sup_{v\in [0,t]}L_v\right)\right|
	\end{align*}
	and
	\[
		\lim_{n\rightarrow \infty} \mathbb{1}_{[T_n,\infty)}(t)\frac{\partial}{\partial c}\rho_t\left(
				 c_0\vee \sup_{v\in [T_n,t]}L_v\right)=\mathbb{1}_{(T,\infty)}\frac{\partial}{\partial c}\rho_t\left(c_0\vee 
				 \sup_{v\in (\tilde{T},t]}L_v\right).
	\] 
		By integrability assumption
	\eqref{Main:33} on $L$ and because $(T_n)_{n\in \NZ}$ was arbitrary, we can use dominated convergence to conclude
	\begin{align}\label{Main:73}
	\left(\int_{[\cdot,\infty)}\frac{\partial}{\partial c}\rho_t\left(c_0\vee
				 \sup_{v\in [\cdot,t]}L_v\right)\md R_t\right)_{\Tilde{T}+}=\int_{(\tilde{T},\infty)}\frac{\partial}{\partial c}\rho_t\left(c_0\vee
				 \sup_{v\in (\tilde{T},t]}L_v\right)\md R_t.
	\end{align}
	Moreover, we have for a.e. $\omega\in \Gamma$ and all $t>T(\omega)$ that
	\begin{align}\label{Main:49}
	    \sup_{v\in (T(\omega),t]}L_v(\omega)=\sup_{v\in [0,t]}L_v(\omega).
	 \end{align}

	Let us now prove ``$\leq$'' in~\eqref{Main:53}.  By \BBF,\ we know
	\begin{align}\label{Main:51}
		{^\Lambda\left(\int_{[\cdot,\infty)}
		\frac{\partial}{\partial c}\rho_t\left(
				 \uP \vee \sup_{v\in [0,t]}L_v\right) \md R_t\right)}_T^\ast\leq {^\mathcal{O}\left(\int_{(\cdot,\infty)}
		\frac{\partial}{\partial c}\rho_t\left(
				 \uP \vee \sup_{v\in [0,t]}L_v\right) \md R_t\right)}_T.
	\end{align}
	Since $\Gamma\in \FA_T$, we obtain
	\begin{align}\label{Main:47}
		&{^\mathcal{O}\left(\int_{(\cdot,\infty)}
		\frac{\partial}{\partial c}\rho_t\left(
				 \uP \vee \sup_{v\in [0,t]}L_v\right) \md R_t\right)}_T\mathbb{1}_{\Gamma}\\
		&\hspace{13ex}\overset{\eqref{Main:49}}{=} \EW\left[\left(\int_{(T,\infty)}
		\frac{\partial}{\partial c}\rho_t\left(c_0\vee
				 \sup_{v\in (T,t]}L_v\right) \md R_t \right)\mathbb{1}_{\Gamma}\midG \FA_{T}\right]\\
		&\hspace{13ex}\overset{\eqref{Main:73}}{=} \EW\left[\left(\int_{[\cdot,\infty)}
		\frac{\partial}{\partial c}\rho_t\left(c_0\vee
				 \sup_{v\in [\cdot,t]}L_v\right) \md R_t \right)_{T+}\mathbb{1}_{\Gamma}\midG \FA_{T}\right]\\
		&\hspace{13ex}={^\mathcal{O}\left(\left(\int_{[\cdot,\infty)}
		\frac{\partial}{\partial c}\rho_t\left(c_0\vee
				 \sup_{v\in [\cdot,t]}L_v\right) \md R_t\right)_+\right)}_T \mathbb{1}_{\Gamma}.
	\end{align}
	Now, we can apply again \BBF, to obtain on $\Gamma$ that
	\begin{align}\label{Main:30}
		{^\mathcal{O}\left(\left(\int_{[\cdot,\infty)}
		\frac{\partial}{\partial c}\rho_t\left(c_0\vee
				 \sup_{v\in [\cdot,t]}L_v\right) \md R_t\right)_+\right)}_T	 \leq {^\Lambda\left(\int_{[\cdot,\infty)}
		\frac{\partial}{\partial c}\rho_t\left(c_0\vee
				 \sup_{v\in [\cdot,t]}L_v\right) \md R_t\right)}_{T\ast},
	\end{align}
	where $(\cdot)_\ast$ denotes the right-lower-semicontinuous
	envelope defined in~\eqref{Main:164}.
	Next, we will need the following claim, which we will prove at the end:
	
	\emph{Claim:} There exists a sequence $(T_n)_{n\in \NZ}\subset \stm$ such that for $n\in \NZ$ we have
	$T_{n}\geq T_{n+1}\geq \tilde{T}$, $T_n>\tilde{T}$ on $\{\tilde{T}<\infty\}$, $L_{T_n}\geq c_0$ on $\{T_n<\infty\}$ and $\lim_{n\rightarrow \infty}T_n=\tilde{T}$.
	
	Using the sequence of $\Lambda$-stopping times from the previous claim leads on $\Gamma$ to
	\begin{align}\label{Main:3}
	    {^\Lambda\left(\int_{[\cdot,\infty)}
		\frac{\partial}{\partial c}\rho_t\left(c_0\vee
				 \sup_{v\in [\cdot,t]}L_v\right) \md R_t\right)}_{T\ast}\nonumber
		\leq&\ \limsup_{n\rightarrow \infty} 
		{^\Lambda\left(\int_{[\cdot,\infty)}
		\frac{\partial}{\partial c}\rho_t\left(c_0\vee
				 \sup_{v\in [\cdot,t]}L_v\right) \md R_t\right)}_{T_n}\\
		=&\ \limsup_{n\rightarrow \infty} \nonumber
		{^\Lambda\left(\int_{[\cdot,\infty)}
		\frac{\partial}{\partial c}\rho_t\left(
				 \sup_{v\in [\cdot,t]}L_v\right) \md R_t\right)}_{T_n}\\
		\leq &\ 
		{^\Lambda\left(\int_{[\cdot,\infty)}
		\frac{\partial}{\partial c}\rho_t\left(
				 \sup_{v\in [\cdot,t]}L_v\right) \md R_t\right)}_{T}^\ast.
	\end{align}
	Combining~\eqref{Main:51}, \eqref{Main:47}, \eqref{Main:30} and~\eqref{Main:3} completes the proof of ``$\leq$'' in~\eqref{Main:53} once we have proven the above claim.
	
	\emph{Proof of the claim:} The proof is analogous to the proof of \BBD. For $n\in \NZ$ we set $\epsilon_n:=2^{-n}$ and
\begin{align}
	B_n:=\stsetO{\tilde{T}}{\infty}\cap \stsetRO{0}{\tilde{T}+\frac1n,
	}\cap \left\{L\geq c_0\right\}\in \Lambda.
\end{align}
Then there exists by the Meyer Section Theorem (see Theorem~\ref{Main:84}) for each $n\in \NZ$ a $\Lambda$-stopping time $S_n$ such that the graph of $S_n$ is contained in $B_n$
and $\WM(S_n<\infty)>\WM(\pi(B))-\epsilon_n$. Now we set 
$T_n:=\min_{k\in \{1,\dots,n\}} S_k$, $n\in \NZ$. Using $\pi(B_n)=\Gamma$,
a Borel-Cantelli argument shows that the sequence $(T_n)_{n\in \mathbb{N}}\subset \stm$
will satisfy the desired properties.
\end{proof}

Now we have our tools at hand to prove the main result of this section:		

\textbf{Proof of Theorem~\ref{Main:14}:} 
	For the sake of notational simplicity, we will again just write $L$ instead of $L^\Lambda$. We start with the observation that convexity of $\rho=\rho_t(c)$ in $c$ gives
	\[
	    \inf_{c\in \RZ}\rho(c)-\rho_t(c_0)
	    \leq \rho(C^L)-\rho(c_0)\leq \frac{\partial}{\partial c}\rho\left(
				 C^{L}\right) (C^{L}-\uP).
	\]
	So, \eqref{Main:81} in conjunction with~\eqref{eq:cintegrable} and~\eqref{eq:infintegrable} implies that $\frac{\partial}{\partial c}\rho\left(C^{L}\right) (C^{L}-\uP)$ is $\WM\otimes \md R$-integrable and that $C^L$ of~\eqref{Main:48} has finite risk in the sense of~\eqref{Main:70}. In particular, the expectation in~\eqref{Main:22} is finite. 
	
	As $L$ solves~\eqref{Main:13} we obtain by monotonicity of $\Lambda$-projections
	\[
		{^\Lambda P}\leq {^\Lambda\left(\int_{[\cdot,\infty)}
		\frac{\partial}{\partial c}\rho_t\left(
				 C^{L}_t\right) \md R_t\right)}.
	\]
	Therefore, we get for any bounded $C\in \Cm$ that
	\begin{align}\label{Main:21}
		\EW \left[ \int_{[0,\infty)} {^\Lambda P}_s \mdS C_s\right]
		&\leq  \EW\left[ \int_{[0,\infty)} {^\Lambda\left(\int_{[\cdot,\infty)}
		\frac{\partial}{\partial c}\rho_t\left(
				 C^{L}_t\right) \md R_t\right)}_s \mdS C_s\right]\\
		&=\ \EW\left[ \int_{[0,\infty)} \int_{[s,\infty)}
		\frac{\partial}{\partial c}\rho_t\left(
				 C^{L}_t\right) \md R_t  \mdS C_s\right]\nonumber\\
		&=\ \EW\left[ \int_{[0,\infty)} 
		\frac{\partial}{\partial c}\rho_t\left(
				 C^{L}_t\right) (C_t-\uP) \md R_t  \right].\label{Main:78}
	\end{align}
	Here, by boundedness of $C$ and integrability conditions~\eqref{eq:cintegrable}, \eqref{eq:infintegrable} and~\eqref{Main:33}, 
	we were allowed to use in the first 
	equality the projection identity of Proposition~\ref{app:5} and in the second 
	equality the Fubini-like Proposition~\ref{app:1}.
	The preceding estimate together with convexity of $\rho$ yields
	\begin{align}\label{Main:46}
		V(C)&\leq 	\EW\left[ \int_{[0,\infty)} 
		\left\{\frac{\partial}{\partial c}\rho_t\left(
				 C^{L}_t\right) (C_t-\uP)-\rho_t(C_t)\right\} \md R_t \right]\\
		&\leq 	\EW\left[ \int_{[0,\infty)} 
		\left\{\frac{\partial}{\partial c}\rho_t\left(
				 C^{L}_t\right) (C^{L}_t-\uP)-\rho_t(C^{L}_t)\right\} \md R_t \right].
	\end{align}			 
	Using Proposition~\ref{Main:38}, we thus have found an upper bound on the value $V(C)$ generated by an arbitrary admissible control $C\in \Cm$. 
	
	It remains to prove that $C^L$ satisfies the admissibility condition~\eqref{Main:9} and the identity in~\eqref{Main:22}. 
	For admissibility note first that
	\begin{align}\label{Main:100}
	    \int_{[0,t]} \left|\frac{\partial}{\partial c}\rho_t\left(C_t^L
		\right)\right| \mdS C_s^L=\left|\frac{\partial}{\partial c}\rho_t\left(C_t^L\right)\right|(C_t^L-c_0)\
		\in \mathrm{L}^1(\WM\otimes \md R).
	\end{align}
    Thus, we can apply Fubini's theorem for $\mdS$-Integrals (see Proposition~\ref{app:1}) to deduce that 
	\begin{align}\label{Main:145}
	    \EW\left[\int_{[0,\infty)}\int_{[s,\infty)} \left|\frac{\partial}{\partial c}\rho_t\left(C_t^L
		\right)\right| \md R_t \mdS C_s^L\right]<\infty.
	\end{align}
	It follows that
		\begin{align}
	    \infty&>\EW\left[\int_{[0,\infty)}\int_{[s,\infty)} \frac{\partial}{\partial c}\rho_t\left(C_t^L
		\right)\vee 0\, \md R_t \mdS C_s^L\right]\\
		&\geq \EW\left[\int_{[0,\infty)}{^\Lambda\left( \int_{[s,\infty)} \frac{\partial}{\partial c}\rho_t\left(C_t^L
		\right)\vee 0\, \md R_t\right)}  \mdS C_s^L\right]\\
		&\geq  \EW\left[\int_{[0,\infty)}{^\Lambda \left(\int_{[s,\infty)} \frac{\partial}{\partial c}\rho_t\left(\sup_{v\in [s,t]}L_v
		\right)\md R_t\right)}  \vee 0 \, \mdS C_s^L\right]
		\overset{\eqref{Main:13}}{=} \EW\left[\int_{[0,\infty)}{^\Lambda P_s} \vee 0\, \mdS C_s^L\right],
	\end{align}
	where we have used Proposition \ref{app:5} in the first estimate.
	Hence,
	\begin{align}
	    \EW\left[\int_{[0,\infty)}{^\Lambda P_s} \vee 0\, \mdS C_s^L\right]< \infty,
	\end{align}
	and so we can use monotone convergence to obtain analogously to \eqref{Main:78} that
	\begin{align*}
	    	\EW \left[ \int_{[0,\infty)} |{^\Lambda P}_s| \mdS C_s^L\right]
	    	&= \lim_{n\rightarrow \infty} \EW \left[ \int_{[0,\infty)} |{^\Lambda P}_s| \mdS (C_s^L\wedge n)\right]\\
		&\leq   \EW\left[ \int_{[0,\infty)} 
		\left|\frac{\partial}{\partial c}\rho_t\left(
				 C^{L}_t\right) (C^L_t-\uP)\right| \md R_t  \right]<\infty,
	\end{align*}
	which shows that $C^L$ is admissible. 
	
	Clearly, for the identity in \eqref{Main:22}, we now only have to show equality in~\eqref{Main:21} for $C=C^L$, i.e.
	\begin{align}\label{Main:26}
		\EW &\left[ \int_{[0,\infty)} {^\Lambda\left(\int_{[\cdot,\infty)}
		\frac{\partial}{\partial c}\rho_t\left(
				 \sup_{v\in [\cdot,t]}L_v\right) \md R_t\right)}_s \mdS C_s^L\right]\\
		&\hspace{15ex}=  \EW\left[ \int_{[0,\infty)} {^\Lambda\left(\int_{[\cdot,\infty)}
		\frac{\partial}{\partial c}\rho_t\left(\uP\vee \sup_{v\in [0,t]}L_v
		\right) \md R_t\right)}_s \mdS C_s^L\right].
	\end{align}
	To this end, note first that the latter expectation is well defined due to~\eqref{Main:100} and~\eqref{Main:145}.
	
	Let us split the $\mdS C^L$-integral in \eqref{Main:26} into the contributions coming from $\rdc{(C^L)}$ and from $\ldc{(C^L)}$ (see \eqref{Main:143}).
		For the $\rdc{(C^L)}$-contribution (see~\eqref{Main:143}) to \eqref{Main:26} we obtain by \citing{KS98}{Proposition 2.26}{10}, that there exists
	a sequence $(T_n)_{n\in \NZ}$ of $\mathcal{F}$-stopping times with disjoint graphs (exactly as in the proof of Proposition~\ref{Main:307})
	which exhaust the jumps of $\rdc{(C^L)}$. Hence, we obtain 
	\begin{align}
		 & \int_{[0,\infty)} {^\Lambda\left(\int_{[\cdot,\infty)}
		\frac{\partial}{\partial c}\rho_t\left(
				 \sup_{v\in [\cdot,t]}L_v\right) \md R_t\right)}_s^\ast \md \rdc{(C^L)}_s\\
		&\hspace{15ex}=  \int_{[0,\infty)} {^\Lambda\left(\int_{[\cdot,\infty)}
		\frac{\partial}{\partial c}\rho_t\left(\uP \vee\sup_{v\in [0,t]}L_v
		\right) \md R_t\right)}_s^\ast \md \rdc{(C^L)}_s,\nonumber
	\end{align}
	by application of Lemma~\ref{Main:72} to
	$T_n$ for $n\in \NZ$.	
	We will argue at the end that the contribution from $\ldc{(C^L)}$ satisfies
	\begin{align}\label{Main:23}
		& \int_{[0,\infty)} \int_{[s,\infty)}
		\frac{\partial}{\partial c}\rho_t\left(
				 \sup_{v\in [s,t]}L_v\right) \md R_t \md \ldc{(C^L)}_s\\
		&\hspace{20ex}=   \int_{[0,\infty)} \int_{[s,\infty)}
		\frac{\partial}{\partial c}\rho_t\left(\uP \vee \sup_{v\in [0,t]}L_v
		\right) \md R_t \md \ldc{(C^L)}_s.\nonumber
	\end{align}
	Granted this identity, we get by Proposition~\ref{app:5} and~\eqref{Main:145} that we can
	drop the projections in~\eqref{Main:26}, which completes our proof of identity~\eqref{Main:26}.

	For the proof of~\eqref{Main:23}, we fix $\omega\in \Omega$. We will show the result separately for the left jumps of $C$ and for its continuous part $C^c$ (see~\eqref{Main:143}).
	 First, we get for $s\in [0,\infty)$ with
	$C_{s-}^{L}(\omega)<C_s^{L}(\omega)$ that $L_s(\omega)
	>c_0\vee\sup_{v\in [0,s)} L_v(\omega)$ and therefore
	$\sup_{v\in [s,t]} L_v(\omega)=\uP \vee\sup_{v\in [0,t]} L_v(\omega)$
	for $t\geq s$.
	For the $\md C^c$-contribution we can restrict to points $s\in [0,\infty)$ such that $\Delta R_s(\omega)=0$ as for fixed $\omega$ the process $R$ can have only countably many jumps. Moreover, the measure $\md C^c(\omega)$ is supported by the set $\{s\in [0,\infty)\,|\,C_t^c(\omega)>C_s^c(\omega) \text{ for all } t>s\}$.
	So, let $s\in [0,\infty)$ such that $\Delta R_s(\omega)=0$ and for all $t>s$ we have $C_s^{c}(\omega)<C_t^{c}(\omega)$.
	Then we get for any $t>s$ that $\sup_{v\in (s,t]} L_v(\omega)
	>c_0\vee \sup_{v\in [0,s]} L_v(\omega)$ and therefore again
	$\sup_{v\in [s,t]} L_v(\omega)=\uP\vee \sup_{v\in [0,t]} L_v(\omega)$. Hence,
	\begin{align}
	     \int_{(s,\infty)}\left|\frac{\partial}{\partial c}\rho_t\left(\omega,\sup_{v\in [s,t]}L_v(\omega)\right)\right|\md R_t(\omega)
	    =\, \int_{(s,\infty)}\left|\frac{\partial}{\partial c}\rho_t\left(\omega,C_t^L(\omega)\right)\right|\md R_t(\omega)
	\end{align}
	and the rest follows because $\Delta R_s(\omega)=0$.
\hfill $\square$

\section{Optimal irreversible investment in a compound Poisson setting with jump sensor}\label{sec:Poisson}

In this section we will illustrate in a compound
    Poisson process framework how different
    Meyer-$\sigma$-fields lead to different optimal
    controls in an irreversible investment problem that
    can be solved explicitly using Theorem~\ref{Main:14}. Specifically, let us fix a probability
    space $(\Omega,\mathbb{F},\WM)$ with a compound
    Poisson process $\tilde{P}$  starting in $\tilde{p}$
    of the form, i.e., 
\begin{align}
	\tilde{P}_t=\tilde{p}+\sum_{k=1}^{N_t}Y_k,	\quad P_t:=\mathrm{e}^{-rt} \tilde{P}_t,\quad t\in [0,\infty),
	\quad P_\infty:=\Tilde{P}_\infty:=0,
\end{align}
where $\tilde{p}\in \RZ$, $r>0$, $N$ is a Poisson process with intensity $\lambda>0$, independent of
  the i.i.d.\ sequence of $(Y_k)_{k\in \mathbb{N}}\subset \mathrm{L}^2(\WM)$ with mean $m:=\EW[Y_1]\in \RZ$ and $\WM(Y_1=0)=0$. Let
$\mathcal{F}:=(\mathcal{F}_t)_{t\geq 0}$ be the $\WM$-augmented filtration generated by $\tilde{P}$.
Define the risk clock 
\[
	R_t:=\int_{(0,t]} \mathrm{e}^{-rs} 
\mathrm{d} N_s, \quad t\in [0,\infty),
\]
and, furthermore, choose $\rho_t(c):=\frac12 c^2$,
so that obviously $\frac{\partial}{\partial c} \rho_t(c)=c$, $c\in \RZ$.
One can readily check that Assumption~\ref{Main:66} is satisfied. 

We want to allow immediate reactions by our controller only for sufficiently large jumps and thus restrict controls to be  $\Lambda^\eta$-measurable, 
where
\begin{align}\label{Main:92}
	\Lambda^\eta := 
	{^\WM \sigma\left(Z \text{ c\`adl\`ag and $\tilde{\FA}^\eta$-measurable}\right)}
\end{align}
for a fixed sensitivity threshold $\eta \in [0,\infty]$, where $\tilde{\FA}^\eta$ is generated by
$\tilde{P}^\eta:=\tilde{P}_-+\Delta\tilde{P}\mathbb{1}_{\{|\Delta \tilde{P}|\geq \eta\}}$
 and $\WM$ symbolizes that we consider the $\WM$-completion of the $\sigma$-field at hand. The probability with which the controller's sensor fails to alert is thus
\begin{align}\label{Main:91}
	    p(\eta):=\WM(|Y_1|< \eta) \in [0,1].    
	\end{align}
The optimization problem~\eqref{Main:2} now takes the form
 \begin{align}\label{Main:8}
	        \sup_{C\in \Cm}\EW\left[\int_{[0,\infty)} {^{\Lambda^\eta} P}_t \mdS C_t
		-\frac12 \int_{[0,\infty)} C_t^2\md R_t\right].
\end{align}
In this problem, our controller is confronted with a reward process $P$ that will jump due to external shocks $(Y_k)_{k \in \mathbb{N}}$ hitting at exponential times that also trigger the risk assessments in the clock $R$. For large enough shocks (when $|Y_k|\geq \eta$) the controller receives a warning signal from a sensor that affords her the opportunity to adjust the control $C$ before the risk assessment is done; for smaller shocks, though, the controller receives no such signal and can only react after they have struck.

We want to construct an optimal control for problem~\eqref{Main:8} via Theorem~\ref{Main:14}. For that we need that the ${\Lambda^\eta}$-projection of $P$ satisfies Assumption~\ref{Main:6}. This will not hold true in the optional case $p(\eta)=0$ which we will cover separately in Section~\ref{sec:optional} below. For a fallible sensor, i.e.~in the case $p(\eta)>0$, the following proposition characterizes the ${\Lambda^\eta}$-projection and shows that Theorem~\ref{Main:14} will indeed lead to an optimal control. Moreover, it shows that the maximal solution $L^{\Lambda^\eta}$ to~\eqref{Main:33} is a function of the reward process and the sensor:

\begin{Lem}\label{Main:11}
\begin{enumerate}[label=(\roman*)]
    \item 
   For all $\eta\in [0,\infty]$ we have at every $\Lambda^\eta$-stopping time $T$ that
   \begin{align}\label{Main:77}
        \Delta N_T&=
                \mathbb{1}_{\{|Y_{N_T}|\geq \eta\}\cap \{\Delta N_T>0\}}, \quad \text{ $\WM$-a.s.},\\
       {^{\Lambda^\eta} \tilde{P}_T}&=\tilde{P}_T,
                \quad\text{ $\WM$-a.s.},\label{Main:55}
   \end{align}
   and we have $\WM$-almost surely that
		\begin{align}\label{Main:35}
		{^{\Lambda^\eta}(\Delta N)}_t&=\mathbb{1}_{\{|Y_{N_t}|\geq \eta\}\cap \{\Delta N_t>0\}},\\\label{Main:83}
		    {^{\Lambda^\eta} P}_t= {^{\Lambda^\eta} \tilde{P}}_t \mathrm{e}^{-rt}&=\tilde{P}^\eta_t\mathrm{e}^{-rt}
		    :=\left(\tilde{P}_{t-}+\Delta  \tilde{P}_t\mathbb{1}_{\{|\Delta  \tilde{P}_t| \geq \eta\}}\right)\mathrm{e}^{-rt},\quad t\in [0,\infty).
		\end{align}
		In the case $p(\eta)=0$ (resp.\  $p(\eta)=1$), we have $\Lambda^\eta=\mathcal{O}(\FA)$ (resp.\  $\Lambda^\eta=\mathcal{P}(\FA)$). 
	\item If the sensor is imperfect, i.e.~if
	$p(\eta)>0$ (see~\eqref{Main:91}),
	then ${^{\Lambda^\eta}P}$ satisfies Assumption~\ref{Main:6}. The maximal solution $L^{\Lambda^\eta}$ to~\eqref{Main:33} and~\eqref{Main:13} is given by
	\begin{align}\label{Main:34}
		    L_t^{\Lambda^\eta}=\ell^\eta({^{\Lambda^\eta} \tilde{P}_t},{^{\Lambda^\eta}(\Delta  N)}_t)
		    =\begin{cases}
		        \ell^\eta(\tilde{P}_{t-},0),& |\Delta P_t|\geq \eta,\\
		        \ell^\eta(\tilde{P}_t,1),& |\Delta P_t|< \eta,
		    \end{cases}
		    \quad  t\in [0,\infty),
		\end{align}
	where, for $p\in \RZ$, $\Delta \in \{0,1\}$,
	\begin{align}\label{Main:126}
        \ell^\eta(p,\Delta):=\inf_{0<T\in \stme}\ell_T(p,\Delta)>-\infty
    \end{align}
	and, for random times $T>0$,
    \begin{align}\label{Main:121}
        \ell_T(p,\Delta):=
       \frac{\left(1-
    			\EW\left[\mathrm{e}^{-rT}\right]\right)p
    			-\EW\left[\mathrm{e}^{-rT}	\sum_{k=1}^{N_T}Y_k \right]}{\EW\left[R_{T-}
    			\right]+\Delta}.
    \end{align}
	  with the convention $\frac{\cdot}{0}=\infty$. Moreover, the corresponding control $C^{L^{\Lambda^\eta}}$  from~\eqref{Main:48} satisfies~\eqref{Main:81} and it is optimal with a finite value $V(C^{L^{\Lambda^\eta}})<\infty$.
	\end{enumerate}
\end{Lem}

\begin{proof}
\emph{Proof of (i).} 
    First, it is well known that ${^{\mathcal{P}} N}=N_-$ up to indistinguishability.
	    Next for $T\in \stme$ and $A:=\{|\Delta  \tilde{P}_T|\geq \eta\}=\{|Y_{N_T}|\geq \eta\}\cap \{\Delta N_T>0\}$ one can see that the $\FA$-stopping time $T_A$ is a $\Lambda^\eta$-stopping time and $T_{A^c}$ is $\Lambda^\eta$-totally inaccessible (see Definition~\ref{Main:138}).  Hence, we get by Proposition~\ref{Main:165} that
    \begin{align}\label{Main:175}
        {^{\Lambda^\eta} \tilde{P}_T}=\tilde{P}_{T}\mathbb{1}_A+{^\mathcal{P}\tilde{P}_{T}}\mathbb{1}_{A^c}
        =\Delta \tilde{P}_T\mathbb{1}_{\{|\Delta \tilde{P}_T|\geq \eta\}}+\tilde{P}_{T-},
    \end{align}
    which implies~\eqref{Main:83} by Corollary~\ref{Main:68} of the Meyer-Section Theorem. Moreover, from ${^\mathcal{P} N}=N_-$ we get
    \[
        {^\mathcal{P}\tilde{P}_{T}}\mathbb{1}_{A^c}=
        {\tilde{P}_{T-}}\mathbb{1}_{A^c}
        =\left(\tilde{p}+\sum_{k=1}^{N_{T-}}Y_k\right)\mathbb{1}_{A^c}
        =\left(\tilde{p}+\sum_{k=1}^{N_{T}}Y_k\right)\mathbb{1}_{A^c}
        =\tilde{P}_T\mathbb{1}_{A^c},
    \]
    which shows in combination with~\eqref{Main:175} equation~\eqref{Main:55}.
    The same argument also shows~\eqref{Main:77}, which then implies again by Corollary~\ref{Main:68} equation~\eqref{Main:35}.
    
    \emph{Case $p(\eta)=0$.} In this case we have for any $\FA$-stopping time $T$ that $\tilde{P}_T=\tilde{P}^\eta_T$, almost surely, which implies, by a corollary of the Meyer Section Theorem (see Corollary~\ref{Main:68}), that $\tilde{P}$ and $\tilde{P}^\eta$ are indistinguishable. Hence, by Definition and Theorem~\ref{Main:107} the process $\tilde{P}$ is $\Lambda^\eta$-measurable.
	   Moreover, using \citing{DM78}{Theorem 97 (a)}{147}, we get
	   \begin{align}
	       \mathcal{O}(\tilde{\FA})=\mathcal{P}(\tilde{\FA})\vee \sigma(\tilde{P})\subset \Lambda^\eta \subset \mathcal{O}(\FA).
	   \end{align}
        This finishes our proof, as by \citing{EL80}{Example 1\hspace{-0.01cm}\raisebox{0.1cm}{$\circ$})}{509}, the $\WM$-completion of an optional $\sigma$-field with respect to a right-continuous filtration is the optional $\sigma$-field with respect to the augmented filtration.
        
        \emph{Case $p(\eta)=1$.} In this case we have for any $\FA$-stopping time $T$ that $\tilde{P}_{T-}=\tilde{P}^\eta_T$, almost surely, which implies, by a corollary of the Meyer Section Theorem (see Corollary~\ref{Main:68}), that $\tilde{P}_-$ and $\tilde{P}^\eta$ are indistinguishable. Hence, by Definition and Theorem~\ref{Main:107} the process $\tilde{P}^\eta$ is $\mathcal{P}(\FA)$-measurable.
       Next, one can derive analogously to \citing{DM78}{Theorem 97 (a)}{147}, that
       $\tilde{\Lambda}^\eta=\mathcal{P}(\tilde{\FA})\vee\sigma(\tilde{P}^\eta)$.
        As $\tilde{P}^\eta$ is $\mathcal{P}(\FA)$-measurable we obtain 
        \[
       \mathcal{P}(\tilde{\FA})\subset \tilde{\Lambda}^\eta=\mathcal{P}(\tilde{\FA})\vee\sigma(\tilde{P}^\eta)\subset \mathcal{P}(\FA).
        \]
       The rest follows now by \citing{EL80}{Example 2\hspace{-0.01cm}\raisebox{0.1cm}{$\circ$})}{509}, which gives us that the $\WM$-completion of $\mathcal{P}(\tilde{\FA})$ is given by $\mathcal{P}(\FA)$.
    
    \emph{Proof of (ii):} ${^{\Lambda^\eta} P}$ obviously satisfies (i) of Assumption~\ref{Main:6}. Property (ii) of this assumption holds by Fatou's lemma via $^\mathcal{P}N=N_-$. Finally, for (iii) of Assumption~\ref{Main:6} note that for  $\Lambda^\eta$-stopping times $S,T$ we have
	   $\WM(\md R([S,T))>0|\aFAe_S)>0$ on $\{T>S\}$ by~\eqref{Main:91}. Hence, any sequence of $\Lambda^\eta$-stopping times as considered in
        condition (iii) must decrease to $S$ almost surely and therefore Assumption
	  ~\ref{Main:6} is satisfied by right-continuity of $P$ and Fatou's Lemma. Now, the process $L^{\Lambda^\eta}$ exists by Lemma~\ref{Main:19} and $C^{L^{\Lambda^\eta}}$ is optimal by Theorem~\ref{Main:14} since it satisfies the integrability conditions as verified next:
	    
	    \emph{\eqref{Main:81} is satisfied.}  As $L^{\Lambda^\eta}$ satisfies~\eqref{Main:24} we see that
	    \[
	        L^{\Lambda^\eta}_S\leq \ell^{\Lambda^\eta}_{S, \infty} =\frac{{^{\Lambda^\eta} \tilde{P}_S}}{\EW[R_{\infty-}]+{^{\Lambda^\eta} (\Delta N)_S}}\leq
	        \frac{|{^{\Lambda^\eta} \tilde{P}_S}|}{\EW[R_{\infty-}]}=\frac{r}{\lambda}|{\tilde{P}^\eta_S}|, \quad S\in \stme,
	    \]
	    and, hence, \eqref{Main:81} follows by $Y_1\in \mathrm{L}^2(\mathbb{P})$ and $\EW[R_{\infty-}]=\frac{\lambda}{r}<\infty$.
	   
	   \emph{Equation~\eqref{Main:34}.}
	   We claim that it is enough to show
\begin{align}\label{Main:31}
	            L_S^{\Lambda^\eta}=\inf_{0< T\in \stme}
	            \frac{\left(1-
    			\EW\left[\mathrm{e}^{-rT}\right]\right){^{\Lambda^\eta} \tilde{P}_S}
    			-\EW\left[\mathrm{e}^{-rT}	\sum_{k=1}^{N_T}Y_k \right]}{\EW\left[R_{T-}
    			\right]+ {^{\Lambda^\eta} (\Delta N)}_S},\quad S\in \stme.
	        \end{align}
	  Indeed, by~\eqref{Main:31} and Corollary~\ref{Main:68} it then suffices to establish that 
		the right hand side of~\eqref{Main:34} defines a ${\Lambda^\eta}$-measurable process or, equivalently, that $\ell^\eta$ is measurable. As $\Delta$ admits only two values we only have to argue why $p\mapsto \ell^\eta(p,\Delta)$ is measurable for fixed $\Delta$, which is clear as it is concave. Moreover, if for fixed $\Delta\in \{0,1\}$ we have $\ell^\eta(\bar{p},\Delta)=-\infty$ for some $\bar{p}\in \RZ$, then also $\ell^\eta(p,\Delta)=-\infty$ for all $p\leq \bar{p}$. By~\eqref{Main:31} this would then contradict the representation property~\eqref{Main:13} of $L^{\Lambda^\eta}$. It follows that indeed $\ell^\eta(p,\Delta)>-\infty$ for all $\Delta\in \{0,1\}$ and $p\in \RZ$.
	        
\emph{Proof of~\eqref{Main:31}.} Let $S\in \stme$. By~\eqref{Main:24}, we know that
	\[
	    L^{\Lambda^\eta}_S=\essinf_{T>S,\, T\in \stme} \ell^{\Lambda^\eta}_{S,T}.
	\]
	Due to our convention $\frac{\cdot}{0}=\infty$, we can write
	\begin{align}\label{Main:500}
	    \ell^{\Lambda^\eta}_{S,T}=\frac{\left(1-
	\EW\left[\mathrm{e}^{-r(T-S)}\midG \aFAe_S\right]\right){^{\Lambda^\eta} \tilde{P}_S}
	-\EW\left[\left(\tilde{P}_T-\tilde{P}_S\right)\mathrm{e}^{-r(T-S)}
	\midG \aFAe_S\right]}{\mathrm{e}^{rS} \EW\left[R_{T-}-R_S\midG \aFAe_S
	\right]+{^{\Lambda^\eta}(\Delta  N)_S}}. 
	\end{align}
	 We will argue next how to reduce the analysis to the
        special case $S=0$ which will become possible by results of
        \cite{DM78}, p.145-149 and \cite{CP65} on the general
        theory of processes when working on the canonical
        space with lifetime.  For this, we now assume, without
        loss of generality, that $\Omega$ is the space of
        $\RZ\cup\{\cem\}$-valued c\`adl\`ag paths with
        lifetime (see \citing{DM78}{Definition 94}{145}). We let $\tilde{P}$ denote the canonical process $\tilde{P}_t(\omega)=\omega(t)$ with its natural filtration $\tilde{\FA}$; $\WM$ is the probability under which $\tilde{P}$ follows the same compound Poisson process dynamics considered above. We define, the process $X^\eta:\Omega\times [0,\infty)\rightarrow \RZ$ by
	\begin{align*}
	    X^\eta_{0}(\omega)=\omega_0, \;X^\eta_t(\omega) := \omega_{t-}+\Delta\omega_t \mathbb{1}_{\{|\Delta\omega_t|\geq \eta\}}, \quad \omega \in \Omega, \; t \in (0,\infty).
	\end{align*}
	Here, $\cem$ denotes some point isolated from $\RZ$, $\omega_\cem$ is the element of $\Omega$ with $X(\omega_\cem)\equiv\cem$ and $\mathcal{B}_{\cem}$ the $\sigma$-field on $\RZ\cup\{\cem\}$ generated by $\mathcal{B}(\RZ)$. We have that $\bar{X}^\eta:=(\Omega,\omega_\cem,(X^\eta_t)_{t\in [0,\infty)},\RZ\cup \{\cem\},\mathcal{B}_{\cem},\cem)$ is a \emph{stochastic function with cemetery taking values in $(\RZ,\mathcal{B}(\RZ))$} (see \citing{CP65}{Def. 0.3}{248-249}, and 
	\citing{CP65}{A 1.3}{267-268}). Moreover, one can check that $\bar{X}^\eta$ satisfies the \emph{linking property} (see \citing{CP65}{Definition 5.1}{263}). As a consequence, we can apply Lemma~4.3 and Theorem~5.3 in \cite{CP65} for $\bar{X}^\eta$, which clarify the structure of $\tilde{\FA}^\eta$-stopping times, where $\tilde{\FA}^\eta$ is the filtration generated by $X^\eta$.
	To obtain results on the stopping times of the Meyer-$\sigma$-field
	\begin{align}\label{Main:71}
	    \tilde{\Lambda}^\eta:=\sigma\left\{Z \text{ is c\`adl\`ag and $\tilde{\FA}^\eta$-adapted}\right\}
	\end{align}
	we adapt \citing{DM78}{Theorem 97}{147}, which is only stated for the predictable and optional-$\sigma$-field. This adaptation can be done via the mapping
	\[
	    (h_t^\eta(\omega))_s:=(\kappa_t(\omega))_s\mathbb{1}_{\{|\Delta_t \omega|<\eta\}}+(\alpha_t(\omega))_s\mathbb{1}_{\{|\Delta_t \omega|\geq \eta\}},\quad \omega\in \Omega, \quad s,t\in[0,\infty),
	\]
	where $\kappa$ denotes the killing operator and $\alpha$ the stopping operator of \citing{DM78}{Definition 95}{146}.
	For this operator one can show that $(\omega,t)\mapsto h_t^\eta(\omega)$ is $\tilde{\Lambda}^\eta$-measurable and $\tilde{\FA}_t^\eta=(h_t^\eta)^{-1}(\tilde{\FA}_\infty)$, which then implies as in the proof of  \citing{DM78}{Theorem 97}{147}, that a process is $\tilde{\Lambda}^\eta$-measurable, if and only if it is $\tilde{\FA}^\eta$-adapted. In particular, for $S:\Omega\rightarrow [0,\infty]$ we have
	\begin{align}\label{Main:28}
	    \text{$S$ is a $\tilde{\Lambda}^\eta$-stopping time if and only if it is an $\tilde{\FA}^\eta$-stopping time.}
	\end{align}
	Finally, recall that 
	 $\Lambda^\eta$ (see~\eqref{Main:92}) is the $\WM$-completion (see Definition and Theorem~\ref{Main:107}) of $\tilde{\Lambda}^\eta$, i.e.
	\begin{align}\label{Main:52}
	    \Lambda^\eta={^\mathbb{P} \tilde{\Lambda}^\eta}={^\mathbb{P}\left(\sigma\left\{Z \text{ is c\`adl\`ag and $\tilde{\FA}^\eta$-adapted}\right\}\right)}.
	\end{align}
	
	\emph{Proof of ``$\leq$'' in~\eqref{Main:31}:} Let $T\in \stme$ with $T>0$. By \citing{EL80}{Theorem 3}{508},
	and~\eqref{Main:52} we have that $S$ and $T$ are almost surely equal to $\tilde{\Lambda}^\eta$-stopping times $\tilde{S}$, $\tilde{T}$ and, hence,
    $\tilde{\mathcal{F}}^\eta$-stopping times (see~\eqref{Main:28}). Define $U(\omega):=\tilde{T}(\tilde{P}(\omega)-\tilde{P}_0(\omega)+\tilde{p})$, $\omega \in \Omega$. One can readily check (or use \citing{CP65}{Theorem~1.3}{251} to see) that $U(\cdot)$ is an $\tilde{\FA}^\eta$-stopping time.
    Hence,
	\[
	    \hat{T}(\omega):=\tilde{S}(\omega)+U(\theta_{\tilde{S}}(\omega))
	    >\tilde{S}(\omega), \quad \omega \in \Omega,
	\]
	defines an $\tilde{\mathcal{F}}^\eta$-stopping time by \citing{CP65}{Lemma 4.3}{260}. It is
	also a $\tilde{\Lambda}^\eta$-stopping time (see~\eqref{Main:28}) and thus in particular a $\Lambda^\eta$-stopping time. We will argue next that almost surely
	\begin{align}\label{Main:501}
	 \ell^{\Lambda^\eta}_{S,\hat{T}} = \frac{\left(1-
    			\EW\left[\mathrm{e}^{-r\hat{T}}\right]\right){^{\Lambda^\eta} \tilde{P}_S}
    			-\EW\left[\mathrm{e}^{-r\hat{T}}	\sum_{k=1}^{N_{\hat{T}}}Y_k \right]}{\EW\left[R_{\hat{T}-}
    			\right]+ {^{\Lambda^\eta} (\Delta N)}_S}
	\end{align}
	which readily yields ``$\leq$'' in~\eqref{Main:31}.  For~\eqref{Main:501} we need to manipulate the conditional expectations in~\eqref{Main:500}. For the sake of brevity, we will do so only for the most complicated one:
	\begin{align}\label{Main:32}
	 \EW&\left[\left(\tilde{P}_{\hat{T}}-\tilde{P}_S\right)\mathrm{e}^{-r(\hat{T}-S)}
	\midG \aFAe_S\right]=
	\EW\left[\EW\left[\left(\tilde{P}_{\tilde{S}+U(\theta_{\tilde{S}})}-\tilde{P}_{\tilde{S}}\right)\mathrm{e}^{-rU(\theta_{\tilde{S}})}
	\midG \tilde{\FA}_{\tilde{S}}\right]\midG \tilde{\FA}^\eta_{\tilde{S}}\right]\\
	&=   \EW\left[\EW_{\tilde{P}_{\tilde{S}}}\left[(\tilde{P}_{U}-\tilde{P}_0)\mathrm{e}^{-r U}\right]\midG \tilde{\FA}^\eta_{\tilde{S}}\right]
	=\EW\left[\EW_{\tilde{p}}\left[(\tilde{P}_{\tilde{T}}-\tilde{p})\mathrm{e}^{-r\tilde{T}}\right]\midG \tilde{\FA}^\eta_{\tilde{S}}\right]=\EW\left[(\tilde{P}_{T}-\tilde{p})\mathrm{e}^{-rT}\right].
	\end{align}
	Here, we used for the first equality that, by~\eqref{Main:28}, we can replace the $\aFAe$-conditional expectation by an $\tilde{\FA}^\eta_{\tilde{S}}$-conditional expectation and that $\tilde{\FA}^\eta_{\tilde{S}}\subset \tilde{\FA}_{\tilde{S}}$; the second equality is due to the strong Markov property of $\tilde{P}$ with respect to $\tilde{\FA}$; the third equality is due to the L{\'e}vy property of $\tilde{P}$ and the choice of $U$; for the final identity we recall that under $\WM$ the canonical process $\tilde{P}$ starts in $\tilde{p}$ almost surely and $T=\tilde{T}$ almost surely.
	
	\emph{Proof of ``$\geq$'' in~\eqref{Main:31}:} Let $T\in \stme$ with $T>S$. Again by \citing{EL80}{Theorem 3}{508}, and~\eqref{Main:52},
	we have that $S$ and $T$ are almost surely equal to $\tilde{\Lambda}^\eta$-stopping times $\tilde{S}$, $\tilde{T}$ and, hence,
     $\tilde{\mathcal{F}}^\eta$-stopping times (see~\eqref{Main:28}). By \citing{CP65}{Theorem 5.3}{264},
     there thus exists an $\tilde{\FA}^\eta_{\tilde{S}}\otimes
            \tilde{\FA}^\eta_\infty$-measurable mapping
            $\tilde{U}:\Omega\times \Omega\rightarrow
            [0,\infty]$ such that, for every  $\omega\in
            \Omega$,  we have 
	\[
	    \tilde{T}(\omega)=\tilde{S}(\omega)+\tilde{U}(\omega,\theta_{\tilde{S}}(\omega))
	\]
	and also that $\omega'\mapsto \tilde{U}(\omega,\omega')>0$ is an $\tilde{\mathcal{F}}^\eta$-stopping time.
	We have to mention here that \citing{CP65}{Theorem 5.3}{264}, actually assumes that $T$ is an 
	$\tilde{\mathcal{F}}^\eta_+$-stopping time and then $\omega'\mapsto \tilde{U}(\omega,\omega')$ also would be an $\tilde{\mathcal{F}}^\eta_+$-stopping time.
	If $T$ is even an $\tilde{\mathcal{F}}^\eta$-stopping time,
	one can prove along the same lines as given in
                \citing{CP65}{Theorem 5.3}{264}, that
                $\tilde{U}$ can be found such that  $\omega'\mapsto \tilde{U}(\omega,\omega')>0$ is even an $\tilde{\mathcal{F}}^\eta$-stopping time. Now, we need again to manipulate the conditional expectations in~\eqref{Main:500}. For instance, we get that, for $\omega \in \Omega$, 
	\begin{align}
	    \EW\left[\mathrm{e}^{-r(T-S)}\midG \aFAe_{S}\right](\omega)
	    &=\EW\left[\mathrm{e}^{-r(\tilde{T}-\tilde{S})}\midG \tilde{\FA}^\eta_{\tilde{S}}\right](\omega)
	   \\& =\EW\left[\mathrm{e}^{-r\tilde{U}(\cdot,\theta_{\tilde{S}}(\cdot))}\midG \tilde{\FA}^\eta_{\tilde{S}}\right](\omega)
	    =\EW_{\tilde{P}_{\tilde{S}(\omega)}}\left[\mathrm{e}^{-r\tilde{U}(\omega,\cdot)}\right],
	\end{align}
	where we have used that $\tilde{U}$ is $\tilde{\FA}^\eta_{\tilde{S}}$-measurable in the first component and then the Markov property of $\tilde{P}$. Letting $T_\omega$ denote the $\tilde{\FA}^\eta$-stopping time $T_\omega:=\tilde{U}(\omega,\cdot-\tilde{p}+\tilde{P}_{\tilde{S}}(\omega))$, we can use the L{\'e}vy property of $\tilde{P}$ to write the last expectation as $\EW[\mathrm{e}^{-r T_\omega}]$. The other conditional expectations from~\eqref{Main:500} can be treated similarly with the same stopping time $T_\omega$, which by~\eqref{Main:28} is a $\Lambda^\eta$-stopping time.
\end{proof}

 The previous lemma shows that if~\eqref{Main:91} is satisfied, an optimal control to~\eqref{Main:8} exists. Additionally, one can see that for an explicit calculation of such an optimal control it suffices to find the maximal solution $L^{\Lambda^\eta}$ to~\eqref{Main:33} and~\eqref{Main:13} or, equivalently, to find $\ell^\eta$ from~\eqref{Main:126}. In Section~\ref{sec:optional} we show that also in the case $p(\eta)=0$ an optimal control can be constructed, but via a slight change of $P$. 
		
\subsection{Optimal predictable controls}

	We start now with the simplest case and assume that
$\WM(|Y_1|< \eta)=p(\eta)=1$. This corresponds to our controller operating without any sensor and using predictable controls: $\Lambda^\eta=\mathcal{P}=\Lambda^\infty$; see Lemma~\ref{Main:11} (i).
	\begin{Thm}[Optimal predictable control]
	\label{Main:18}
		In the case $p(\eta)=1$, an optimal control for~\eqref{Main:8} is given by
		\[
			C_t^{\mathcal{P}}:=\uP \vee \sup_{v\in [0,t]} 
			 L^{\mathcal{P}}_v,\quad t\in [0,\infty), 
		\]
		with
		\begin{align}\label{Main:12}
			L^{\mathcal{P}}_t
			=a(\tilde{P}_{t-}-b),\quad t\in [0,\infty),
		\end{align}
		where the constants $a,b$ are given by
		\begin{align*}
		a:=\frac{1}{\EW[R_{\infty-}]}=
		\frac{r}{\lambda},\quad
			b:=\sup_{0<T\in \stp} \frac{\EW\left[\mathrm{e}^{-rT}	\sum_{k=1}^{N_T}Y_k \right]}{1-
    			\EW\left[\mathrm{e}^{-rT}\right]}.
		\end{align*}
		In fact, the process $L^{\mathcal{P}}$ from~\eqref{Main:12} is the maximal solution to~\eqref{Main:33}, \eqref{Main:13}.
	\end{Thm}

	\begin{proof}
	    By Lemma~\ref{Main:11} we only have to compute $\ell^\infty$.
	    From~\eqref{Main:34} and ${^\mathcal{P} \tilde{P}}=\tilde{P}_-$, ${^\mathcal{P} (\Delta N)}\equiv 0$ (see Lemma~\ref{Main:11} (i)) we get that the maximal solution 
		$L^{\mathcal{P}}$ to~\eqref{Main:33}, \eqref{Main:13} satisfies
    \begin{align}
        L^{\mathcal{P}} &= \inf_{0<T\in \stp}\ell_{T}(\tilde{P}_-,0)
        =\inf_{0<T\in \stp} 
       \frac{1-
    			\EW\left[\mathrm{e}^{-rT}\right]}{\EW\left[R_{T-}
    			\right]}\left(\tilde{P}_--\frac{\EW\left[\mathrm{e}^{-rT}	\sum_{k=1}^{N_T}Y_k \right]}{1-
    			\EW\left[\mathrm{e}^{-rT}\right]}\right)\\
    			&=\inf_{0<T\in \stp} 
      \frac{r}{\lambda}\left(\tilde{P}_--\frac{\EW\left[\mathrm{e}^{-rT}	\sum_{k=1}^{N_T}Y_k \right]}{1-
    			\EW\left[\mathrm{e}^{-rT}\right]}\right)
    	=
      \frac{r}{\lambda}\left(\tilde{P}_--\sup_{0<T\in \stp} \frac{\EW\left[\mathrm{e}^{-rT}	\sum_{k=1}^{N_T}Y_k \right]}{1-
    			\EW\left[\mathrm{e}^{-rT}\right]}\right),
    \end{align}
    where the third equality holds because for any predictable stopping time $T$ one has
    \[
      \EW[R_{T-}]=\EW[R_T]=\lambda\EW\left[\int_{[0,T]}\mathrm{e}^{-rt}\md t\right]=\frac{\lambda}{r}\left(1-\EW\left[\mathrm{e}^{-rT}\right]\right). 
    \]
	\end{proof}
	
	\begin{Pro}\label{Main:89}
	\begin{enumerate}[label=(\roman*)]
	    \item The constant $b$ from
                      Theorem~\ref{Main:18} can also be
                      obtained by taking a supremum over all
                      stopping times and it has an alternative representation in terms of the running supremum over $\tilde{P}_-$:
    \begin{align}\label{Main:88}
        b=\sup_{0<T\in \st}\frac{\EW\left[\mathrm{e}^{-rT}	\sum_{k=1}^{N_T}Y_k \right]}{1-
    			\EW\left[\mathrm{e}^{-rT}\right]}=\frac{\EW\left[\int_{[0,\infty)} \left(\sup_{v\in [0,t]}
					\tilde{P}_{v-}-\tilde{p}\right) \md R_t\right]}{\EW[R_{\infty-}]}.
    \end{align}
    \item For $\eta\in [0,\infty]$ with $p(\eta)>0$ and $\Delta\in \{0,1\}$, we have $\ell^\eta(p,\Delta)<0$ if and only if $p<b$. 
	\end{enumerate}
\end{Pro}

\begin{proof}
    \emph{(i):} 
    The first equation in~\eqref{Main:88} follows  from Theorem~\ref{Main:18} as any $T\in \stme$ can be approximated by  $(T_n)_{n\in \NZ}\subset \mathcal{S}^{\mathcal{P}}$ defined by $T_n:=T+\frac{1}{n}$, $n\in \NZ$.
    It remains to show that $b$ satisfies also the second equality.	For that fix a predictable stopping time $S$
		and remember $L^{\mathcal{P}}=a({^{\mathcal{P}}\tilde{P}}-b)
		=a(\tilde{P}_{-}-b)$. Then we obtain by the strong Markov property
		of L\'evy processes, ${^{\mathcal{P}} N}_S=\EW[N_S|\FA_{S-}] 
		=N_{S-}$ and~\eqref{Main:13} that
		\begin{align}
		&\tilde{P}_{S-}\mathrm{e}^{-rS}={^{\mathcal{P}} P}_S=\EW\left[\int_{[S,\infty)}
				\sup_{v\in [S,t]} L_v^{\mathcal{P}}
				\md R_t \midG \FA_{S-}\right]
				=\EW\left[\int_{(S,\infty)}
				 	\sup_{v\in [S,t]} L_v^{\mathcal{P}}
				\md R_t \midG \FA_{S-}\right]\\
				&=a\EW\left[\int_{(S,\infty)}
				\sup_{v\in [S,t]} (\tilde{P}_{v-}-\tilde{P}_{S-})\,
				\md R_t\midG \FA_{S-}\right]
				+a\EW\left[\left(\tilde{P}_{S-}-b\right)\left(R_{\infty-}-R_S\right)
				\midG \FA_{S-}\right]\\
				&=a\EW\left[\int_{[0,\infty)}
					\left(\sup_{v\in [0,t]} \sum_{k=1}^{N_{v-}}Y_k\right)\,
				\md R_t\right]\mathrm{e}^{-rS}
				+a(\tilde{P}_{S-}-b)\mathrm{e}^{-rS}
				\EW\left[R_{\infty-}\right].
		\end{align}
		Solving for $b$ and recalling that $\EW[R_{\infty-}]=\frac{\lambda}{r}=\frac{1}{a}$ gives the second equality for $b$.
		
    \emph{(ii):} By definition of $\ell^\eta$ (see~\eqref{Main:126}) we get
    \[
        \ell^\eta(p,\Delta)=
        \inf_{0<T\in \stme} \frac{1-
    			\EW\left[\mathrm{e}^{-rT}\right]
    			}{\EW\left[R_{T-}            			\right]+\Delta}\left(p-\frac{\EW\left[\mathrm{e}^{-rT}	\sum_{k=1}^{N_T}Y_k \right]}{1-
    			\EW\left[\mathrm{e}^{-rT}\right]}\right),
    \]
    which shows (ii) by (i).
\end{proof}

\subsection{Optimal controls 
for an imperfect sensor}\label{sec:3.2}

	Next we consider controls that can use an imperfect jump sensor, i.e.~a sensor with probability of failing to alert 
	   $p(\eta)=\WM(|Y_1|<\eta)\in(0,1)$.
	In this class, we find the following optimal control:

	\begin{Thm}[Optimal control with  $\eta$-sensor]
	\label{Main:41}
		In the case $p(\eta)\in(0,1)$, an optimal control for~\eqref{Main:8} is
		\[
			C_t^{\Lambda^\eta}:=\uP \vee\sup_{v\in [0,t]} 
			L^{\Lambda^\eta}_v ,\quad t\in [0,\infty),
		\]
		with the maximal solution $L^{\Lambda^\eta}$ to~\eqref{Main:33} and~\eqref{Main:13} given explicitly by
		\begin{align}\label{Main:57}
				L_t^{\Lambda^\eta} = \begin{cases}
					0,& \tilde{P}^\eta_t\geq b,\
										|\Delta \tilde{P}^\eta_t|\geq \eta,
										\vspace{2ex}\\
				\frac{r}{\lambda}	(\tilde{P}^\eta_t-b),
							& \tilde{P}^\eta_t\geq b,\
							|\Delta \tilde{P}^\eta_t|
							< \eta,\vspace{2ex}\\
				\inf\limits_{\gamma^0 \in (0,B_0^{\eta}\cdot(b-\tilde{P}^\eta_t))} 
				f^\eta_1(\gamma^0,0,\tilde{P}^\eta_t)<0 
					,& \tilde{P}^\eta_t< b,\
							|\Delta \tilde{P}^\eta_t|\geq \eta,\vspace{2ex}\\
					\inf\limits_{\gamma^1 \in (-B_1^{\eta}\cdot(b-\tilde{P}^\eta_t),0)} 
				f^\eta_0(0,\gamma^1,\tilde{P}^\eta_t)<0 ,& \tilde{P}^\eta_t
						< b,\ |\Delta \tilde{P}^\eta_t|< \eta.
				\end{cases}
		\end{align}
		Here,  $b$ is as in Corollary~\ref{Main:89}, 
		\begin{align}\label{Main:105}
		    B_0^{\eta}:=1-\frac{\lambda r}{\lambda+r}\cdot\frac{1-p(\eta)}{r+\lambda\left(1-\delta\right)}<1,	\quad
		     B_1^{\eta}:=\frac{1-\lambda p(\eta)+\frac{(\lambda+r)\lambda}{r}\left(1-\delta\right)}{\lambda p(\eta)}>0 	
		\end{align}
		with 
		 \begin{align}\label{Main:106}
        \delta:=\EW\left[\mathrm{e}^{-rT^0}\right] \quad \text{ for }\quad T^0:=\inf\left\{t\geq 0\midG \sum_{k=1}^{N_t}Y_k\geq 0,\ N_t\geq 1\right\};
    \end{align}
    moreover, for $\Delta\in \{0,1\}$,
		the functions $f^\eta_\Delta:
		\RZ \times \RZ\times \RZ 
		\rightarrow \RZ$
		are given by
		\begin{align}\label{Main:16}
			f^\eta_\Delta(\gamma^0,\gamma^1,p):=\frac{\left(1-
			\EW\left[\mathrm{e}^{-rT^\eta(\gamma^0,\gamma^1)}\right]\right)p
			-\EW\left[\mathrm{e}^{-rT^\eta(\gamma^0,\gamma^1)} 
			\sum\limits_{k=1}^{N_{T^\eta(\gamma^0,\gamma^1)}}Y_k\right]}{
			\frac{\lambda}{r}\left(1-
			\EW\left[\mathrm{e}^{-rT^\eta(\gamma^0,\gamma^1)}\right]\right)
			-\EW\left[\mathrm{e}^{-rT^\eta(\gamma^0,\gamma^1)}
			\mathbb{1}_{\{|\Delta \tilde{P}_{T^\eta(\gamma^0,\gamma^1)}|\geq \eta\}}\right]+\Delta },
		\end{align}
		with
		\begin{align}\label{Main:62}
			T^\eta(\gamma^0,\gamma^1):=\inf\Big\{t \in \{{^{\Lambda^\eta} N}>0\} 
		&\left.	\midG \left(|\Delta \tilde{P}_t|< \eta\text{ and } \tilde{P}_{t-}-\Tilde{p}\geq \gamma^0\right)\right.\\
         &\left.   \text{ or } \left(|\Delta \tilde{P}_t|\geq  \eta\text{ and } \tilde{P}_t-\Tilde{p}\geq \gamma^1\right)\right\}.
		\end{align}
	\end{Thm}
	The proof of the previous theorem will be given at the end of this section.  A separate result on the optional case $p(\eta)=0$ is obtained in Section~\ref{sec:optional} and an
    illustration of all our findings is deferred to Section~\ref{sec:ill}. Before proceeding to the proof of this theorem, let
               us note how the minimal storage level
               $L^{\Lambda^\eta}$ approaches that of the
               predictable case $L^{\mathcal{P}}$ when the
               sensor becomes more and more useless as its
               failure probability tends to 1: 		

\begin{Cor}\label{Main:65}
	In the setting of Theorem~\ref{Main:41},
                consider a sequence $(\eta_n)_{n\in
                  \NZ}\subset [0,\infty]$ such that $\lim_n
                p(\eta_n)= 1$. Then the solution $L^{\Lambda^{\eta_n}}$, $n\in \NZ$, converges to $L^{\mathcal{P}}$ (see Theorem~\ref{Main:18}) for $n\rightarrow \infty$:
	\[
		\lim_{n\rightarrow \infty} L^{\Lambda^{\eta_n}}_t(\omega)
		=L^{\mathcal{P}}_t(\omega),\quad  t \in
                        [0,\infty), \quad \omega \in \Omega.
	\]
\end{Cor}
\begin{proof}
    For $\omega\in \Omega$ and $t\in [0,\infty)$, we have for any
    $\eta\geq \Delta \tilde{P}_t(\omega)$ that $\tilde{P}^\eta_t(\omega)=\tilde{P}_{t-}(\omega)$ and ${^{\Lambda^\eta} (\Delta N)_t}(\omega)=0$. 
    Hence, due to~\eqref{Main:34}, we only have to show $\lim_{n\rightarrow \infty} \ell^{\eta_n}(p,0)=\ell^{\infty}(p,0)$ for any $p\in \RZ$. In case $p\geq b$, we have by $\eqref{Main:12}$ and $\eqref{Main:57}$ that $ \ell^{\eta}(p,0)
		=\ell^{\infty}(p,0)$  for any $\eta\in[0,\infty]$. Assume henceforth $p< b$. By definition of $\ell^\eta$ we have that $\eta\mapsto \ell^\eta(p,0)$ is increasing. Hence, we infer that $\ell^{\eta_n}(p,0)
		    \leq\ell^\infty(p,0)$
	     and it suffices to show $\lim_{n\rightarrow \infty} \ell^{\eta_n}(p,0)  \geq \ell^{\infty}(p,0)$. We obtain by Theorem~\ref{Main:18} and the characterization of $a$ from Corollary~\ref{Main:93} below that
		\begin{align}\label{Main:20}
		    \ell^\infty(p,0)
		   =\frac{1-
			\EW\left[\mathrm{e}^{-rT^0}\right]}{
			\frac{\lambda}{r}\left(1-
			\EW\left[\mathrm{e}^{-rT^0}\right]\right)
			}(p-b).
		\end{align}
		In order to conclude our assertion, we will show next how to find a lower bound for $\ell^{\eta}(p,0)=\inf_{\gamma_1<0}f^\eta_0(0,\gamma_1,p)$, $\eta\in [0,\infty]$, from~\eqref{Main:16}  which converges to  $\ell^\infty(p,0)$ when $p(\eta)\rightarrow  1$. For $\gamma_1<0$, we have $T^0\geq T^\eta(0,\gamma_1)$, $\eta\in[0,\infty]$,
		and therefore, using also 
		\begin{align}
		b\geq \frac{1}{1-
			\EW\left[\mathrm{e}^{-rT^\eta(0,\gamma^1)}\right]}\EW\left[\mathrm{e}^{-rT^\eta(0,\gamma^1)} 
			\sum\limits_{k=1}^{N_{T^\eta(0,\gamma^1)}}Y_k\right],
		\end{align}
		establishes
		\begin{align}
		    f^\eta_0(0,\gamma^1,p)&\geq \frac{1-
			\EW\left[\mathrm{e}^{-rT^0}\right]}{
			\frac{\lambda}{r}\left(1-
			\EW\left[\mathrm{e}^{-rT^\eta(0,\gamma^1)}\right]\right)
			-\EW\left[\mathrm{e}^{-rT^\eta(0,\gamma^1)}
			\mathbb{1}_{\{|\Delta \tilde{P}_{T(0,\gamma^1)}|\geq \eta\}}\right] }\left(p-b\right).\label{Main:63}
		\end{align}
		Comparing~\eqref{Main:20} and~\eqref{Main:63} shows that it remains  to prove that the denominator in~\eqref{Main:63} converges to the denominator in~\eqref{Main:20} uniformly in $\gamma_1<0$ when $p(\eta) \rightarrow 1$:
		\begin{align}
		    \lim_{p(\eta)\rightarrow 1}\sup_{\gamma_1<0}\left|
			\frac{\lambda}{r}\left(\EW\left[\mathrm{e}^{-rT^0}\right]-
			\EW\left[\mathrm{e}^{-rT^\eta(0,\gamma^1)}\right]\right)
			-\EW\left[\mathrm{e}^{-rT^\eta(0,\gamma^1)}
			\mathbb{1}_{\{|\Delta \tilde{P}_{T(0,\gamma^1)}|\geq \eta\}}\right] 
			\right|=0.
		\end{align}
		This uniform convergence is a consequence of the estimates
		\begin{align}
		    0\leq \EW\left[\mathrm{e}^{-rT(0,\gamma^1)}
			\mathbb{1}_{\{|\Delta \tilde{P}_{T(0,\gamma^1)}|\geq \eta\}}\right]
			\leq 1-p(\eta),
		\quad 
		    0\leq  \EW\left[\mathrm{e}^{-rT(0,\gamma^1)}
		    -\mathrm{e}^{-rT^0}\right]
		    \leq 1-p(\eta),
		\end{align}
		which hold for any $\gamma_1<0$.
\end{proof}

The rest of this section is devoted to the \textbf{proof of Theorem~\ref{Main:41}}.	
	From Lemma~\ref{Main:11} we know that we only have to determine $\ell^\eta$ from~\eqref{Main:126}. Let us prepare this by some auxiliary results for $\ell^\eta$.

\begin{Pro}\label{Main:128}
   If $p(\eta)>0$, the functions $\ell^\eta(\cdot,0)$ and $\ell^\eta(\cdot,1)$ are both continuous, increasing and concave on $\RZ$. Moreover, on $(-\infty,b)$, both functions  are strictly increasing  and satisfy
   $\ell^\eta(\cdot,0)<\ell^\eta(p,1)<0$. On $[b,\infty)$, $\ell^\eta(\cdot,0)$ and $\ell^\eta(\cdot,1)$ are determined by
   \begin{align}
       \ell^\eta(p,0)=a(p-b) \geq 0 = \ell^\eta(p,1), \quad p\in [b,\infty),
   \end{align}
   where $a$, $b$ are the constants from Theorem~\ref{Main:18}.
\end{Pro}

\begin{proof}
  For $\Delta\in \{0,1\}$ we obtain easily from~\eqref{Main:126}, \eqref{Main:121} that $p\mapsto \ell^\eta(p,\Delta)$ is increasing. 
    Moreover, concavity follows as the infimum of affine functions is  concave. Moreover, as $\ell^\eta(\cdot,0)$, $\ell^\eta(\cdot,0)$ are concave functions taking real-values, we have that they are also continuous.
    
   Next, we obtain from Corollary~\ref{Main:89} that for $p\in \RZ$, $\Delta\in \{0,1\}$, we have $\ell^\eta(p,\Delta)<0$ if and only if $p< b$. 
   Moreover, in the case $p<b$ we can restrict in~\eqref{Main:126} to $0<T\in \stme$ with $\ell_T(p,\Delta)<0$. This immediately shows $\ell^\eta(p,0)<\ell^\eta(p,1)<0$ for $p<b$.
   
    \emph{$\ell^\eta(\cdot,1)$ for $p\geq b$:} We get with $T_n=\frac{1}{n}$ in~\eqref{Main:126} 
    that $0\leq \ell^\eta(p,1)\leq \lim_{n\rightarrow \infty} \ell_{\frac1n}(p,1)=0$.
    
    \emph{$\ell^\eta(\cdot,0)$ for $p= b$:} We have by Theorem~\ref{Main:18} and Lemma~\ref{Main:11} (ii), that $\ell^\infty(p,0)=0$ and
    $0\leq \ell^\eta(p,0)\leq \ell^\infty(p,0)=0$. 
    
    \emph{$\ell^\eta(\cdot,0)$ for $p>b$:}
    If $p>b$, we get from Theorem~\ref{Main:18}, Lemma~\ref{Main:11} (ii) and ~\eqref{Main:88} that
    \begin{align}
	    \ell^\infty(p,0)\geq \ell^\eta(p,0)&\geq \inf_{0<T\in \stm}\frac{1-
    			\EW\left[\mathrm{e}^{-rT}\right]}{\EW\left[R_{T-}
    			\right]}\left(p-b\right)\\&\geq \inf_{0<T\in \stm}\frac{1-
    			\EW\left[\mathrm{e}^{-rT}\right]}{\EW\left[R_{T}
    			\right]}\left(p-b\right)=\frac{r}{\lambda}\left(p-b\right) =\ell^\infty(p,0);
	\end{align}
   so, we must have equality everywhere and, in particular, $\ell^\eta(p,0)=\ell^\infty(p,0)>0$.

    Finally, we prove strict monotonicity on $(-\infty,b)$. For that fix $\Delta\in \{0,1\}$ and
    assume by way of contradiction that
    $p_1<p_2<b$ with $\ell^\eta(p_1,\Delta)=\ell^\eta(p_2,\Delta)$. Now there exists a sequence $(T_n)_{n\in \NZ}$ such that
    $\ell_{T_n}(p_2,\Delta)$ (see~\eqref{Main:121}) decreases to $\ell^\eta(p_2,\Delta)$ with maximal distance $\frac1n$. Then
    \begin{align}
        0=\ell^\eta(p_1,\Delta)-\ell^\eta(p_2,\Delta)&\leq \ell_{T_n}(p_1,\Delta)-\ell_{T_n}(p_2,\Delta)+\frac1n \\\nonumber
       & =
       \frac{1-
    			\EW\left[\mathrm{e}^{-rT_n}\right]}{\EW\left[R_{T_n-}
    			\right]+\Delta}(p_1-p_2)+\frac1n\leq 0+\frac1n
    \end{align}
    and therefore, by $p_1<p_2$,
    \begin{align}
        \lim_{n\rightarrow \infty}\frac{1-
    			\EW\left[\mathrm{e}^{-rT_n}\right]}{\EW\left[R_{T_n-}
    			\right]+\Delta}=0.
    \end{align}
    Hence, we get by~\eqref{Main:88} that
    \begin{align*}
    0&>\ell^\eta(p_2,\Delta)=\lim_{n\rightarrow \infty}\ \frac{1-
    			\EW\left[\mathrm{e}^{-rT_n}\right]}
    			{\EW\left[R_{T_n-}
    			\right]+\Delta}\left(p_2-\frac{\EW\left[\mathrm{e}^{-rT_n}	\sum_{k=1}^{N_{T_n}}Y_k \right]}{1-
    			\EW\left[\mathrm{e}^{-rT_n}\right]}\right)\\
    			&\geq \lim_{n\rightarrow \infty}\ \frac{1-
    			\EW\left[\mathrm{e}^{-rT_n}\right]}
    			{\EW\left[R_{T_n-}
    			\right]+\Delta}\left(p_2-b\right)=0,
    \end{align*}
    which is the desired contradiction.
\end{proof}

The next lemma shows how the infimum in~\eqref{Main:126} is attained in a relaxed sense:

\begin{Lem}\label{Main:119}
    In the case $p(\eta)>0$, we have
    for $p\in \RZ$ and $\Delta\in \{0,1\}$ that
	\begin{align}\label{Main:142}
	    \ell^\eta(p,\Delta)=\frac{\left(1-
    			\EW\left[\mathrm{e}^{-rT^{\ell^\eta}(p,\Delta)}\right]\right)p
    			-\EW\left[\mathrm{e}^{-rT^{\ell^\eta}(p,\Delta)}	\sum_{k=1}^{N_{T^{\ell^\eta}(p,\Delta)}}Y_k \right]}{\EW\left[\Delta R_{T^{\ell^\eta}(p,\Delta)}\mathbb{1}_{(H^{\ell^\eta}(p,\Delta))^c}+R_{T^{\ell^\eta}(p,\Delta)-}
    			\right]+\Delta},
	\end{align}
	where 
	\begin{align}\label{Main:59}
	    	T^{\ell^\eta}(p,\Delta)&:=\inf\left\{t\in \{{^{\Lambda^\eta} N}> {0}\} \midG \ell^\eta({^{\Lambda^\eta} \tilde{P}}_t-\tilde{p}+p,{^{\Lambda^\eta} (\Delta N)_t})\geq \ell^\eta(p,\Delta) \right\},\\\label{Main:60}
 	    H^{\ell^\eta}(p,\Delta)&:=\{{\ell^\eta}({^{\Lambda^\eta} P}_{T^{\ell^\eta}(p,\Delta)}-\tilde{p}+p,{^{\Lambda^\eta} (\Delta N)_{T^{\ell^\eta}(p,\Delta)}})\geq {\ell^\eta}(p,\Delta)\}\cap \{{^{\Lambda^\eta} N_{T^{\ell^\eta}(p,\Delta)}}>0 \}.
 	\end{align}
\end{Lem}

\begin{Rem}
    \begin{enumerate}[label=(\roman*)]
    \item
    The quadruple 
    \[
        \tau^{\ell^\eta}(p,\Delta):=(T^{\ell^\eta}(p,\Delta),\emptyset,H^{\ell^\eta}(p,\Delta),H^{\ell^\eta}(p,\Delta)^c)
    \]
    with $T^{\ell^\eta}(p,\Delta)$ and $H^{\ell^\eta}(p,\Delta)$ defined in the previous lemma is a so-called divided stopping time. The theory of divided stopping times was developed in \cite{EK81} and we will give some more details in Appendix~\ref{app:integration}.
    \item Note that~\eqref{Main:142} does not by itself construct $\ell^\eta$ from~\eqref{Main:126} as $\ell^\eta$ is also contained on the right-hand side. It shows, however, that we should restrict attention to a certain class of stopping times~\eqref{Main:59} when computing the infimum in~\eqref{Main:126}. This is the key to our explicit solution.
    \end{enumerate}
\end{Rem}

 \begin{proof}
    To verify~\eqref{Main:96},
    we will just write $L$ and $\Lambda$ instead of $L^{\Lambda^\eta}$ and $\Lambda^\eta$ in the following. 
    
    	Lemma~\ref{Main:119} will follow from Lemma~\ref{Main:11} (ii) and from the fact that $\tilde{P}$ is a L\'evy process, if we can show for any $S\in \stm$ that
	\begin{align}\label{Main:96}
	    L_S^{\Lambda} = \frac{{^\Lambda P}_S-\EW\left[P_{T_{S}^L}\midG\aFA_S\right]}{\EW\left[\Delta R_{T_{S}^L}\mathbb{1}_{H^c_S}+ R_{T_S^L-}-R_{S-}
	\midG \aFA_S\right]},
	\end{align}
	where
	\begin{align}
	    	T_S^{L}&:=\inf\left\{t> S\midG L_t\geq L_{S} \text{ and } {^\Lambda N}_t> {N_{S-}}\right\},\\
 	    H_S&:=\{L_{T_S^{L}}\geq  L_S\}\cap \{{^\Lambda N_{T_S^{L}}}> { N_{S-}}\}
 	    \cap\{S<T_S^{L}\}.
 	\end{align}
	   	            Let us consider the sequence 
\begin{align}\label{Main:58}
	T^L_{S,n}:=(T_S^L)_{H_S}\wedge \left((T_{S}^L)_{H^c_S}+\frac1n\right),\quad 
	n\in \NZ,
\end{align}
By \citing{DM78}{Theorem 50}{116}, the random time $T_S^L$ is an $\FA$-stopping time, whence	$(T_{S}^L)_{H^c_S}+\frac1n$
is a predictable $\FA$-stopping time and, therefore, also a
$\Lambda$-stopping time. Moreover we see that the graph of $(T_S^L)_{H_S}$ is contained
in $\left\{t> S\midG L_t\geq L_{S} \text{ and } {^\Lambda N}_t>{^\Lambda N}_{S-}\right\}$ and therefore by \citing{EL80}{Corollary 2}{504}, $(T_S^L)_{H_S}$ is also a $\Lambda$-stopping time. As the minimum of two $\Lambda$-stopping times is a $\Lambda$-stopping time we obtain
that $T_{S,n}^L$ from~\eqref{Main:58} is a $\Lambda$-stopping time for every $n\in \NZ$. 
We define next $T_S^{L,1}:=\inf\{t> T_S^L|N_t>N_{T_S^L}\}$. As $N$ is piecewise constant, this implies $T_S^{L,1}>T_S^L$
almost surely on $\{T_S^L<\infty\}$. Now we define the sequence $(A_n)_{n\in \NZ}\subset \aFA_{T_{S,n}^L}$ via
\begin{align}
    A_n:=\{T_S^{L,1}>T_{S,n}^L\}\cap \{T_S^L<\infty\}
\end{align}
and observe that $A_n\subset A_{n+1}$ for $n\in \NZ$ with $\bigcup_{n\in \NZ} A_n=\{T_S^L<\infty\}$. Note that
\begin{align}\label{Main:111}
    \md R((T_S^L,T_{S,n}^L])=0\quad \text{ on } \quad A_n.
\end{align}
By~\eqref{Main:31}  we have on $A_n\cap H_S^c$ that $L^\ast_{T_S^L}=L_{T_S^L+\frac{1}{n}}$ implying 
\begin{align}\label{Main:112}
    \sup_{v\in (T_S^L,t]}L_v =\sup_{v\in [T_{S,n}^L,t]}L_v\quad \text{ on } \quad A_n\cap \{t\geq T_{S,n}^L\}\cap H_S^c \quad\text{ for }\quad t\in (0,\infty).
\end{align}
Next, we get by definition of $A_n$ and~\eqref{Main:13} that
\begin{align}\label{Main:114}
    ^\Lambda P_S&=\EW\left[\int_{[S,\infty)}\sup_{v\in [S,t]}L_v\md R_t\midG\aFA_S\right]\\
    &=\lim_{n\rightarrow \infty}\EW\left[\mathbb{1}_{A_n}\int_{[S,\infty)}\sup_{v\in [S,t]}L_v\md R_t\midG\aFA_S\right]
    +L_S\EW\left[\mathbb{1}_{\{T_S^L=\infty\}}(R_{\infty-}-R_{S-})\midG\aFA_S\right].
\end{align}
The first conditional expectation above is equal to
\begin{align}
   &\ \EW\left[\mathbb{1}_{H^c_S\cap A_n}\left(
                \int_{(T_{S}^L,\infty)}\sup_{v\in (T_S^L,t]}L_v\md R_t+\int_{[S,T_{S}^L]}L_S\md R_t\right)\right.\\
                &\left.\hspace{5ex}+\mathbb{1}_{H_S\cap A_n}\left(
                \int_{[T_{S}^L,\infty)}\sup_{v\in [T_S^L,t]}L_v\md R_t+\int_{[S,T_{S}^L)}L_S\md R_t\right)\midG\aFA_S\right]\\
               \overset{\eqref{Main:111},\eqref{Main:112}}{=}&\ L_S \EW\left[\mathbb{1}_{A_n}(\Delta R_{T_{S}^L}\mathbb{1}_{H^c_S}+ R_{T_S^L-}-R_{S-})
	\midG \aFA_S\right]\\&\qquad\qquad+
               \EW\left[\mathbb{1}_{A_n} \int_{[T_{S,n}^L,\infty)}\sup_{v\in [T_{S,n}^L,t]}L_v\md R_t\midG\aFA_S\right]\\
                \overset{\eqref{Main:13}}{=}&\ L_S \EW\left[\mathbb{1}_{A_n}(\Delta R_{T_{S}^L}\mathbb{1}_{H^c_S}+ R_{T_S^L-}-R_{S-})
	\midG \aFA_S\right]
                +\EW\left[\mathbb{1}_{A_n}P_{T_{S,n}^L}\midG\aFA_S\right].
\end{align}
Plugging this into~\eqref{Main:114} leads to
\begin{align}\label{Main:17}
    ^\Lambda P_S = L_S\EW\left[\Delta R_{T_{S}^L}\mathbb{1}_{H^c_S}+ R_{T_S^L-}-R_{S-}
	\midG \aFA_S\right]
                +\EW\left[P_{T_{S}^L}\midG\aFA_S\right],
\end{align}
where we used the right-continuity of $P$ and its class($D^\Lambda$) property to apply Lebesgue's theorem. Now, \eqref{Main:17} is equivalent to the desired identity~\eqref{Main:96}.
 \end{proof}
 
 \begin{Cor}\label{Main:93}
    The constants $a,b$ defined in Theorem~\ref{Main:18} have the alternative descriptions
    \begin{align}\label{Main:141}
        a=\hat{a}:=\frac{1-\EW\left[\mathrm{e}^{-rT^0}\right]}{\EW[R_{T^0}]},\quad
        b=\hat{b}:=\frac{\EW\left[\mathrm{e}^{-rT^0}\sum_{k=1}^{N_{T^0}}Y_k\right]}{1-\EW\left[\mathrm{e}^{-rT^0}\right]},
    \end{align}
    where $T^0$ is defined in~\eqref{Main:106}.
\end{Cor}

\begin{proof}
    For $\eta=\infty$ we have $p(\eta)=1$ and ${^{\Lambda^\eta}(\Delta N)}\equiv 0$. Hence, by
     Lemma~\ref{Main:11}, Theorem~\ref{Main:18} and Lemma~\ref{Main:119}, we obtain that $T^{\ell^\eta}(p,0)=T^0$,  $H^{\ell^\eta}(p,0)=\emptyset$ giving~\eqref{Main:141}. 
\end{proof}

\begin{proof}[Proof of Theorem~\ref{Main:41}]
Lemma~\ref{Main:11} shows that we only have to determine the function  $\ell^\eta$ of~\eqref{Main:126}. For $p\geq b$ and $\Delta\in \{0,1\}$ this is already done by Proposition~\ref{Main:128}. So, assume henceforth that $p<b$. We will prove in the next lines that
\begin{align}\label{Main:40}
   \ell^\eta(p,\Delta)=f^\eta_{\Delta}(\Gamma^0(p,\Delta)-p,\Gamma^1(p,\Delta)-p)
\end{align}
with $f^\eta_\Delta$ defined in~\eqref{Main:16} and
\begin{align}\label{Main:69}
        \Gamma^0(p,\Delta)
        &:=\min\left\{\hat{p}\in (-\infty,b]\midG \ell^\eta(\hat{p},0)\geq \ell^\eta(p,\Delta)\right\},\\
        \Gamma^1(p,\Delta)&:= \min\left\{\hat{p}\in (-\infty,b]\midG \ell^\eta(\hat{p},1)\geq \ell^\eta(p,\Delta)\right\}.\label{Main:36}
\end{align}

We prove~\eqref{Main:40} by clarifying the structure of
$T^{\ell^\eta}(p,\Delta)$ and $H^{\ell^\eta}(p,\Delta)$ from ~\eqref{Main:59} and~\eqref{Main:60} respectively.
As $p<b$ we obtain by Proposition~\ref{Main:128} that the minima in~\eqref{Main:69} and~\eqref{Main:36} are well-defined and $\Gamma^0(p,\Delta)>\Gamma^1(p,\Delta)$. Moreover,
    as $\hat{p}\mapsto \ell^\eta(\hat{p},\Delta)$ is strictly increasing on $(-\infty,b)$ (see again Proposition~\ref{Main:128}) we get
    \begin{align}\label{Main:130}
        \Gamma^0(p,0)=p\quad \text{ and }\quad  \Gamma^1(p,1)=p.
    \end{align}
	Now, using the definitions of $\Gamma^0$, $\Gamma^1$ we can see that \[
	    T^{\ell^\eta}(p,\Delta)= T^\eta(\Gamma^0(p,\Delta)-p,\Gamma^1(p,\Delta)-p,\Delta),
	 \]
	and
    \begin{align}
        H^{\ell^\eta}(p,\Delta)&=\{|\Delta \tilde{P}_{T^{\ell^\eta}(p,\Delta)}|\geq \eta\},
    \end{align}
     where $T^\eta$,  $H^{\ell^\eta}$ are defined in~\eqref{Main:62} and~\eqref{Main:60} respectively.
    Hence, plugging those observations into~\eqref{Main:142} proves~\eqref{Main:40}.
    
     Equation~\eqref{Main:40} implies immediately  that
    \begin{align}\label{Main:127}
        \ell^\eta(p,\Delta)=\inf_{\gamma^0,\gamma^1\in (-\infty,b-p]}f^\eta_{\Delta}(\gamma^0,\gamma^1,p).
    \end{align}
    Hence, to finish the proof of Theorem~\ref{Main:41} we have to show that the infimum in~\eqref{Main:127} coincides with those in~\eqref{Main:57} for the respectively pertinent cases. This will be accomplished by characterizing further the optimal choices $\hat{\gamma}^0:=\Gamma^0(p,\Delta)-p$, $\hat{\gamma}^1:=\Gamma^1(p,\Delta)-p$.

	\emph{Let $\Delta = 1$.}  By~\eqref{Main:130} we get $\hat{\gamma}^1=0$ and hence we can reduce the infimum in~\eqref{Main:127} to an infimum over $f^\eta_{1}(\gamma^0,0,p)$. Moreover, by $\hat{\gamma}^0>\hat{\gamma}^1=0$ we can restrict to $\gamma^0>0$. Next we want to obtain the stated upper bound for the relevant $\gamma^0$. 
    Using $T^\eta(\gamma^0,0)\geq T^1$, with $T^1$ the first jump of $N$, we obtain for any $\hat{p}\leq b$ that
	\begin{align}
	    \ell^\eta(\hat{p},0)&\geq 
	    \frac{1-\EW\left[\mathrm{e}^{-rT^0}\right]}{
			\frac{\lambda}{r}\left(1-
			\EW\left[\mathrm{e}^{-rT^1}\right]\right)
			-\EW\left[\mathrm{e}^{-rT^1}
			\mathbb{1}_{\{|\Delta \tilde{P}_{T^1}|\geq \eta\}}\right]}\left(\hat{p}-b\right)\\
			&=
			\frac{\left(\lambda +r\right)\left(1-\EW\left[\mathrm{e}^{-rT^0}\right]\right)}{\lambda p(\eta)}(\hat{p}-b). \label{Main:7}
	\end{align}
	On the other hand, using $T_n:=T^0+\frac1n$, $n\in \NZ$, with $T^0$ from~\eqref{Main:106}, we obtain due to~\eqref{Main:141} an upper bound on $\ell^\eta(p,1)$:
	\begin{align}\label{Main:15}
	    \ell^\eta(p,1)\leq \frac{1-	\EW\left[\mathrm{e}^{-rT^0}\right]}{\frac{\lambda}{r}	\left(1-\EW\left[\mathrm{e}^{-rT^0}\right]\right)+1}(p-b).
	\end{align}
	Observe that the right-hand side of~\eqref{Main:7} will be not smaller than the right-hand side of~\eqref{Main:15} if and only if
\begin{align}\label{Main:101}
	    \hat{p}-p\geq \frac{r+\lambda\left(1-\EW\left[\mathrm{e}^{-rT^0}\right]\right)-\frac{\lambda r}{\lambda+r}(1-p(\eta))}{r+\lambda\left(1-\EW\left[\mathrm{e}^{-rT^0}\right]\right)} (b-p).
	\end{align}
	For such $\hat{p} (\leq b)$ we will have $\ell^\eta(\hat{p},0)\geq \ell^\eta(p,1)$ and thus $\Gamma^0(p,1)\leq \hat{p}$, which is tantamount with $\hat{\gamma}^0$ being less than the right-hand side in~\eqref{Main:101}.
	This shows that we can restrict to $\gamma_0<B_0^{\eta}(b-p)$, which we wanted to show.
	
	\emph{Let $\Delta = 0$.}   By~\eqref{Main:130} we get $\hat{\gamma}^0=0$. Hence, we can assume in~\eqref{Main:127} that
	$\gamma_0=0$ and by $\hat{\gamma}^0>\hat{\gamma}^1=0$ we only have to consider $\gamma^1< 0$. 
	To obtain a lower bound, we can use the boundaries established in~\eqref{Main:7} and~\eqref{Main:15} to see that  $\ell^\eta(p,0)> \ell^\eta(\hat{p},1)$, if
	\begin{align}
	    \hat{p}-p< -\left(\frac{1-\lambda p(\eta)+\frac{(\lambda+r)\lambda}{r}\left(1-\EW\left[\mathrm{e}^{-rT^0}\right]\right)}{\lambda p(\eta)}\right)  (b-p).
	\end{align}
	This shows that $\gamma_1$ has to be chosen larger than $-B_1^{\eta}(b-p)$, which finishes our proof.
\end{proof}		

\subsection{Optimal optional controls}\label{sec:optional}
	Let us complete our analysis and determine the optimal control
	in the case of $p(\eta)=0$ where $\Lambda=\mathcal{O}(\FA)$ (see Lemma~\ref{Main:11} (ii)).
	This case is special since Theorem~\ref{Main:14}
	is not directly applicable as explained in the
	following proposition, which at the same time provides us with a remedy:
	\begin{Pro}\label{Main:42}
		Assume $p(\eta)=0$.
		\begin{enumerate}[label=(\roman*)]
		    \item The process $P$ is indistinguishable from ${^\Lambda P}$.
			\item For $\tilde{p}<m\frac{\lambda}{r}$ the process $P$ is \emph{not} 
						$\md R$-right-upper-semicontinuous
						in expectation at time $0$
						(see Assumption~\ref{Main:6} (iii)). In particular, $P$ does \emph{not} satisfy Assumption
						\ref{Main:6}.
			\item The process $\bar{P}$ given by 
						\begin{align}\label{Main:74}
							\bar{P}_t:=	\begin{cases}
														\frac{\lambda}{\lambda + r}(\tilde{P}_t+m)\mathrm{e}^{-rt}
																&		\text{for } \tilde{P}_t<m\frac{\lambda}{r}, \
																\Delta \tilde{P}_t=0,\\
														P_t		&		\text{else},
													\end{cases}\quad t\in [0,\infty),
						\end{align}
						with $\bar{P}_\infty:=0$,
						satisfies Assumption~\ref{Main:6} and, for any stopping time $S\in \mathcal{S}$,
						\begin{align}\label{Main:67}
						\bar{P}_S&=\esssup_{T\geq S, \md R([S,T))=0 \text{ a.s.}}
								\EW[P_T|\FA_S]\\&=P_S\vee \EW[P_{T_S^1}|\FA_S]=P_S\vee\frac{\lambda}{\lambda+r}(\tilde{P}_S+m)\mathrm{e}^{-rS},
						\end{align}
						where
		\begin{align}\label{Main:54}
		    T_S^1:=\inf\{t\geq S|N_t>N_S\}.
		\end{align} 
						In particular, $\bar{P}\geq P$, up to an evanescent set, and there exists a process $\bar{L}$ satisfying~\eqref{Main:33} and~\eqref{Main:13} with $P$ replaced by $\bar{P}$. In fact, $\bar{P}$ is the smallest optional process larger than $P$ satisfying	Assumption~\ref{Main:6}, i.e.~for any other optional process $\hat{P} \geq P$ satisfying Assumption~\ref{Main:6} we have that $\{\hat{P}<\bar{P}\}$ is evanescent.
		\end{enumerate}
	\end{Pro}
		
	\begin{proof}
	\emph{(i)} is an immediate consequence of $p(\eta)=0$ and  Corollary~\ref{Main:68} of the Meyer Section Theorem.
	
	\emph{(ii):} 
		With $T_1$ denoting the first jump time of the Poisson process 
		$N$, we have $\md R([0,T_1))=0$ and also
		\begin{align}
			\EW[P_0]=\tilde{p}<(\tilde{p}+m)\frac{\lambda}{\lambda+r}
			=\EW[P_{T_1}], 
		\end{align}
		where the inequality holds because  $\tilde{p}<m\frac{\lambda}{r}$ by assumption. 
		It follows that $P$ is not $\md R$-right-upper-semicontinuous
		in expectation at $0$. 
		
		\emph{(iii):}  Let us first argue~\eqref{Main:67} for any stopping time $S$. We have on $\{\Delta \tilde{P}_S=0\}\cap \{S<\infty\}$ that
		\begin{align}
			\bar{P}_S=\esssup_{T\geq S, \md R([S,T))=0 \text{ a.s.}}
								\EW[P_T|\FA_S]&=P_S\vee \EW[P_{T_S^1}|\FA_S]=P_S\vee\frac{\lambda}{\lambda+r}(\tilde{P}_S+m)\mathrm{e}^{-rS},
		\end{align}
		where $T_S^1$ is defined in~\eqref{Main:54}. As $P_S<\frac{\lambda}{\lambda+r}(\tilde{P}_S+m)\mathrm{e}^{-rS}$ is equivalent
		to $\tilde{P}_S<m\frac{\lambda}{r}$, we obtain 
		\begin{align}
		    \esssup_{T\geq S, \md R([S,T))=0 \text{ a.s.}}
								\EW[P_T|\FA_S]
					&=\begin{cases}
		                \frac{\lambda}{\lambda+r}(\tilde{P}_S+m)\mathrm{e}^{-rS},
		                &\text{ if } \tilde{P}_S<m\frac{\lambda}{r},\, \Delta \tilde{P}_S=0,\, S<\infty,\\
		                P_S,&\text{ else.}
		             \end{cases}
		\end{align}
		This proves~\eqref{Main:67}.
		We will show next that $\bar{P}$ satisfies Assumption~\ref{Main:6}.
		Part (i) of Assumption~\ref{Main:6} is clear. Part (ii)
		follows by Fatou's lemma and $\Delta \bar{P}_S=0$ a.s.\ at
		every predictable $\FA$-stopping times since for $\omega\in \{\Delta \tilde{P}_S=0\}$ such that $\tilde{P}_S(\omega)=\tilde{P}_{S-}(\omega)<m\frac{\lambda}{r}$ we will also have $\tilde{P}_{S_n}(\omega)<m\frac{\lambda}{r}$ for $n$ large enough. Hence, it remains to prove (iii)
		of Assumption~\ref{Main:6}. For that fix an $\FA$-stopping time $S$ 
		and a sequence $(S_n)_{n\in \NZ}$ of $\FA$-stopping times with $S_n\geq S$
		for all $n\in \NZ$ such that we have 
		$\lim_{n\rightarrow \infty} \md R([S,S_n))=0$ almost surely. Then for almost every $\omega \in \Omega$, we will have
		$S_n(\omega)\leq T_S^1(\omega)$ for sufficiently large $n$. Letting $\tilde{S}_n:=(S_n)_{\{S_n\leq T_S^1\}}$, we thus obtain by
		definition of $\bar{P}$ and Fatou's lemma that
		\begin{align}
		    \limsup_{n\rightarrow \infty}
		    \EW[\bar{P}_{S_n}]&\leq \limsup_{n\rightarrow \infty}\EW\left[
		    \bar{P}_{\tilde{S}_n}\right]
		    +\EW\left[\limsup_{n\rightarrow \infty}
		    \bar{P}_{S_n}\mathbb{1}_{\{S_n> T_S^1\}}\right]\\
		    &\overset{\eqref{Main:67}}{=} \limsup_{n\rightarrow \infty}\EW\left[
		    \esssup_{T\geq \tilde{S}_n, \md R([\tilde{S}_n,T))=0 \text{ a.s.}}
								\EW[P_T|\FA_{\tilde{S}_n}]\right]+0\\
			&= \limsup_{n\rightarrow \infty}\EW\left[
		    \esssup_{T\geq \tilde{S}_n, \md R([S,T))=0 \text{ a.s.}}
								\EW[P_T|\FA_{\tilde{S}_n}]\right]\\
		    &=  \limsup_{n\rightarrow \infty}\EW\left[
		    \esssup_{T\geq \tilde{S}_n, \md R([S,T))=0 \text{ a.s.}}
								\EW[P_T|\FA_S]\right]\\
			&\leq \EW\left[
		    \esssup_{T\geq S, \md R([S,T))=0 \text{ a.s.}}
								\EW[P_T|\FA_S]\right]\overset{\eqref{Main:67}}{=} \EW[\bar{P}_{S}].
		\end{align}
		Here, we have used in the second equality that on $\{S_n\leq T_S^1\}$ also $\md R([S,S_n))=0$; in the third equality we used that the essential supremum is upwards directed and so dominated convergence allows us to interchange the $\esssup$ and the $\FA_S$-conditional expectation.
		Hence, Assumption~\ref{Main:6} (iii) is satisfied. It now remains to
		show that $\bar{P}$ is the smallest optional process larger than $P$ satisfying 
		Assumption~\ref{Main:6}. For that assume there exists another  
		process $\hat{P}$ with $\hat{P}\geq P$ satisfying Assumption~\ref{Main:6}.
		By the Meyer Section Theorem it is enough to prove $\bar{P}_S\leq \hat{P}_S$ at every $\FA$-stopping time $S$. On the complement of  $A:=\{\tilde{P}_S<m\frac{\lambda}{r}\}\cap\{\Delta \tilde{P}_S=0\}\cap\{S<\infty\}=\{P_S<\bar{P}_S\}$ we have by definition of $\bar{P}$ that $\bar{P}_S=P_S\leq \hat{P}_S$. Therefore, let us 
		focus on $A$ and assume by way of contradiction $\WM(A)>0$. Then we can define
		the (constant) sequence $S_n:=(T_S^1)_A$, $n\in \NZ$, satisfying $S_n\geq S_A$ 
		and $\lim_{n\rightarrow\infty}\md R([S_A,S_n))=0$. Hence,
		as $\hat{P}$ satisfies Assumption~\ref{Main:6} (iii), 
		we get
		\begin{align}
		    \EW[\hat{P}_{(T_S^1)_A}]\leq \EW[\hat{P}_{S_A}]
		    <\EW[\bar{P}_{S_A}]\overset{\eqref{Main:67}}{=}
		    \EW[\EW[P_{(T_S^1)_A}|\FA_{S_A}]]
		    =\EW[P_{(T_S^1)_A}]\leq  \EW[\hat{P}_{(T_S^1)_A}],
		\end{align}
		which is the desired contradiction.
    \end{proof}

	Now we get the following analogue to
	Theorem~\ref{Main:41} for the optional case:

	\begin{Thm}[Optimal control in the optional case]
	\label{Main:39}
		In the case $p(\eta)=0$, the value of the optimization problem in~\eqref{Main:8} remains the same when we replace $P$ by $\bar{P}$:
			   \begin{align}\label{Main:50}
			    v:=\sup_{C\in \Cm}V(C)=\sup_{C\in \bar{\mathcal{C}}(\uP)}\bar{V}(C),
			   \end{align}
			   where $\bar{\mathcal{C}}(\uP)$ and $\bar{V}(C)$ denote the set of admissible controls and the value when $P$ is replaced by $\bar{P}$ from~\eqref{Main:74}. Moreover, an optimal optional control for both optimization problems is given by
		\[
			C_t^{\mathcal{O}}:=
			\uP \vee \sup_{v\in [0,t]}  L^{\mathcal{O}}_v,\quad t\in[0,\infty),
		\]
		with
		\begin{align}\label{Main:87}
				L_t^{\mathcal{O}} = \begin{cases}
					0,& \tilde{P}_t\geq b,\
										|\Delta\hspace{-0.1ex}
										\tilde{P}_t|>0,
										\vspace{2ex}\\
				\frac{r}{\lambda}	(\tilde{P}_t-b),
							& \tilde{P}_t\geq b,\
							\Delta\hspace{-0.1ex} \tilde{P}_t
							=0,\vspace{2ex}\\
			\frac{r}{\lambda +r}(b-\tilde{P}_t),& \tilde{P}_t< b,\
							|\Delta\hspace{-0.1ex}
							\tilde{P}_t|>0,\vspace{2ex}\\
				\inf\limits_{\gamma \in (-\infty,0)} 
				f(\gamma,\tilde{P}_t)<0,& m\frac{\lambda}{r}\leq \tilde{P}_t
						< b,\ \Delta\hspace{-0.1ex} \tilde{P}_t =  0,\\
				-\infty ,& \tilde{P}_t
						< m\frac{\lambda}{r},\ \Delta\hspace{-0.1ex} \tilde{P}_t =  0.
				\end{cases}
		\end{align}
		Here, $b$ is as in Theorem~\ref{Main:18} 
		and 
		the function $f:
		\RZ \times \RZ 
		\rightarrow \RZ$
		is given by
		\begin{align}\label{Main:86}
			f(\gamma,p):=\frac{\left(1-
			\EW\left[\mathrm{e}^{-rT(\gamma)}\right]\right)p
			-\EW\left[\mathrm{e}^{-rT(\gamma)} 
			\sum\limits_{k=1}^{N_{T(\gamma)}}Y_k\right]}{
			\frac{\lambda}{r}\left(1-
			\EW\left[\mathrm{e}^{-rT(\gamma)}\right]\right)
			-\EW\left[\mathrm{e}^{-rT(\gamma)}
			\right]},
		\end{align}
		with
		\begin{align}\label{Main:1}
			T(\gamma):=\inf\left\{t \in \{{ N}>0\} \right.
		&\left.	\midG |\Delta \tilde{P}_t|>0\text{ and } \tilde{P}_t-\Tilde{p}\geq \gamma \right\}.
		\end{align}
	\end{Thm}
\begin{Rem}
Notice, that the process $L^\mathcal{O}$ of~\eqref{Main:87} is a solution to~\eqref{Main:33} and~\eqref{Main:13} with $P$ replaced by $\bar{P}$, but it is \emph{not necessarily the maximal one} with~\eqref{Main:24}. This is without harm for our claim of optimality because for an application of Theorem~\ref{Main:14} we can use \emph{any} solution, not necessarily the maximal one.
\end{Rem}

\begin{proof}
   Analogously to Lemma~\ref{Main:11} one can show that we only have to establish the explicit form of $L^{\mathcal{O}}$ to prove that $C^{\mathcal{O}}$ attains the supremum over $\bar{V}$. Moreover, one can see that, with $L^{\mathcal{O}}$ of the given form,
    $C^{\mathcal{O}}$ only increases, when $\Delta \Tilde{P}>0$ or $\Tilde{P}\geq m\frac{\lambda}{r}$. This shows
$V(C^{\mathcal{O}})=\Bar{V}(C^{\mathcal{O}})$, which also establishes~\eqref{Main:50} since $P\leq\Bar{P}$. It thus suffices to establish the stated characterization of $L^{\mathcal{O}}$. For this denote by $\bar{L}^{\mathcal{O}}$ the maximal solution to~\eqref{Main:33} and~\eqref{Main:13} (for $\bar{P}$ instead of $P$) which then satisfies~\eqref{Main:24}, i.e.
\begin{align}\label{Main:76}
   \bar{L}^{\mathcal{O}}_S= \essinf_{S< T\in \st} \frac{\bar{P}_S-\EW[\bar{P}_T|\FA_S]}{\EW[\md R([S,T))|\FA_S]},
    \quad S\in \st.
\end{align}
Let $S\in \st$. We will argue now why for any $T\in \stm$ with $T> S$ we can replace $\EW[\bar{P}_T|\FA_S]$ by $\EW[P_T|\FA_S]$ in~\eqref{Main:76}. 
By definition of $\bar{P}$ we have, that $\EW[\bar{P}_T|\FA_S]\leq \EW[\bar{P}_{\hat{T}}|\FA_S]$ for $\hat{T}$ given by $\hat{T}:=T_T^1$ (see~\eqref{Main:54}) on $\{\bar{P}_T>P_T\}$ and $\hat{T}:=T$ else. We observe that $\md R([S,T))=\md R([S,\hat{T}))$ and $\bar{P}_{\hat{T}}=P_{\hat{T}}$. Hence, 
\begin{align}
   \frac{\bar{P}_S-\EW[\bar{P}_T|\FA_S]}{\EW[\md R([S,T))|\FA_S]}
   \geq \frac{\bar{P}_S-\EW[\bar{P}_{\hat{T}}|\FA_S]}{\EW[\md R([S,\hat{T}))|\FA_S]}
   =\frac{\bar{P}_S-\EW[P_{\hat{T}}|\FA_S]}{\EW[\md R([S,\hat{T}))|\FA_S]},
\end{align}
which shows our claim.
Now, we can establish analogously to Lemma~\ref{Main:11} (ii) that
\begin{align}\label{Main:90}
		    \bar{L}_t^{\mathcal{O}}=\bar{\ell}({\tilde{P}_t},\Delta  N_t),
		    \quad  t\in [0,\infty),
		\end{align}
	where $\bar{\ell}(p,\Delta):=\inf_{0<T\in \mathcal{S}}\bar{\ell}_T(p,\Delta)$
	and, for random times $T>0$,
    \begin{align}
        \bar{\ell}_T(p,\Delta):=
       \frac{\frac{\lambda}{\lambda+r}(p+m) \mathbb{1}_{\{p<m\frac{\lambda}{r},\Delta=0\}}+p\mathbb{1}_{\{p<m\frac{\lambda}{r},\Delta=0\}^c}-
    			\EW\left[\mathrm{e}^{-rT}\right]p
    			-\EW\left[\mathrm{e}^{-rT}	\sum_{k=1}^{N_T}Y_k \right]}{\EW\left[R_{T-}
    			\right]+\Delta}
    \end{align}
	  with the convention $\frac{\cdot}{0}=\infty$. For $p>m\frac{\lambda}{r}$ or $p<m\frac{\lambda}{r}$ and $\Delta = 1$ one can establish analogous results to Proposition~\ref{Main:128}, Lemma~\ref{Main:119} and deduce that  $\bar{L}^{\mathcal{O}}$ is equal to $L^{\mathcal{O}}$ given in~\eqref{Main:87} on the set $\{\bar{P}=P\}$. Note that 
	  for $\gamma_0\geq 0$ we have $T(\gamma^0,0)=T^0$,
	  which explains the explicit form of $L^{\mathcal{O}}$ in the case 
	  $\tilde{P}<b$ and $\Delta N>0$.
 On $\{\bar{P}_S>P_S\}=\{N_S=N_{S-}\}\cap\{P_S<b\}$
we have by~\eqref{Main:67} that $\bar{P}_S=\EW[P_{T_S^1}|\FA_S]$. As $\bar{L}^{\mathcal{O}}$ satisfies~\eqref{Main:13} we get with the definition
of $\bar{P}$, that 
\begin{align}
    \EW\left[\int_{[T_S^1,\infty)}
			\frac{\partial}{\partial c}\rho_t\left(
					\sup_{v\in [S,t]} \bar{L}^{\mathcal{O}}_v\right) \md R_t
					\midG \FA_{S}\right]&=\EW\left[\int_{[S,\infty)}
			\frac{\partial}{\partial c}\rho_t\left(
					\sup_{v\in [S,t]} \bar{L}^{\mathcal{O}}_v\right) \md R_t
					\midG \FA_{S}\right]\\	
	&=\bar{P}_S\overset{\eqref{Main:67}}{=}\EW[P_{T_S^1}|\FA_S]
	=\EW[\bar{P}_{T_S^1}|\FA_S]\\
	&=\EW\left[\int_{[T_S^1,\infty)}
			\frac{\partial}{\partial c}\rho_t\left(
					\sup_{v\in [T_S^1,t]} \bar{L}^{\mathcal{O}}_v\right) \md R_t
					\midG \FA_{S}\right],
\end{align}
where we have used that $\md R([S,T_S^1))=0$. Hence, as $\md R(\{T_S^1\})>0$ on $\{S<\infty\}$ we obtain
almost surely $\sup_{v\in [S,T_S^1]} \bar{L}^{\mathcal{O}}_v
    =\bar{L}^{\mathcal{O}}_{T_S^1}$ on $\{S<\infty\}$.
Therefore we can replace $\bar{L}^{\mathcal{O}}$ by $-\infty$ on $\{\bar{P}>P\}=\{N=N_{-}\}\cap\{P<b\}$ to obtain
another solution to~\eqref{Main:33}
and~\eqref{Main:13}, which is exactly $L^{\mathcal{O}}$ of~\eqref{Main:87}.
\end{proof}

In analogy to the predictable case,
the next corollary shows how the minimal storage level $L^{\Lambda^\eta}$ approaches
the minimal storage level under full immediate information $L^{\mathcal{O}}$ when the sensor's probability to fail tends to 0:
\begin{Cor}
	In the setting of Theorem~\ref{Main:41},
                consider a sequence $(\eta_n)_{n\in
                  \NZ}\subset [0,\infty]$ such that $\lim_n
                p(\eta_n)= 0$. Then the solution $L^{\Lambda^{\eta_n}}$, $n\in \NZ$, converges to $L^{\mathcal{O}}$ (see Theorem~\ref{Main:39}) for $n\rightarrow \infty$:
		\begin{align*}
			 	\lim_{n \rightarrow \infty} L^{\Lambda^{\eta_n}}_t(\omega)
		=L^{\mathcal{O}}_t(\omega),\quad  t \in
                        [0,\infty),\quad  \omega \in \Omega.
\end{align*}
\end{Cor}

\begin{proof}
    In the case $\tilde{P}_t(\omega)\geq b$ we have nothing to show. Assume now 
    $\tilde{P}_t(\omega)< b$ and $\Delta N_t(\omega)>0$. Using~\eqref{Main:142} and~\eqref{Main:90} shows $L_t^{\mathcal{O}}(\omega)\leq L_t^{\Lambda^\eta}(\omega)$ for all $\eta\in [0,\infty]$ with $\eta\leq |\Delta \tilde{P}_t(\omega)|$. Moreover, for $n\in \NZ$ set $T_n^k:=(T^0)_{\{|\Delta \tilde{P}_T^0|\geq \eta_n\}}\wedge (T^0+\frac1k)$, $k\in \NZ$, where $T^0$ is defined in~\eqref{Main:106}. Then we obtain with $\delta$ as in~\eqref{Main:106} that
    \begin{align}
        \lim_{n\rightarrow \infty}L_t^{\Lambda^{\eta_n}}(\omega)&\leq\lim_{n\rightarrow \infty} \lim_{k\rightarrow \infty}\ell_{T_n^k}(\tilde{P}_t(\omega),1)\\
       & =\lim_{n\rightarrow \infty}\frac{(1-\delta)(\tilde{P}_t(\omega)-b)}{\EW[R_{T^0-}]+1+\EW\left[\mathrm{e}^{-rT^0}\mathbb{1}_{\{|\Delta \tilde{P}_{T_0}|<\eta_n\}}\right]}=
        \frac{r}{r+\lambda}(\tilde{P}_t(\omega)-b)=L_t^{\mathcal{O}}(\omega).
    \end{align}
    Next,  let $m\frac{\lambda}{r}\leq \tilde{P}_t(\omega)< b$ and $\Delta N_t(\omega)=0$. First, $L_t^{\mathcal{O}}(\omega)\leq L_t^{\Lambda^\eta}(\omega)$ for all $\eta\in [0,\infty)$ follows as before and we claim that for any $\gamma<0$ we have for $f$ from~\eqref{Main:86} that
    \begin{align}\label{Main:94}
        f(\gamma,p)
        =\lim_{n\rightarrow \infty} f^{\eta_n}(0,\gamma,p).
    \end{align}
    Indeed,  one can see that for $\gamma<0$ we have $T(\gamma)\leq T^{\eta_n}(0,\gamma)$ for any $n\in \NZ$, where $T(\gamma)$ is given by~\eqref{Main:1}. Moreover, as $\{T(\gamma)<T^{\eta_n}(0,\gamma)\}\subset \{|\Delta \tilde{P}_{T(\gamma)}|<\eta_n\}$ we see that $\lim_{n\rightarrow \infty} \WM(T^{\eta_n}(0,\gamma)>T(\gamma))=0$. 
    The rest follows now by Lebesgue's theorem.
    As $B_1^{\eta_n}$ of~\eqref{Main:105} decreases to $-\infty$ for $p(\eta_n)$ going to zero, we can choose $\eta_n$ small enough to ensure that $\gamma>B_1^{\eta_n}(b-\tilde{P}_t(\omega)$. Therefore, we obtain with~\eqref{Main:94} that
    \begin{align*}
        f(\gamma,\tilde{P}_t(\omega))
        =\lim_{n\rightarrow \infty} f^{\eta_n}(0,\gamma,\tilde{P}_t(\omega))
        \geq \lim_{n\rightarrow \infty} L^{\eta_n}_t(\omega),
    \end{align*}
    which shows $L^{\mathcal{O}}_t=\inf\limits_{\gamma <0} 
				f(\gamma,\tilde{P}_t(\omega))\geq \lim_{n\rightarrow \infty} L^{\eta_n}_t(\omega)$.
				
Finally let $p:=\tilde{P}_t(\omega)<m\frac{\lambda}{r}$ and $\Delta N_t(\omega)=0$. One can see that $B_1^{\eta}$ from~\eqref{Main:105} converges to $-\infty$. Moreover, for $(p(\eta_n))_{n\in \NZ}$ going to zero and $(\gamma_n)_{n\in \NZ}$ going to $-\infty$ we get that $(T^{\eta_n}(0,\gamma_n))_{n\in \NZ}$ is converging to $T^1$, which denotes the first jump of $N$. Hence,
\[
    \lim_{n\rightarrow \infty}f^{\eta_n}(0,\gamma_n,p)
    =\frac{\left(1-
			\EW\left[\mathrm{e}^{-rT^1}\right]\right)p
			-\EW\left[\mathrm{e}^{-rT^1} 
			Y_1\right]}{
			\frac{\lambda}{r}\left(1-
			\EW\left[\mathrm{e}^{-rT^1}\right]\right)
			-\EW\left[\mathrm{e}^{-rT^1}\right] }
	=\frac{\frac{r}{\lambda+r}(p-m\frac{\lambda}{r})}{0}=-\infty,
\]
where the last equities are meant as a limiting procedure as the nominator converges to a strictly negative number and the denominator to zero.
Hence, we obtain our result as $\lim_{n\rightarrow \infty}L^{\Lambda^{\eta_n}}_t(\omega)=\lim_{n\rightarrow \infty}\inf_{\gamma\in(B_1^{\eta_n}(b-p),0)}f^{\eta_n}_0(0,\gamma,p)=-\infty$.

\end{proof}

	\subsection{Illustration}\label{sec:ill}

For a detailed illustration of Theorem~\ref{Main:18}, Theorem~\ref{Main:41} and Theorem~\ref{Main:39}, let us fix $\lambda=\frac12$, let $Y_1$ be the bi-modal distribution $\frac12(\mathcal{N}(-3,2)+ \mathcal{N}(6,2))$ and choose $\tilde{p}=-10$ and $r=1$, which gives $m=1.5$ and $b \approx 1.37642$. 

To obtain $L^{\Lambda^\eta}$, $\eta\in [0,\infty]$, we need to calculate numerically $f^\eta_\Delta$, $\Delta\in \{0,1\},$ from~\eqref{Main:16}, which can be done via Monte Carlo ($10^5$ samples). Figure~\ref{figure:6} plots a trajectory for the expected undiscounted reward process $\tilde{P}^\eta$ (gray) with its critical level $b$ (dotted gray) along with the optimal Meyer-measurable controls for $\eta=3$, $\eta=6$,  predictable ($\eta=\infty$) control, the optional ($\eta=0$) one, all starting in $\uP=-12$. 
		\begin{figure}
			\centering
				\includegraphics[width=0.32\textwidth]{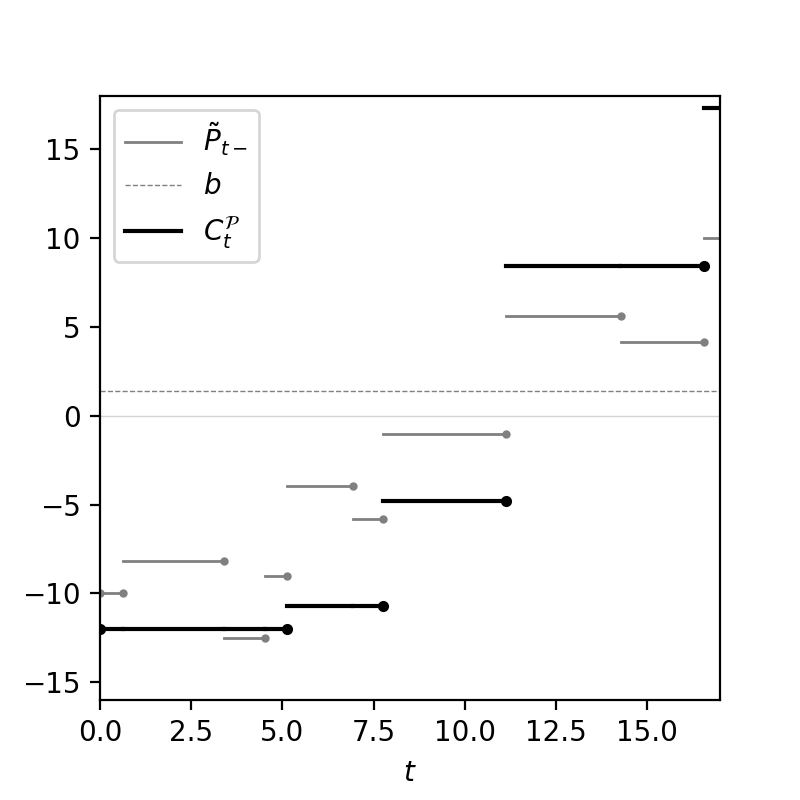}
			\includegraphics[width=0.32\textwidth]{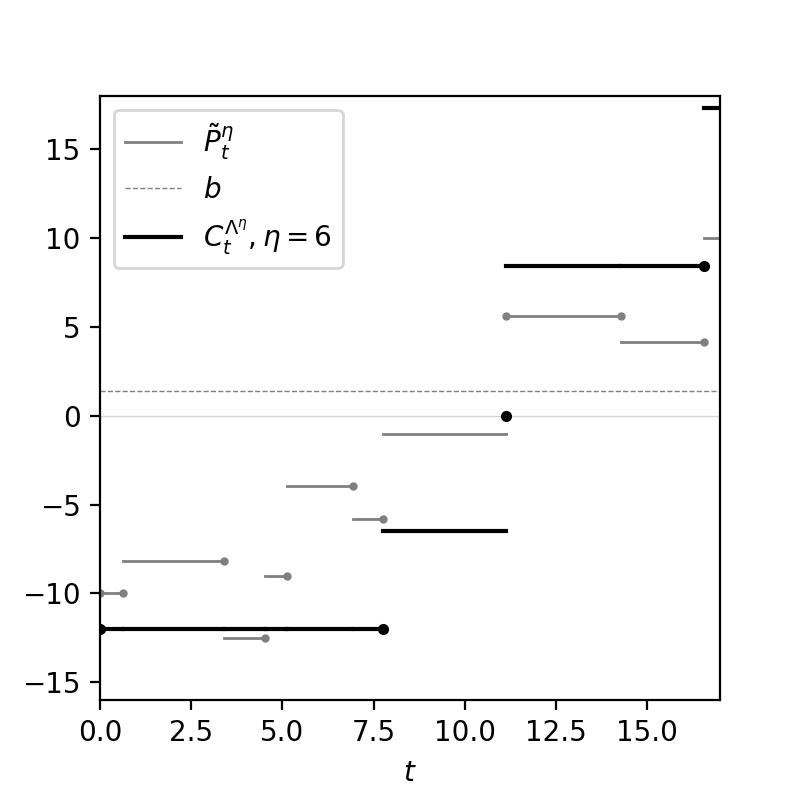}\\
				\includegraphics[width=0.32\textwidth]{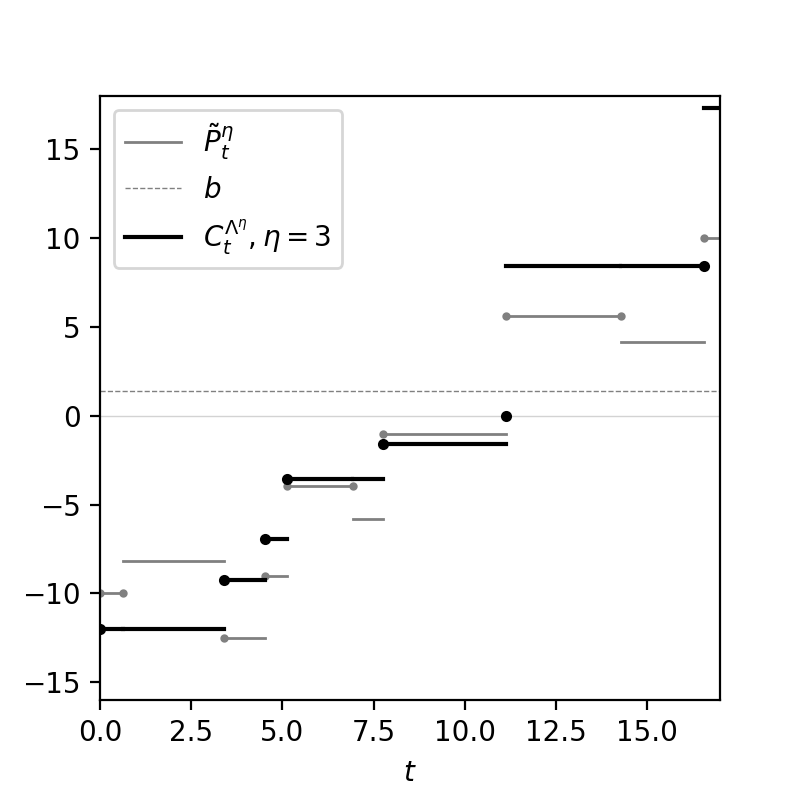}
				\includegraphics[width=0.32\textwidth]{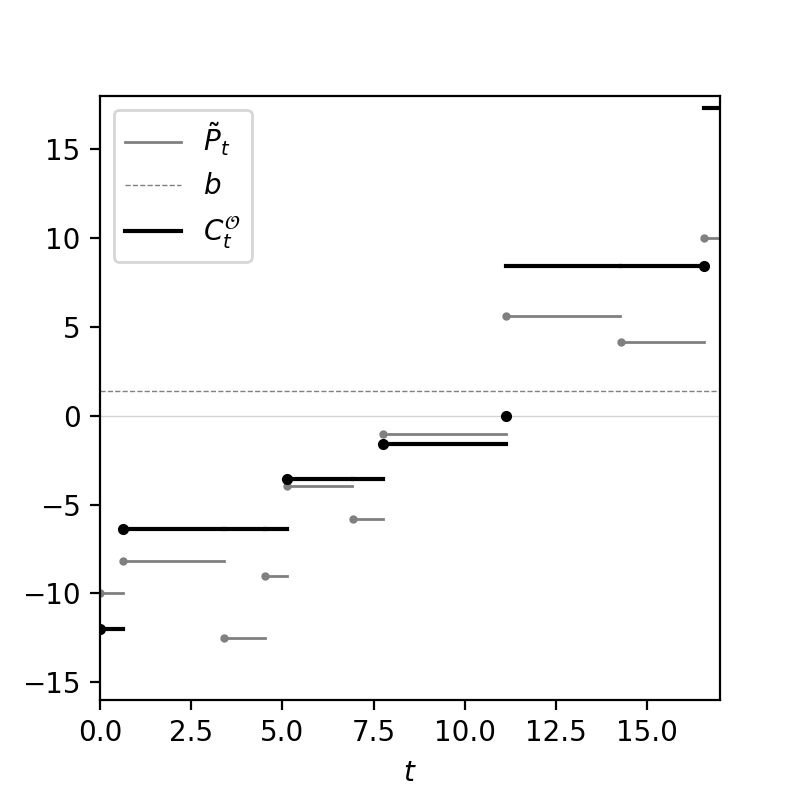}
			
			\caption{\small $\tilde{P}^\eta$ (gray), $b$ (dotted gray) and optimal controls for $\eta=0$ (optional, black), $\eta=3$, $\eta=6$ (black) and $\eta=\infty$ (predictable, black). The dots indicate the processes' value at their jump times.}
			\label{figure:6}
		\end{figure}
		Observe first of all that both of the Meyer controls fail to detect some jumps of $\tilde{P}$ (and $P$) immediately as they happen. The Meyer control with $\eta=6$ also does not adjust its level after the jump number four, despite the higher rewards
		obtainable then. So the controller here gambles on her ability to detect future risk shocks in time to benefit then from even higher rewards and the risk reduction. In fact, after the sixth jump of the reward process in this scenario, the accumulated value of the undetected jumps of $\tilde{P}$ is finally high enough to make her adjust her position. The predictable controller with no warnings about jumps can only adopt her position after the reward process has changed. This leads to a left-continuous optimal control in the predictable case. Moreover, in this case the position is increased whenever $\tilde{P}$ reaches a new all-time high, the only exception being the first jump at time $t_1\approx 0.5$, because the endogenously given starting position $c_0=-12$ is higher than the minimal position a predictable trader would tolerate, which is approximately $L_{t_1}^{\mathcal{P}}\approx  -19.12$. By contrast, the optional controller with perfect sensor can always intervene when rewards increase and she chooses to do so before the risk clock rings whenever the reward process $\tilde{P}$ is below the critical threshold $b$; with rewards beyond $b$, it is optimal for her and all the other controllers to only react to jumps after they have happened.
		One point of special interest is the moment $T_b$ where $\tilde{P}^\eta$ passes the critical value $b$ for the first time. The optional controller intervenes proactively to eliminate all risk ($C_{T_b}=0$) and is reacting once more after the jump to profit from the newly available high rewards.
	    Note that also the Meyer controllers with $\eta\in\{3,6\}$ act here in a similar way because the jump in this scenario happens to be larger than their respective detection thresholds and so the controllers become aware of this jump in the moment when it occurs. Hence, one can see that	 neither the optimal control in the Meyer case nor the optimal control in the
		optional case is in general left- or right-continuous; they are both just l\`adl\`ag, which illustrates the necessity of the general framework chosen in Section~\ref{sec:1.2}. Let us observe also that at the time of the fourth jump, both the predictable and the optional controller intervene while one of the Meyer controllers ($\eta=6$) abstains; hence, the Meyer controls cannot be viewed as simple interpolations between those two extreme cases. 
		
       We can use Monte Carlo simulation also to compute the value $v(\eta)$ of an optimal control depending on $\eta$ for the fixed initial values $p$ and $c_0$ set above, see Figure~\ref{figure:8}.
       Obviously, $v(\eta)$ is decreasing in the detection threshold $\eta$, its maximum $v(0)\approx -22$ corresponding to the optional case with perfect sensor considerably exceeding its minimum $v(\infty)\approx -33$ without sensor. In between these extremes, we can see the value function to be concave for ``small'' and convex for ``large'' values of $\eta$. The switching point between these regimes is around $\eta\approx 6.5$ which marks the detection threshold where small improvements of the sensor will be most effective. By contrast, the same small improvements will have very little effect when $\eta$ is small (as most jumps will be detected anyhow) and $\eta$ is large (when only few jumps will be large enough to be detected). 
        \begin{figure}[h]
			\centering
			\includegraphics[width=0.6\textwidth]{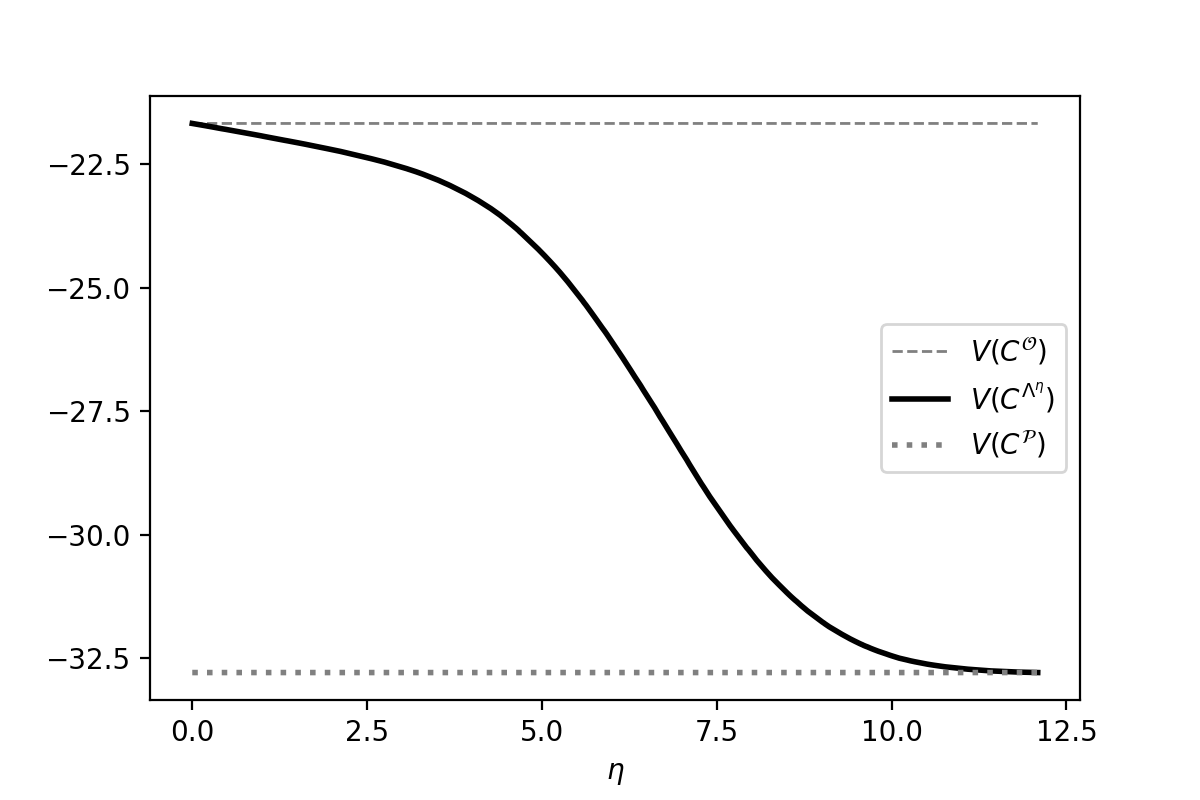}
			
			\caption{\small Value of an optimal control for $\eta$ between $0$ and 12 in our numerical example.}
			\label{figure:8}
		\end{figure}

\begin{appendix}
\section{Properties of the integral for l\`adl\`ag integrators}
\label{app:integration}

We will prove some results for the  
$\mdS$-integral defined in~\eqref{Main:4}, which
are well known for classical integrals and still
valid for such $\mdS$-integrals. But before doing this we will start with
a motivation for our definition of this integral and we close this
chapter
with a comparison to similar
definitions in the literature.

For the rest of this section fix a filtered probability space $(\Omega,\mathbb{F}, 
\FA:=(\FA_t)_{t\geq 0},\WM)$ with $\mathbb{F}:=\FA_\infty:=\bigvee_{t} \FA_t$ and $\FA$ fulfills the usual
conditions. 
Furthermore, we consider a $\WM$-complete
Meyer-$\sigma$-field $\Lambda$ (see Definition and Theorem~\ref{Main:107}), 
such that $\mathcal{P}(\mathcal{F})\subset \Lambda\subset
\mathcal{O}(\mathcal{F})$, where $\mathcal{P}(\FA)$ and
$\mathcal{O}(\FA)$ denote, respectively, the predictable and the
optional $\sigma$-field associated with $\FA$.

\paragraph{Motivation of $\mdS$-integrals.}

	Apart from the approximation argument for the $\mdS$-integral in the proof of Proposition~\ref{Main:307}, one can motivate our $\mdS$-integral
	by showing that optimizing our $\mdS$-integrals is equal to an optimal stopping problem over
	divided stopping times. This can be viewed as a version of a result of \cite{Bis79} for divided stopping times instead of ordinary stopping times. Divided stopping times have been introduced
	in \cite{EK81}:
	\begin{Def}[\citing{EK81}{Definition 2.37}{136-137}]
	A given quadruple $\tau:=(T,H^-,H,H^+)$ is called a \emph{divided
	stopping time}, if $T$ is an $\mathcal{F}$-stopping time
	and $H^-,H,H^+$ build a partition 
	of $\Omega$ such that
	\begin{enumerate}[label=(\roman*)]
		\item $H^-\in \mathcal{F}_{T-}$ and $H^-\cap \{T=0\}=\emptyset$,
		\item $H\in \mathcal{F}^\Lambda_T$,
		\item $H^+ \in \mathcal{F}_{T}^\Lambda$ and $H^+\cap 
		\{T=\infty\}=\emptyset$,
		\item $T_{H^-}$ is an $\mathcal{F}$-predictable 
		stopping time,
		\item $T_{H}$ is a $\Lambda$-stopping time.
	\end{enumerate}
	The set of all divided stopping times $\tau$ will be denoted 
	as $\stmd$.			
	For a $\Lambda$-measurable positive process $Z$
	we define the value attained 
	at a divided stopping time $\tau=(T,H^-,H,H^+)$ as
	$Z_\tau:=\lsl{Z}_{T}\mathbb{1}_{H^-}+Z_T\mathbb{1}_{H}
		+\lsr{Z}{T}\mathbb{1}_{H^+}$, where $(\cdot)^\ast$ is defined in~\eqref{Main:5} and $^\ast(\cdot)$ is its analogue where we take the $\limsup$ from the left
		(see for example \BBC).
		
\end{Def}
	One key advantage of divided stopping times over ordinary stopping times is
	that under fairly mild conditions an optimal
	divided stopping time exists; see \citing{EK81}{Theorem 2.39}{138}. The following result shows that $\mdS$-integrals yield a convex relaxation of optimal stopping over divided stopping times that is analogous to \citing{Bis79}{Equation (2.1)}{938}:
	
	\begin{Thm}
	Denote by $\mathcal{C}$ the set of increasing, 
	$\Lambda$-measurable processes $C$ satisfying $C_{0-}:=0$
	and $C_t\leq 1$, $t\in [0,\infty)$. Furthermore, let $Z$ be a $\Lambda$-measurable nonnegative process of class($D^\Lambda$) with $Z_\infty=0$.
	Then
	\begin{align}\label{app:3}
		\sup_{\tau \in \stmd}
		\EW[Z_\tau]
		=\sup_{C \in \mathcal{C}} \EW\left[\int_{[0,\infty)}
		Z_s\mdS C_s\right],
	\end{align}
	and  there exists a divided stopping time 
	$\hat{\tau}:=(\hat{T},\hat{H}^-,\hat{H},\hat{H}^+)$ attaining the value on the left hand side of~\eqref{app:3}.
	
	If $Z$ is additionally left-upper-semicontinuous
		in expectation at every predictable $\FA$-stopping time
		(see Assumption~\ref{Main:6} (ii)) then
		\[
		    \hat{C}:=
	\mathbb{1}_{\stsetRO{\hat{\tau}}{\infty}}:= \begin{cases}
	                                                \mathbb{1}_{\stsetRO{\hat{T}}{\infty}}&\text{ on } \hat{H}^-\cup \hat{H},\\
	                                                \mathbb{1}_{\stsetO{\hat{T}}{\infty}}&\text{ on } (\hat{H}^-\cup \hat{H})^c,
	                                            \end{cases}
	\]
	will solve the optimal
	control problem, i.e
	\[
	    \sup_{C \in \mathcal{C}} \EW\left[\int_{[0,\infty)}
		Z_s\mdS C_s\right]=\EW\left[\int_{[0,\infty)}
		Z_s\mdS \hat{C}_s\right].
	\]
	\end{Thm}
	
	\begin{proof}
	By \citing{EK81}{Theorem 2.39}{138}, there exists an optimal
	divided stopping time $\hat{\tau}=
	(\hat{T},\hat{H}^{-},\hat{H},\hat{H}^{+})$ for the left-hand side of~\eqref{app:3}. By definition
	$S:=\hat{T}_{\hat{H}^-}$ is an $\FA$-predictable stopping time. Note that $S>0$ as $\{\hat{T}=0\}\cap \hat{H}^-=\emptyset$ by definition of a divided stopping time.  Hence,  by \BBG,\ there exists a sequence $(S_n)_{n\in \NZ}$ of $\Lambda$-stopping times 
	such that $S_n<S$ for all $n \in \NZ$, $\lim_{n\rightarrow \infty}S_n=S$ and $\lim_{n\rightarrow \infty} Z_{S_n}=
	{^\ast Z}_S$. 
	Now define, for $n\in\NZ$, 
	\begin{align}
	    \hat{C}^n:=
			\bigg(&{^\Lambda \left(\mathbb{1}_{\hat{H}^-}\right)_{S_n}}
			 \mathbb{1}_{\stsetRO{S_n}{\hat{T}_{\hat{H}^-}}}
			+\mathbb{1}_{\stsetRO{\hat{T}_{\hat{H}^-}}{\infty}} \\
			&+{^\Lambda \left(\mathbb{1}_{(\hat{H}^-)^c}\right)_{\hat{T}_{\hat{H}}\wedge S_n}}
			 \mathbb{1}_{\stsetRO{\hat{T}_{\hat{H}}}{\infty}}
			+\mathbb{1}_{(\hat{H}^-)^c}
			 \mathbb{1}_{\stsetO{\hat{T}_{\hat{H}^+}}{\infty}}\bigg)\wedge 1.
	\end{align}
	Clearly, $\hat{C}^n$ is $\Lambda$-measurable taking values in $[0,1]$; moreover, it is increasing which we verify separately on each part of the partition $\Omega=\hat{H}^-\cup\hat{H}\cup\hat{H}^+$: On the set $\hat{H}^-$, we have
	    \begin{align}
	    \hat{C}^n
			={^\Lambda \left(\mathbb{1}_{\hat{H}^-}\right)_{ S_n}}
			 \mathbb{1}_{\stsetRO{S_n}{\hat{T}}}
			+\mathbb{1}_{\stsetRO{\hat{T}}{\infty}}\leq 1; \label{Main:95}
	\end{align}
     on the set $\hat{H}$, we have
	    \begin{align}
	    \hat{C}^n
	    &={^\Lambda \left(\mathbb{1}_{\hat{H}^-}\right)_{S_n}}
			 \mathbb{1}_{\stsetRO{S_n}{\infty}}
			+{^\Lambda \left(\mathbb{1}_{(\hat{H}^-)^c}\right)_{\hat{T}\wedge S_n}}
			 \mathbb{1}_{\stsetRO{\hat{T}}{\infty}}\\
			&
			=
			 {^\Lambda \left(\mathbb{1}_{\hat{H}^-}\right)_{S_n}}\mathbb{1}_{\stsetRO{S_n}{\hat{T}}}
			+{^\Lambda \left(\mathbb{1}_{(\hat{H}^-)^c}\right)_{\hat{T}}}\mathbb{1}_{\stsetRO{\hat{T}}{S_n}}
			+\mathbb{1}_{\stsetRO{S_n\vee\hat{T}}{\infty}}
			\leq 1;\label{Main:98}
	\end{align}
    finally, on the set $\hat{H}^+$, we have
	\begin{align}
	    \hat{C}^n
	    &=\left({^\Lambda \left(\mathbb{1}_{\hat{H}^-}\right)_{S_n}}
			 \mathbb{1}_{\stsetRO{S_n}{\infty}}
			+\mathbb{1}_{(\hat{H}^-)^c}\label{Main:99}
			 \mathbb{1}_{\stsetO{\hat{T}}{\infty}}\right)\wedge 1\\
	&=
			 \mathbb{1}_{\stsetRO{S_n}{\infty}\cap \stsetO{\hat{T}}{\infty}}
			 +{^\Lambda \left(\mathbb{1}_{\hat{H}^-}\right)_{S_n}}\nonumber
			 \mathbb{1}_{\stsetC{S_n}{\hat{T}}}
			+
			 \mathbb{1}_{\stsetO{\hat{T}}{S_n}}.
	\end{align}
	It remains to show that 
	\begin{align}\label{Main:224}
		\EW[Z_{\hat{\tau}}]=\lim_{n\rightarrow \infty} \EW\left[\int_{[0,\infty)}
	Z_s\mdS \hat{C}_s^n\right],
	\end{align}
	to conclude ``$\leq$'' in~\eqref{app:3}.
				
To this end, we first calculate with~\eqref{Main:95}, \eqref{Main:98} and~\eqref{Main:99} that
		\begin{align*}
		     \int_{[0,\infty)}
		Z_s\mdS \hat{C}_s^n
		=\ &\mathbb{1}_{\{\hat{T}=\infty\}}{^\Lambda \left(\mathbb{1}_{\hat{H}^-}\right)_{S_n}Z_{S_n}}\\
		&+\mathbb{1}_{\{\hat{T}<\infty\}}\mathbb{1}_{\hat{H}^-}\left({^\Lambda \left(\mathbb{1}_{\hat{H}^-}\right)_{S_n}Z_{S_n}}
		+\left(1-{^\Lambda \left(\mathbb{1}_{\hat{H}^-}\right)_{S_n}}\right)Z_{\hat{T}}\right)\\
		&+\mathbb{1}_{\{\hat{T}<\infty\}}\mathbb{1}_{\hat{H}\cap\{S_n\leq \hat{T}\}}\left({^\Lambda \left(\mathbb{1}_{\hat{H}^-}\right)_{S_n}Z_{S_n}}+\left(1-{^\Lambda \left(\mathbb{1}_{\hat{H}^-}\right)_{S_n}}\right)Z_{\hat{T}}\right)\\
		&+\mathbb{1}_{\{\hat{T}<\infty\}}\mathbb{1}_{\hat{H}\cap\{S_n> \hat{T}\}}\left({^\Lambda \left(\mathbb{1}_{(\hat{H}^-)^c}\right)_{\hat{T}}Z_{\hat{T}}}+\left(1-{^\Lambda \left(\mathbb{1}_{(\hat{H}^-)^c}\right)_{\hat{T}}}\right)Z_{S_n}\right)\\
		&+\mathbb{1}_{\hat{H}^+\cap\{S_n\leq \hat{T}\}}\left({^\Lambda \left(\mathbb{1}_{\hat{H}^-}\right)_{S_n}Z_{S_n}}+\left(1-{^\Lambda \left(\mathbb{1}_{\hat{H}^-}\right)_{S_n}}\right)\lsr{Z}{\hat{T}}\right)\\
		&+\mathbb{1}_{\hat{H}^+\cap\{S_n> \hat{T}\}}\lsr{Z}{\hat{T}},
		\end{align*}
		where we have used that $\hat{H}^+\cap \{\hat{T}=\infty\}=\emptyset$.
		Hence, from
		\begin{align*}
			\lim_{n\rightarrow \infty} {^\Lambda \left(\mathbb{1}_{\hat{H}^-}\right)_{S_n}}&=\EW\left[\mathbb{1}_{\hat{H}^-}| \FA_{S-}\right]
			=\mathbb{1}_{\hat{H}^-},\\
			\lim_{n\rightarrow \infty}\mathbb{1}_{\{S_n\leq \hat{T}\}}&=\mathbb{1}_{\hat{H}^-\cup\{\hat{T}=\infty\}},\quad
			\lim_{n\rightarrow \infty}\mathbb{1}_{\{S_n> \hat{T}\}}=\mathbb{1}_{(\hat{H}^-)^c\cap \{\hat{T}<\infty\}},
		\end{align*}
		we obtain, since
		$\hat{H}^-,\hat{H},\hat{H}^+$ form a partition of
		$\Omega$, that
		\begin{align}
		    \lim_{n\rightarrow \infty} \EW\left[\int_{[0,\infty)}
			Z_s\mdS \hat{C}_s^n\right]
			=&\ \EW\Big[ \mathbb{1}_{\hat{H}^-}\lsl{Z}_{\hat{T}}+\mathbb{1}_{\hat{H}^+}\lsr{Z}{\hat{T}}+\mathbb{1}_{\hat{H}}Z_{\hat{T}}
			\Big],
		\end{align}
		where we used $(\hat{H}^-)^c=\in \aFA_{\hat{T}_H}$ and $Z_\infty=0$. This is our assertion~\eqref{Main:224}.
		
		We show next  ``$\geq$'' in~\eqref{app:3}. 
		Denote by $\bar{Z}$ El Karoui's 
		$\Lambda$-Snell envelope of $Z$ (see \BBH), which is l\`adl\`ag (see \BBI), 
		of class$(D^\Lambda)$ (see \BBJ) and by \BBK,\ we have
		\begin{align}\label{app:2}
		    \EW[\bar{Z}_0]=\sup_{\tau \in \stmd}
			\EW[Z_\tau].
		\end{align}
		Furthermore, we can decompose $\bar{Z}=\bar{M}-\bar{A}$ into	a $\Lambda$-martingale $\bar{M}$ of class($D^\Lambda$) 
		and  an increasing $\Lambda$-measurable process $\bar{A}$
		with $\bar{A}_{0}=0$ and $\EW[\bar{A}_{\infty}]<\infty$ (see \BBL). By $\bar{Z}_\infty=Z_\infty=0$, or equivalently, by
		$\bar{A}_\infty=\bar{M}_\infty$ we obtain for
		any $\Lambda$-stopping time $T$ that
		$				    \left({^\Lambda \bar{A}_\infty}\right)_T
		    =\EW[\bar{A}_\infty|\aFA_T]
		    =\bar{M}_T$.
		Hence, for any $C\in \mathcal{C}$ we obtain by using $\bar{Z}=\left({^\Lambda \bar{A}_\infty}\right)-\bar{A}$ and Proposition~\ref{app:5} below that
		\begin{align}
		    \EW\left[\int_{[0,\infty)}
			Z_s\mdS C_s\right]&\leq \EW\left[\int_{[0,\infty)}
			\bar{Z}_s\mdS C_s\right]
			= \EW\left[\int_{[0,\infty)}
			\bar{A}_\infty-\bar{A}_s \mdS C_s\right]
			\leq \EW\left[\bar{A}_\infty\right]=
			\EW\left[\bar{Z}_0\right],
		\end{align}
		which shows with the help of~\eqref{app:2}
	that also ``$\geq$'' is satisfied in~\eqref{app:3}.
		
		If, additionally, $Z$ is left-upper-semicontinuous
		in expectation at every predictable $\FA$-stopping time,
		we obtain by \BBM,\ that $^\ast Z_S\leq {^\mathcal{P} Z_S}$ at any predictable
		$\FA$-stopping time $S$. Hence, we can assume without loss of generality that the optimal stopping time
		$\hat{\tau}$ is of the form $\hat{\tau}=(\hat{T},\emptyset,
		H,H^+)$. Then we have
		\[
			\EW[Z_{\hat{\tau}}]=\EW\left[\int_{[0,\infty)}
		Z_s\mdS \hat{C}_s\right],
		\]
		which shows the result by~\eqref{app:3}.
	\end{proof}

\paragraph{Two classical integration results for the $\mdS$-integral.}

Let us first verify that Fubini's theorem is still valid for
a specific class of integrands.

\begin{Pro}[Fubini's Theorem for $\mdS$-Integrals]\label{app:1}
	Let $A,B:[0,\infty)\rightarrow \RZ$ be two increasing functions 
	 with $B$ right-continuous
	and $A_{0-},B_{0-}\in \RZ$.
	Additionally, let $\phi:=\phi_{s,t}:[0,\infty)\times[0,\infty)\rightarrow \RZ$ be a measurable function, which 
	admits right limits in the first argument $s\in [0,\infty)$ for $t\in [0,\infty)$ fixed and which satisfies
	\begin{align}\label{main:302}
		\int_{[0,\infty)}\sup_{s\in [0,\infty)}|\phi_{s,t}|\md B_t<\infty
	\end{align}			
	
	\begin{enumerate}[label=(\roman*)]
		\item If $\phi\geq 0$ or $\int_{[0,\infty)}
					\int_{[0,s]} |\phi_{s,t}|
						\ \md B_t \mdS A_s<\infty$
			    or	
		$\int_{[0,\infty)}
					\int_{[t,\infty)} |\phi_{s,t}|
						\  \mdS A_s \md B_t<\infty,
		$	
			  then
				 \[
				 	\int_{[0,\infty)}
							\int_{[0,s]} \phi_{s,t}
								 \, \md B_t\,\mdS A_s=
								\int_{[0,\infty)}
							\int_{[t,\infty)}\phi_{s,t}\,
								\mdS A_s \md B_t. 
				\]		 
		\item If $\phi\geq 0$ or $\int_{[0,\infty)}
					\int_{[s,\infty)} |\phi_{s,t}|
						\ \md B_t \mdS A_s<\infty$
			  or	
			  $\int_{[0,\infty)}
					\int_{[0,t]} |\phi_{s,t}|
						\  \mdS A_s \md B_t<\infty,$
			  then 
			  \[
			 	\int_{[0,\infty)}
						\int_{[s,\infty)} \phi_{s,t}
							 \, \md B_t\,\mdS A_s=
							\int_{[0,\infty)}
						\int_{[0,t]}\phi_{s,t}\,
							\mdS A_s \md B_t.
			  \]	
	\end{enumerate}
\end{Pro}
\begin{proof}
	Observe that because of 	\eqref{main:302} we have for any $s\in [0,\infty)$ that
		\begin{align}\label{main:303}
		\int_{[0,s]} \phi_{s+,t}
					 \, \md B_t=\lim_{u\downarrow s}\int_{[0,u]} \phi_{u,t}
					 \, \md B_t\quad \text{ and }\quad
		\int_{(s,\infty)} \phi_{s+,t}
					 \, \md B_t=\lim_{u\downarrow s}\int_{[u,\infty)} \phi_{u,t}
					 \, \md B_t.
	\end{align}				
	Now, recalling the notation $\ldc{A}$, $\rdc{A}$ from~\eqref{Main:4}, we get by ~\eqref{main:303} result (i) from the standard Fubini Theorem:
	\begin{align*}
		\int_{[0,\infty)}
				\int_{[0,s]} \phi_{s,t}
					 \, \md B_t\,\mdS A_s
					 &=\int_{[0,\infty)}
				\int_{[0,s]} \phi_{s,t}
					 \, \md B_t\,\md \ldc{A}_s+
					 \int_{[0,\infty)}
				\int_{[0,s]} \phi_{s+,t}
					 \, \md B_t\,\md \rdc{A}_s\\
					 &=\int_{[0,\infty)}
				\int_{[t,\infty)} \phi_{s,t}
					\md \ldc{A}_s \, \md B_t+
					 \int_{[0,\infty)}
				\int_{[t,\infty)} \phi_{s+,t}
					 \md \rdc{A}_s \, \md B_t\\	
					&=\int_{[0,\infty)}
				\int_{[t,\infty)}\phi_{s,t}\,
					\mdS A_s \md B_t,					 
	\end{align*}
	where we have used the definition~\eqref{Main:4} in the last step.
	Analogously we obtain by~\eqref{main:303} result (ii).
	
\end{proof}

Next we show that we can replace a suitable process $\phi$
inside of the extended integral by 
the Meyer-projection of this process.

\begin{Pro}\label{app:5}
	Fix a $\Lambda$-measurable increasing process 
	$A:\Omega\times [0,\infty)\rightarrow \RZ$ 
	and an $\FA\otimes \mathcal{B}([0,\infty)$-measurable 
	process $\phi:\Omega\times [0,\infty)\rightarrow \RZ$, such
	that 
	\[
		\EW\left[\int_{[0,\infty)}
			|\phi_t|\ \mdS A_t\right]<\infty.
	\]
	Then we have that
	\[
		\EW\left[\int_{[0,\infty)}
			\phi_t\ \mdS A_t \right]
		\geq				\EW\left[\int_{[0,\infty)}
			{^\Lambda \phi_t}\ \mdS A_t \right]
	\]
	with equality if $\phi$ admits limits from the right or if $A$ is right-continuous.
\end{Pro}

\begin{proof}
	By the integrability assumption we can assume without loss of generality 
	 that the process $\phi$	is bounded. Otherwise we could consider $\phi\wedge M$
	and use dominated convergence for Lebesgue integrals 
	to let $M$ tend to $\infty$ afterwards.		

	Now we use the idea of the proof of \citing{JS03}{Lemma 3.12}{29}, and introduce the inverses of $\ldc{A}$ and
	$\rdc{A}$ by 
	\[
		\ldc{T}_t:=\inf\{s\in [0,\infty)|\ldc{A}_s\geq t\},\quad \rdc{T}_t:=\inf\{s\in [0,\infty)|\rdc{A}_s\geq t\},\quad t\in [0,\infty).
	\]
	As $A$ is $\Lambda$-measurable, $\ldc{T}_t$ and $\rdc{T}_t$ define $\FA$-stopping times. Additionally 
	\[
		\ldc{A}_{\ldc{T}_t}\geq t\quad \text{ on } \quad \{\ldc{T}_t<\infty\}.
	\] 
	Hence, we get by 
	\citing{EL80}{Corollary 2}{504}, that $\ldc{T}$ is even
	a $\Lambda$-stopping time. Next, we have by definition
	\[
		\EW\left[\int_{[0,\infty)}
			\phi_t\ \mdS A_t \right]
		=	\EW\left[\int_{[0,\infty)}
			\phi_t\ \md \ldc{A}_t \right]+\EW\left[\int_{[0,\infty)}
			\phi_t^\ast\ \md \rdc{A}_t \right]
	\]
	and we will show separately
	\begin{align}
		\EW\left[\int_{[0,\infty)}
			\phi_t\ \md \ldc{A}_t \right]&=\EW\left[\int_{[0,\infty)}
			{^\Lambda \phi}_t\ \md \ldc{A}_t \right],\label{app:6}\\
		\EW\left[\int_{[0,\infty)}
			\phi_t^\ast\ \md \rdc{A}_t \right]&\geq \EW\left[\int_{[0,\infty)}
			\left(^\Lambda \phi\right)_t^\ast\ \md \rdc{A}_t \right],\label{app:7}
	\end{align}
	where we have equality in the latter inequality if $\phi$ admits limits from the right.
	
	For~\eqref{app:6} we follow the time-change argument of \citing{JS03}{Lemma 3.12}{29}, and use Fubini's theorem and the definition of the $\Lambda$-projection (see Definition and Theorem~\ref{Main:154}) to obtain			
	\begin{align*}
		\EW\left[\int_{[0,\infty)}
			\phi_t\ \md \ldc{A}_t \right]&=
			\int_{[0,\infty)}\EW\left[
			\phi_{\ldc{T}_t}\mathbb{1}_{\{\ldc{T}_t<\infty\}}\ \right]\md t\\
			&=
			\int_{[0,\infty)}\EW\left[
			{^\Lambda \phi}_{\ldc{T}_t}\mathbb{1}_{\{\ldc{T}_t<\infty\}}\ \right]\md t  
			=\EW\left[\int_{[0,\infty)}
			{^\Lambda \phi}_t\ \md \ldc{A}_t \right].
	\end{align*}				
 Similarly, we get~\eqref{app:7} by
	\begin{align*}
		\EW\left[\int_{[0,\infty)}
			\phi_t^\ast\ \md \rdc{A}_t \right]&=
			\int_{[0,\infty)}\EW\left[
			{^\mathcal{O} (\phi^\ast)}_{\rdc{T}_t}
			\mathbb{1}_{\{\rdc{T}_t<\infty\}}\ \right]\md t 	\\
		&\geq \int_{[0,\infty)}\EW\left[
			\left({^\Lambda \phi}\right)_{\rdc{T}_t}^\ast 
			\mathbb{1}_{\{\rdc{T}_t<\infty\}}\ \right]\md t 	=
		\EW\left[\int_{[0,\infty)}
			\left(^\Lambda \phi\right)_t^\ast\ \md \rdc{A}_t \right].
	\end{align*}
	Here, the inequality is due to
	 \BBF, which would give us 
	equality if $\phi$ admits limits from the right.
\end{proof}

\begin{Rem}
Proposition~\ref{app:1} and Proposition~\ref{app:5} remain valid with reversed inequality 
if we consider $\mdST$ instead of $\mdS$ (see the end of Definition~\ref{Main:139}).
\end{Rem}		
		
\paragraph{Relation of our $\mdS$-integral to similar definitions
in the literature.} Let us compare our definition of an integral
over l\`ad\`ag integrators to other definitions proposed in 
\cite{CS14} (following \cite{CS06}), \cite{GLR12} and  \cite{EL80}. 

\emph{Comparison to \cite{CS14} and \cite{GLR12}:} In \cite{CS14}, p.4-5,
the integral with respect to a l\`adl\`ag process
$\phi$ and a process $A$ of bounded variation with $A_0=A_{0-}$
is defined as
\[
	\int_{[0,t]} \phi_v\ {^{\text{\tiny CS}\!} \md} A_v:=
	\int_{[0,t]}\phi_v\md A^c_v
						+\sum_{v\in (0,t]}\phi_{v-}\Delta^- A_v
						+\sum_{v\in [0,t)} \phi_{v} \Delta^+ A_v,\
						t\in [0,\infty).
\]
with $\Delta^- A_v:=A_v-A_{v-}$ and  $\Delta^+ A_v:=A_{v+}-A_{v}$. Additionally they also define  the integral with integrator
$\phi$ in such a way that an integration by parts formula is satisfied. 
One can then calculate that
\[
	\int_{[0,t]} \phi_v\ {^{\text{\tiny CS}\!} \md} A_v=
	\int_{[0,t]} \phi_u \md A_{u-}-
	\sum_{0< u\leq t} \Delta^- \phi_u\Delta^-A_u
\]
and the right-hand side is actually the definition used 
in \citing{GLR12}{Definition A.6}{766}, apart from the fact 
that \cite{GLR12} additionally assume that $\phi$ has to be
right-continuous. For right-continuous $\phi$,
we get 
\[
	\int_{[0,t]} \phi_v\ {^{\text{\tiny CS}\!} \md} A_v=
	\int_{[0,t]} \phi_{v-}\, \mdS A_v,
\]
which also shows the main difference between the integral 
definition of \cite{CS14}, \cite{GLR12} and our definition: In the
definition of \cite{CS14} the process $A$ is integrated
against the ``previous'' values of $\phi$. Our definition differs to be suitable for the use in our  irreversible investment problem with inventory risk and to connect the target functional $\Vn$ to the relaxed one $\Vm$ (see Proposition~\ref{Main:307}).

\emph{Comparison to \cite{EL80}:} 
As \cite{EL80} lays the foundation for Meyer-$\sigma$-fields 
we also want to mention the special integral proposed in that article: For a measurable and locally bounded process
$\phi$ and a process $A$ of bounded variation, \cite{EL80} lets
\[
	\int_{[0,t]} \phi_v\, \ {^{\text{\tiny L}\!} \md} A_v
	:= \int_{[0,t]} \phi_v\, \md A_{v+} - \phi_t \Delta^+ A_t
\]
and we have
\[
	\int_{[0,t]} \phi_v\, \ {^{\text{\tiny L}\!} \md} A_v
	=\int_{[0,t]} \phi_v\, \mdS A_v
	-\sum_{v<t} (\phi_v^\ast-\phi_v)\Delta^+ A_v.
\]
The latter equation shows that the main difference to our $\mdS$-integral
results from a different treatment of the jumps $\Delta^+ A$.
\end{appendix}


 \section*{Acknowledgements}
We would like to thank two anonymous referees for their very helpful corrections and constructive comments. We also would like to thank Tom Renker for helping us with the numerical simulations in Section~\ref{sec:ill}.
 


\bibliographystyle{imsart-nameyear}  
\bibliography{Literature_ALL}       


\end{document}